\documentclass[11pt]{article}

\usepackage[margin=1in]{geometry}
\usepackage{amsmath,amssymb,amsthm,mathtools,mathrsfs,bm}
\usepackage{booktabs,array}
\usepackage{enumitem}
\usepackage{microtype}
\usepackage{lastpage}
\usepackage{graphicx}
\usepackage[table]{xcolor}
\definecolor{ourgray}{gray}{0.93}

\usepackage[authoryear,round]{natbib}
\usepackage{hyperref}
\hypersetup{colorlinks=true,linkcolor=blue!55!black,citecolor=blue!55!black,urlcolor=blue!55!black}

\usepackage{xeCJK}

\setlist{itemsep=0.25em,topsep=0.4em}
\setlist[enumerate]{leftmargin=2.2em}
\setlist[itemize]{leftmargin=1.8em}
\allowdisplaybreaks
\theoremstyle{plain}
\newtheorem{theorem}{Theorem}
\newtheorem{lemma}[theorem]{Lemma}
\newtheorem{proposition}[theorem]{Proposition}

\theoremstyle{definition}
\newtheorem{definition}[theorem]{Definition}
\newtheorem{remark}[theorem]{Remark}
\numberwithin{theorem}{section}

\newcommand{\equalcontrib}{\textsuperscript{*}}
\newcommand{\corresponding}{\textsuperscript{\dag}}

\title{An $\Omega(\kappa_y^8\epsilon^{-6})$ Lower Bound for Stochastic
NC-SC Bilevel Optimization with First-order Oracles}

\author{%
\begin{tabular}{@{}ccc@{}}
\parbox{0.28\textwidth}{\centering
Zhihao Gu\equalcontrib\\[2pt]
The Pennsylvania State University\\
{\footnotesize\texttt{zbg5155@psu.edu}}
}
&
\parbox{0.28\textwidth}{\centering
Qilong Wu\equalcontrib\\[2pt]
The Chinese University of Hong Kong, Shenzhen\\[2pt]
{\footnotesize\texttt{qilongwu@link.cuhk.edu.cn}}
}
&
\parbox{0.28\textwidth}{\centering
Junchi Yang\corresponding\\[2pt]
The Chinese University of Hong Kong, Shenzhen\\[2pt]
{\footnotesize\texttt{yangjunchi@cuhk.edu.cn}}
}
\end{tabular}%
}
\date{}

\begin{document}
\maketitle

\begingroup
\renewcommand{\thefootnote}{\fnsymbol{footnote}}
\footnotetext[1]{Equal contribution.}
\footnotetext[2]{Corresponding author. Also affiliated with Shenzhen Research Institute of Big Data and 
Shenzhen International Center for Industrial and Applied Mathematics.}
\endgroup

\begin{abstract}
We study the oracle complexity of finding \(\epsilon\)-stationary points of smooth bilevel optimization problems with a nonconvex upper-level objective and a strongly convex lower-level problem. We consider a stochastic first-order oracle that returns unbiased stochastic gradients of both the upper- and lower-level objectives, with variance bounded by \(\sigma^2\). We prove that, for any initial optimality gap \(\Delta>0\) and all sufficiently small \(\epsilon>0\), every adaptive randomized first-order algorithm requires
$
\Omega\!\left(
\Delta \kappa_y^2 \epsilon^{-2}
\max\{1,\sigma^2\kappa_y^6\epsilon^{-4}\}
\right)
$
oracle queries to find an \(\epsilon\)-stationary point of its hyper-objective function, where \(\kappa_y\) denotes the condition number of the lower-level problem. In particular, in the noise-dominated regime, the lower bound is
$
\Omega\!\left(\Delta\sigma^2\kappa_y^8\epsilon^{-6}\right).
$
This establishes the optimality of the \(\epsilon^{-6}\) dependence achieved by the best-known first-order stochastic methods.

\end{abstract}


\section{Introduction}
\label{sec:introduction}

Bilevel optimization is a fundamental framework for hierarchical
decision making in which one optimization problem is nested inside
another \citep{colson2007overview}.  In its canonical unconstrained form, the problem is written as
\begin{equation}
  \min_{\bm{x}\in\mathbb R^{d_x}}
  \Phi(\bm{x})
  \coloneqq
  f\bigl(\bm{x},\bm{y}^*(\bm{x})\bigr),
  \qquad
  \bm{y}^*(\bm{x})
  \coloneqq
  \arg\min_{\bm{y}\in\mathbb R^{d_y}}
  g(\bm{x},\bm{y}).
  \label{eq:introduction-bilevel}
\end{equation}
This nested structure captures a
wide range of modern learning problems, including hyperparameter
optimization and meta-learning
\citep{franceschi2018bilevel,rajeswaran2019meta}, as well as
reinforcement learning and actor-critic methods
\citep{hong2023two}.

We study the smooth nonconvex-strongly-convex setting, where
\(g(\bm{x},\cdot)\) is strongly convex for every \(\bm{x}\), while the
upper-level objective \(f\), and hence the hyper-objective \(\Phi\), may
be nonconvex. Under standard smoothness assumptions, \(\Phi\) is
differentiable, but its gradient involves the inverse lower-level
Hessian. Classical approaches therefore approximate this inverse
\citep{ghadimi2018approximation,ji2021bilevel}, typically requiring
Hessian-vector and thus access to
second-order information. Since such products can be substantially
more expensive than gradients, first-order bilevel methods using
only upper- and lower-level gradients have been developed
\citep{shen2023penalty,kwon2023fully,chen2025near,kwon2024penalty}.
In the stochastic NC-SC setting under bounded variance, the best-known
first-order methods exhibit an \(\epsilon^{-6}\) dependence: the
double-loop F$^3$BSA method achieves
\(\tilde{\mathcal O}(\bar\kappa_y^{11}\epsilon^{-6})\)
\citep{chen2025near}, while the single-loop SGHA method achieves
\(\mathcal O(\bar\kappa_y^{17}\epsilon^{-6})\)
\citep{gu2026sgha}, where \(\bar\kappa_y\) denotes the global condition
number. Whether this \(\epsilon^{-6}\) dependence is unavoidable remains
open.

More recently, \citet{chen2025condition} introduced a nested bilevel zero-chain construction, proving \(\Omega(\kappa_y^{2.5}\epsilon^{-2})\) and \(\Omega(\kappa_y^{4.5}\epsilon^{-4})\) lower bounds for zero-respecting algorithms on a constrained domain under the exact first-order oracle ($\mathcal{FO}$) and the stochastic first-order oracle ($\mathcal{SFO}$), respectively, where $\kappa_y$ denotes the lower-level condition number. Their construction uses a deterministic upper-level objective and a quadratic lower-level objective, with stochasticity entering through selected lower-level gradient responses. Although this yields a stronger dependence on \(\kappa_y\), the stochastic lower bound scales only as \(\epsilon^{-4}\), leaving an \(\epsilon^{-2}\) gap to the known \(\epsilon^{-6}\) upper bounds. This motivates the central question:
$$
\textit{Does stochastic NC-SC bilevel optimization require \(\Omega(\epsilon^{-6})\) standard \(\mathcal{SFO}\) calls?}
$$

In this paper, we answer this question affirmatively. Under a standard 
stochastic first-order oracle, we prove that every randomized adaptive
algorithm requires
\[
  \Omega\left(
    \frac{\Delta\kappa_y^2}{\epsilon^2}
    \max\left\{
      1,
      \frac{\sigma^2\kappa_y^6}{\epsilon^4}
    \right\}
  \right)
\]
oracle calls to find an \(\epsilon\)-stationary point.

Our construction starts from the smooth nonconvex zero-chain of
\citet{carmon2020lower} and its stochastic extension in
\citet{arjevani2023lower}.  In these constructions, each oracle query
can reveal at most one new relevant coordinate.  We use a smooth
indicator to distinguish the coordinates already revealed by previous oracle responses from the incoming coordinate, and randomize only the part of the objective associated with the incoming coordinate.  This part appears in the stochastic lower-level oracle only when a fresh
Bernoulli sample is successful.

To keep the bilevel problem unchanged, we place the
incoming-coordinate term in a two-dimensional strongly convex quadratic block
and add a compensation term.  The lower-level minimizer
remains explicit, and the resulting upper-level objective $\Phi$ coincides with the original zero-chain objective.  The incoming-coordinate term
enters the stochastic lower-level gradient with a coefficient of order
$\kappa_y^{-1}$, while the relevant off-diagonal entry of the inverse
lower Hessian has order $\kappa_y$.  These two factors cancel in the
hypergradient.  Thus the stochastic lower-level gradient has smaller
variance, while the corresponding component of $\nabla\Phi$ has the
same order as in the original zero-chain construction.

Finally, following the random-rotation arguments of
\citet{carmon2020lower} and \cite{arjevani2023lower}, the construction also
applies to adaptive randomized first-order algorithms.  The choice of
the chain parameters gives $T=\Theta\!\left(
    \Delta\kappa_y^2\epsilon^{-2}
  \right)$ coordinates, while the variance constraint permits a progress probability $p$ satisfying $p^{-1}
  =
  \Theta\!\left(
    \max\left\{
      1,
      \sigma^2\kappa_y^6\epsilon^{-4}
    \right\}
  \right)$. Since each successful Bernoulli sample increases the coordinate progress by at most one, the required number of oracle calls is of order $T/p$.  This gives $\Omega\!\left(
    \Delta\kappa_y^2\epsilon^{-2}
    \max\left\{
      1,
      \sigma^2\kappa_y^6\epsilon^{-4}
    \right\}
  \right)$, which reduces to
$\Omega(\Delta\sigma^2\kappa_y^8\epsilon^{-6})$ in the noise-dominated regime.

\begin{table}[t]
\centering
\caption{first-order oracle complexity bounds for finding an
\(\epsilon\)-stationary point in smooth NC-SC minimax and bilevel
optimization. }
\label{tab:stochastic-lower-bounds}

\begingroup
\setlength{\tabcolsep}{5pt}
\renewcommand{\arraystretch}{1.18}

\resizebox{\linewidth}{!}{
\begin{tabular}{lllll}
\toprule
Problem setting
& Reference
& Oracle model
& Method / Algorithm class
& Oracle complexity
\\
\midrule
Stochastic NC-SC Minimax
& \citet{zhang2022sapd+}
& \(\mathcal{SFO}\)
& SAPD+
& \(\mathcal{O}\!\left(
\bar\kappa_y\epsilon^{-4}
  \right)\)
\\

Stochastic NC-SC Minimax
& \citet{li2021complexity}
&  \(\mathcal{SFO}\)
& Zero-respecting algorithms
& \(\Omega\!\left(
    \sqrt{\kappa_y}\epsilon^{-2}
    +
    \kappa_y^{1/3}\sigma^2\epsilon^{-4}
  \right)\)
\\

\midrule

Deterministic NC-SC Bilevel
&\citet{chen2025condition}
& \(\mathcal{FO}\)
& \(\mathrm{F}^2\mathrm{BA}^+\)
& \(\widetilde{\mathcal{O}}\!\left(
    \bar{\kappa}_y^{3.5}\epsilon^{-2}
  \right)\)
\\

Deterministic NC-SC Bilevel
&\citet{gu2026sgha}
& \(\mathcal{FO}\)
& SGHA
& \(\mathcal{O}\!\left(
    \bar{\kappa}_y^{5}\epsilon^{-2}
  \right)\)
\\

Deterministic NC-SC Bilevel
&\citet{chen2025condition}
& $\mathcal{FO}$
& Zero-respecting algorithms
& $\Omega(\kappa_y^{2.5}\epsilon^{-2})$
\\

\midrule

Stochastic NC-SC Bilevel
& \citet{chen2025near}
&  \(\mathcal{SFO}\)
& \(\mathrm{F}^3\mathrm{BSA}\)
& \(\widetilde{\mathcal{O}}\!\left(
    \bar\kappa_y^{11}\epsilon^{-6}
  \right)\)
\\

Stochastic NC-SC Bilevel
&\citet{gu2026sgha}
&  \(\mathcal{SFO}\)
& Stoc-SGHA
& \(\mathcal{O}\!\left(
    \bar{\kappa}_y^{17}\epsilon^{-6}
  \right)\)
\\

Stochastic NC-SC Bilevel
& \citet{kwon2024complexity}
& \(\bm{y}^*\)-aware \(\mathcal{SFO}\)
& Adaptive randomized algorithms
& \(\Omega(\epsilon^{-6})\)
\\

Stochastic NC-SC Bilevel
& \citet{chen2025condition}
&  \(\mathcal{SFO}\)
& Zero-respecting algorithms
& \(\Omega(\kappa_y^{4.5}\epsilon^{-4})\)
\\

\rowcolor{ourgray}
Stochastic NC-SC Bilevel
& \textbf{This work}
&  \(\mathcal{SFO}\)
& Adaptive randomized algorithms
& \(\displaystyle
   \Omega\!\left(
     \frac{\Delta\kappa_y^2}{\epsilon^2}
     \max\left\{
       1,
       \frac{\sigma^2\kappa_y^6}{\epsilon^4}
     \right\}
   \right)\)
\\

\bottomrule
\end{tabular}
}
\endgroup

\vspace{2mm}
\begin{minipage}{0.98\linewidth}
\footnotesize
\textbf{Table notes.}
$\mathcal{FO}$ denotes an exact first-order oracle.  A $\mathcal{SFO}$ uses an independent sample at each oracle call.  The
$\bm y^*$-aware oracle of \citep{kwon2024complexity} imposes unbiasedness and bounded variance only locally around $\bm y^*(\bm x)$.  The quantities $\kappa_y$
and $\bar\kappa_y$ denote the lower-level and global condition
numbers, respectively.  Fixed problem parameters are suppressed
unless displayed. For upper bounds, ``Method / Algorithm class'' reports the proposed algorithm; for lower bounds, it reports the class of algorithms covered by the result.  
\end{minipage}
\end{table}

\section{Related Work}
\label{sec:related-work}

\paragraph{Bilevel optimization.}
Bilevel optimization has its roots in hierarchical decision making and
Stackelberg games \citep{stackelberg1934marktform}, and was introduced as a
mathematical-programming model by \citet{bracken1973mathematical}; see
\citet{colson2007overview} for an early survey. 
A classical approach to hypergradient computation is approximate
implicit differentiation (AID).  After approximately solving the
lower-level problem, AID applies the implicit-function theorem and
approximately solves a linear system involving
\(\nabla_{yy}^2 g\), typically using conjugate gradients or a truncated
Neumann series
\citep{domke2012generic,pedregosa2016hyperparameter,
ghadimi2018approximation,grazzi2020iteration}.
Iterative differentiation (ITD), by contrast, unrolls a finite number
of lower-level iterations and differentiates through the resulting
computational graph in either forward or reverse mode
\citep{maclaurin2015gradient,franceschi2017forward,
shaban2019truncated}.  Nonasymptotic analyses of AID and ITD for the
NC-SC setting were developed by \citet{ji2021bilevel}, with subsequent
extensions to stochastic and variance-reduced settings
\citep{arbel2021amortized,dagreou2022framework,
hong2023two}.  Among these methods, stocBio \citep{ji2021bilevel} achieves the best-known HVP complexity of $\widetilde{\mathcal O}(\bar\kappa_y^6\epsilon^{-4})$ under the stochastic NC-SC bilevel setting.  Although these approaches avoid explicitly forming or inverting a Hessian, they usually require Hessian-vector products (HVP), which dominate the computational cost.

Several first-order methods have been developed for bilevel optimization, using only gradients of the upper- and lower-level objectives. Early approaches include value-function and penalty-based methods such as BOME and penalty-based bilevel gradient descent \citep{liu2022bome,shen2023penalty}, followed by related penalty and sequential reformulations \citep{lu2024first,kwon2024penalty}. The F$^2$SA method \citep{kwon2023fully}, based on approximately minimizing a penalized surrogate, achieves \(\tilde{\mathcal O}(\bar\kappa_y^7\epsilon^{-3})\) deterministic complexity and \(\tilde{\mathcal O}(\operatorname{poly}(\bar\kappa_y)\epsilon^{-7})\) stochastic complexity, where \(\bar\kappa_y\) denotes the global condition number. The F$^2$BA/F$^2$BSA framework of \citet{chen2025near} improves these rates. In the deterministic NC-SC setting, F$^2$BA achieves \(\tilde{\mathcal O}(\bar\kappa_y^4\epsilon^{-2})\), and further improved to \(\tilde{\mathcal O}(\bar\kappa_y^{7/2}\epsilon^{-2})\) by F$^2$BA$^+$ \citep{chen2025condition}. In the stochastic NC-SC setting, F$^2$BSA achieves \(\tilde{\mathcal O}(\bar\kappa_y^{12}\epsilon^{-6})\), and further improved to \(\tilde{\mathcal O}(\bar\kappa_y^{11}\epsilon^{-6})\) by F$^3$BSA \citep{chen2025near}.
Thus, the best-known fully first-order methods achieve only an \(\epsilon^{-6}\) dependence under standard stochastic first-order information, in contrast to the \(\epsilon^{-4}\) dependence attainable by HVP-based or variance-reduced methods under stronger smoothness assumptions \citep{ji2021bilevel,arbel2021amortized}.

\paragraph{Lower bounds and zero-chain constructions.}
The modern lower-bound theory for first-order optimization originates
from chain-structured quadratic instances used in convex optimization
\citep{nemirovskij1983problem,nesterov2013introductory}.  The essential
feature of these instances is that first-order information can reveal
new coordinates only sequentially.  \citet{carmon2020lower,
carmon2021lowerII} formalized this mechanism through zero-chains and
developed robust zero-chain constructions for smooth nonconvex
optimization.  Their results include the classical
\(\Omega(\epsilon^{-2})\) lower bound for smooth nonconvex optimization under the first-order oracle, as well as the
\(\Omega(\epsilon^{-12/7})\) and
\(\Omega(\epsilon^{-8/5})\) lower bounds under second-order and
arbitrarily high-order smoothness, respectively.  In stochastic
optimization, \citet{arjevani2023lower} introduced probability-\(p\)
zero-chains together with random rotations and soft projections,
establishing the tight \(\Omega(\epsilon^{-4})\) lower bound under
bounded variance and the \(\Omega(\epsilon^{-3})\) lower bound under
mean-squared smoothness.  Random rotation is also the standard device
for lifting coordinate-based zero-chain arguments beyond a fixed
coordinate system \citep{carmon2020lower,arjevani2023lower}. 

Lower bounds for general bilevel optimization are more recent.
\citet{liang2023lower} established lower bounds for convex-strongly-convex and strongly-convex-strongly-convex bilevel problems under
second-order oracle access.  In finite-sum bilevel empirical-risk
minimization, \citet{dagreou2024lower} derived a matching
finite-sum lower bound, while \citet{chen2024finding} studied
hardness phenomena that arise when the lower-level problem is not
strongly convex.

For stochastic NC-SC bilevel optimization,
\citet{kwon2024complexity} established an
\(\Omega(\epsilon^{-6})\) lower bound under bounded variance and an
\(\Omega(\epsilon^{-4})\) lower bound under an additional
stochastic-smoothness assumption.  Their oracle is \(\bm{y}^*\)-aware,
supplying an \(O(\epsilon)\)-accurate approximation of \(\bm{y}^*(\bm{x})\),
while requiring the stochastic gradients to be unbiased only locally
around \(\bm{y}^*(\bm{x})\).  Consequently, these lower bounds do not directly
characterize the standard globally unbiased stochastic first-order
oracle considered in this paper. More recently, \citet{chen2025condition} developed strengthened
NC-SC bilevel lower bounds under standard first-order oracle models. In particular, they established
\(\Omega(\kappa_y^{2.5}\epsilon^{-2})\) and
\(\Omega(\kappa_y^{4.5}\epsilon^{-4})\) lower bounds for zero-respecting first-order algorithms on a constrained domain under deterministic and stochastic settings, respectively. As a special case of NC-SC bilevel optimization, NC-SC minimax optimization admits the lower bound
$
\Omega\!\left(\sqrt{\kappa_y}\epsilon^{-2}
+\kappa_y^{1/3}\sigma^2\epsilon^{-4}\right)
$
\citep{li2021complexity,zhang2021complexity}. 

Our new lower bound closes the gap between the best-known \(\epsilon^{-6}\) first-order upper bound and the previous corresponding \(\epsilon^{-4}\) lower bound for stochastic NC-SC bilevel optimization. It shows that, under a stochastic first-order oracle, smooth NC-SC bilevel optimization can be strictly harder in its dependence on \(\epsilon\) than NC-SC minimax optimization, smooth nonconvex minimization, and bilevel optimization with access to Hessian-vector products.

\paragraph{Notation.}
Vectors and matrices are denoted by boldface letters.
For a vector $\bm{x}\in\mathbb{R}^d$, its $i$--th coordinate is denoted by
$x_i$, and $\bm{e}_i$ denotes the $i$--th standard basis vector.
For a matrix $\bm{M}$, $\lVert \bm{M}\rVert_{\mathrm{op}}$ denotes its
operator norm.  For a $C^r$ map
$h:\mathbb{R}^m\to\mathbb{R}^q$, we regard its $r$--th derivative
$D^r h(\bm{x})$ as an $r$--linear map and define $\left\lVert D^r h(\bm{x})\right\rVert_{\mathrm{op}}
  \coloneqq
  \sup_{\left\lVert \bm{u}_1\right\rVert=\cdots=
        \left\lVert \bm{u}_r\right\rVert=1}
  \left\lVert
    D^r h(\bm{x})[\bm{u}_1,\ldots,\bm{u}_r]
  \right\rVert$. For a function of two variables $h(\bm{x},\bm{y})$, we use
$\nabla_xh$, $\nabla_yh$, and $\nabla^2_{xy}h$ for the corresponding
partial gradients and Hessian blocks. For $\bm{z}\in\mathbb{R}^T$ and $k\in\{0,\ldots,T\}$, let
$\bm{P}_k\bm{z}=(z_1,\ldots,z_k,0,\ldots,0)$ denote coordinate truncation onto
the first $k$ coordinates, and set $\bm P_k=\bm I_T$ for $k\ge T$.
For a threshold $c\ge0$, define the progress of $\bm{z}$ by $\operatorname{prog}_c(\bm{z})
  \coloneqq
  \max\bigl(
    \{0\}\cup
    \{j\in[T]:\lvert z_j\rvert>c\}
  \bigr)$. We write $\operatorname{St}(d,T)
  \coloneqq
  \left\{
    \bm{U}\in\mathbb{R}^{d\times T}:\bm{U}^\top \bm{U}=\bm{I}_T
  \right\}$ for the Stiefel manifold. 
   If $\bm{U}=(\bm{u}_1,\ldots,\bm{u}_T)\in\operatorname{St}(d,T)$, then its columns $\bm{u}_1,\ldots,\bm{u}_T$ are orthonormal in $\mathbb{R}^d$.

\section{Problem Setup}
\label{sec:problem}
In this section, we introduce the bilevel function class considered in this paper, the algorithm class to which our lower bound applies, and the complexity measurement of the algorithms. 

\subsection{Function Class}
\label{subsec:function-class}
To study the standard smooth nonconvex-strongly-convex (NC-SC) bilevel problem formulated in \eqref{eq:introduction-bilevel}, we introduce the following class of functions.

\begin{definition}[NC-SC bilevel function class]
\label{def:population-class}
Fix $L_f, L_g, \rho, C_f, \mu_g, \Delta>0$.  We write $\mathcal F(L_f, L_g, \rho, C_f, \mu_g, \Delta)$ for the class of pairs $f,g:
  \mathbb R^{d_x}\times\mathbb R^{d_y}
  \to
  \mathbb R$ satisfying the following conditions.
\begin{enumerate}
  \item $\nabla f$ and $\nabla g$ are globally $L_f$-- and
  $L_g$--Lipschitz, respectively.
  \item $g(\bm{x},\cdot)$ is $\mu_g$--strongly convex for every $\bm{x}$.
  \item The full Hessian $\nabla^2g$ is globally $\rho$--Lipschitz.
  \item The lower partial gradient of the upper-level objective is
  globally bounded, namely $\sup_{\bm{x},\bm{y}}\left\lVert \nabla_y f(\bm{x},\bm{y})\right\rVert \le C_f$.
  \item The hyper-objective $\Phi(\bm{x})$ is lower bounded and satisfies $\Phi(\bm{0})-\inf_{\bm{x}}\Phi(\bm{x})\le\Delta$.
\end{enumerate}
\end{definition}


For the proof, we use two notions of condition numbers. The intrinsic condition number is determined by the actual lower-level Hessian, whereas the global condition number is defined by the parameters specifying the function class.

\begin{definition}[Condition numbers]
\label{def:condition-numbers}
For a pair $(f,g)$ in Definition~\ref{def:population-class}, define
the intrinsic lower-level curvature modulus and lower-level
smoothness modulus by
\begin{equation*}
  \mu_g^{\rm act}
  \coloneqq
  \inf_{\bm{x},\bm{y}}
  \lambda_{\min}\!\left(\nabla^2_{yy}g(\bm{x},\bm{y})\right),
  \qquad
  L_y
  \coloneqq
  \sup_{\bm{x},\bm{y}}
  \left\lVert
    \nabla^2_{yy}g(\bm{x},\bm{y})
  \right\rVert_{\mathrm{op}}.
  \label{eq:actual-lower-constants}
\end{equation*}
Define the aggregate Lipschitz constant by $\bar L
  \coloneqq
  \max\left\{
    C_f,L_f,L_g,\rho
  \right\}$. The intrinsic lower-level and global condition
numbers are, respectively,
\begin{equation*}
  \kappa_y
  \coloneqq
  \frac{L_y}{\mu_g^{\rm act}},
  \qquad
  \bar\kappa_y
  \coloneqq
  \frac{\bar L}{\mu_g}.
  \label{eq:condition-number-definitions}
\end{equation*}
\end{definition}

One always has $\kappa_y\le \bar\kappa_y$. For the hard instances constructed below, the intrinsic lower-level smoothness satisfies $L_y=\Theta(1)$ and the intrinsic curvature modulus satisfies $\mu_g^{\rm act}=\Theta(\kappa^{-1})$.
Consequently, the intrinsic condition number satisfies $\kappa_y=\Theta(\kappa)$.
Our lower bound is stated in terms of this intrinsic condition number.

\subsection{Algorithm Class and Oracle Model}
\label{subsec:algorithm-class}

We consider randomized adaptive algorithms under the stochastic first-order oracle ($\mathcal{SFO}$).  At each iteration, the next query may be
any measurable function of the algorithm's internal randomness and
the previous query-response history.  Let
\[
  \mathsf W
  \coloneqq
  \mathbb R^{d_x}\times\mathbb R^{d_y}
\]
be the query space, and let
\[
  \mathsf Y
  \coloneqq
  \mathbb R^{d_x+d_y}\times\mathbb R^{d_x+d_y}
\]
be the joint oracle-response space.  For any $t\ge0$, define the
space of length-$t$ transcripts by
\[
  \mathsf T_t
  \coloneqq
  (\mathsf W\times\mathsf Y)^t,
  \qquad
  \mathsf T_0
  \coloneqq
  \{\varnothing\}.
\]

\begin{definition}[Randomized adaptive algorithm]
\label{def:algorithm-class}
Fix an oracle-call budget $N\in\mathbb N_0$.  A randomized adaptive algorithm $\mathsf A$ consists of a random seed
$\omega_{\rm alg}$ taking values in a standard Borel probability
space, a sequence of Borel query maps
\begin{equation*}
  \operatorname{Query}_t^{\mathsf A}
  :
  \Omega_{\rm alg}\times\mathsf T_{t-1}
  \longrightarrow
  \mathsf W,
  \qquad
  1\le t\le N,
  \label{eq:algorithm-query-map}
\end{equation*}
and a Borel output map
\begin{equation*}
  \operatorname{Output}_N^{\mathsf A}
  :
  \Omega_{\rm alg}\times\mathsf T_N
  \longrightarrow
  \mathbb R^{d_x}.
  \label{eq:algorithm-output-map}
\end{equation*}

Given an oracle, the interaction is defined recursively.  Starting
from the empty transcript $\mathcal T_0=\varnothing$, at call $t$ the
algorithm selects
\begin{equation*}
  \bm{w}_t=(\bm{x}_t,\bm{y}_t)
  =
  \operatorname{Query}_t^{\mathsf A}
  (\omega_{\rm alg},\mathcal T_{t-1}),
  \label{eq:algorithm-query-recursion}
\end{equation*}
receives an oracle response $\bm{Y}_t\in\mathsf Y$, and forms
\begin{equation*}
  \mathcal T_t
  =
  \bigl(
    \mathcal T_{t-1},\bm{w}_t,\bm{Y}_t
  \bigr)
  \in
  \mathsf T_t.
  \label{eq:original-transcript-recursion}
\end{equation*}
After $N$ calls, the algorithm returns
\begin{equation*}
  \widehat{\bm{x}}_N^{\mathsf A}
  =
  \operatorname{Output}_N^{\mathsf A}
  (\omega_{\rm alg},\mathcal T_N).
  \label{eq:algorithm-output-recursion}
\end{equation*}
The pre-query information before call $t$ is the sigma-field
\begin{equation*}
  \mathscr F_{t-1}
  \coloneqq
  \sigma\bigl(
    \omega_{\rm alg},\mathcal T_{t-1}
  \bigr).
  \label{eq:prequery-filtration}
\end{equation*}
Thus $\bm{w}_t$ is $\mathscr F_{t-1}$-measurable.  We denote by $\mathcal A_N^{\mathcal{SFO}}(d_x,d_y)$ the class of all such algorithms making at most $N$ stochastic first-order oracle calls.
When the dimensions are clear, we simply write
$\mathcal A_N^{\mathcal{SFO}}$.
\end{definition}

An algorithm that stops before call $N$ is included in
$\mathcal A_N^{\mathcal{SFO}}$ by allowing its output map to ignore
all subsequent oracle responses.  No zero-respecting, linear-span, or
other structural restriction is imposed on the query maps.

We next introduce the stochastic first-order oracle.
\begin{definition}[Stochastic first-order oracle ($\mathcal{SFO}$)]
\label{def:sfo}
At a query $(\bm{x}_t,\bm{y}_t)$, a joint stochastic first-order oracle
returns
\begin{equation*}
  \bm{Y}_t
  =
  \mathbb O^{\mathcal{SFO}}
  (\bm{x}_t,\bm{y}_t;\zeta_t,\xi_t)
  =
  \bigl(
    \bm{G}_f(\bm{x}_t,\bm{y}_t;\zeta_t),
    \bm{G}_g(\bm{x}_t,\bm{y}_t;\xi_t)
  \bigr).
  \label{eq:sfo-response}
\end{equation*}
The fresh sample pair $(\zeta_t,\xi_t)$ is independent of the
pre-query sigma-field $\mathscr F_{t-1}$.  At every possibly random
query generated by an algorithm in
$\mathcal A_N^{\mathcal{SFO}}$, the oracle satisfies
\[
  \begin{aligned}
    \mathbb E\!\left[
      \bm{G}_f(\bm{x}_t,\bm{y}_t;\zeta_t)
      \,\middle|\,
      \mathscr F_{t-1}
    \right]
    &=
    \nabla f(\bm{x}_t,\bm{y}_t),
    &
    \mathbb E\!\left[
      \left\lVert
        \bm{G}_f(\bm{x}_t,\bm{y}_t;\zeta_t)-\nabla f(\bm{x}_t,\bm{y}_t)
      \right\rVert^2
      \,\middle|\,
      \mathscr F_{t-1}
    \right]
    &\le
    \sigma^2,
    \\
    \mathbb E\!\left[
      \bm{G}_g(\bm{x}_t,\bm{y}_t;\xi_t)
      \,\middle|\,
      \mathscr F_{t-1}
    \right]
    &=
    \nabla g(\bm{x}_t,\bm{y}_t),
    &
    \mathbb E\!\left[
      \left\lVert
        \bm{G}_g(\bm{x}_t,\bm{y}_t;\xi_t)-\nabla g(\bm{x}_t,\bm{y}_t)
      \right\rVert^2
      \,\middle|\,
      \mathscr F_{t-1}
    \right]
    &\le
    \sigma^2.
  \end{aligned}
\]
\end{definition}
We write $\mathbb O_\sigma^{\mathcal{SFO}}(f,g)$ for the collection of all
such oracles for the function pair $(f,g)$.

\subsection{Worst-Case Oracle Complexity}
\label{subsec:complexity-measure}

Since the hyper-objective $\Phi$ in
\eqref{eq:introduction-bilevel} is nonconvex, global
optimization may require exponentially many oracle calls in the worst
case \citep{nemirovskij1983problem}.  We therefore use the standard
first-order stationarity criterion from smooth nonconvex optimization
\citep{carmon2020lower,carmon2021lowerII,arjevani2023lower}.

\begin{definition}[$\epsilon$-stationarity]
\label{def:epsilon-stationary}
A possibly random point $\widehat{\bm{x}}\in\mathbb{R}^{d_x}$ is
called an $\epsilon$-stationary output of $\Phi$ if
\begin{equation}
  \mathbb{E}
  \left\lVert
    \nabla\Phi(\widehat{\bm{x}})
  \right\rVert
  \le \epsilon,
  \label{eq:epsilon-stationary}
\end{equation}
where the expectation is taken over all randomness of the algorithm
and the stochastic oracle.
\end{definition}

For a function class
$\mathcal F$ and an associated family of admissible stochastic
oracles $\mathbb O_\sigma^{\mathcal{SFO}}$, define the worst-case risk after $N$ oracle
calls by
\begin{align}
  \mathfrak R_N^{\mathcal{SFO}}
  \bigl(
    \mathcal A^{\mathcal{SFO}},
    \mathcal F,
    \mathbb O_\sigma^{\mathcal{SFO}}
  \bigr)
  \coloneqq
  \inf_{\mathsf A\in\mathcal A_N^{\mathcal{SFO}}}
  \sup_{(f,g)\in\mathcal F}
  \sup_{\mathbb O\in
    \mathbb O_\sigma^{\mathcal{SFO}}(f,g)}
  \mathbb{E}_{\mathsf A,\mathbb O}
  \left\lVert
    \nabla\Phi
    \bigl(
      \widehat{\bm{x}}_N^{\mathsf A}
    \bigr)
  \right\rVert .
  \label{eq:worst-case-risk}
\end{align}
Here $\widehat{\bm{x}}_N^{\mathsf A}$ denotes the output of
$\mathsf A$ after at most $N$ oracle calls.

\begin{definition}[Worst-case oracle complexity]
\label{def:worst-case-complexity}
The worst-case $\mathcal{SFO}$ complexity for finding an
$\epsilon$-stationary point is
\begin{equation}
  \operatorname{Compl}_{\epsilon}^{\mathcal{SFO}}
  (\mathcal A^{\mathcal{SFO}}, \mathcal F,\mathbb O_\sigma^{\mathcal{SFO}})
  \coloneqq
  \inf
  \left\{
    N\in\mathbb{N}_0:
    \mathfrak R_N^{\mathcal{SFO}}(\mathcal A^{\mathcal{SFO}}, \mathcal F,\mathbb O_\sigma^{\mathcal{SFO}})
    \le \epsilon
  \right\}.
  \label{eq:worst-case-complexity}
\end{equation}
\end{definition}

Definition \ref{def:worst-case-complexity} follows the uniform minimax convention: a single
algorithm is chosen before the worst-case objective function and
stochastic oracle are selected.  This quantifier order is consistent
with our lower-bound statement, which states that for every
randomized adaptive algorithm there exist an instance and
an admissible oracle on which the algorithm requires the
claimed number of oracle calls.

\section{Lower Bound for Stochastic NC-SC Bilevel Optimization}
\label{sec:lower-bound}

In this section, we first construct and analyze the hard instance for a fixed
orthonormal frame, including the parameter choice.  We then randomize
the frame to obtain the progress bound for randomized adaptive
algorithms.  The proofs are given in the appendices.

\subsection{Main Result}
\label{subsec:main-result-overview}

\begin{theorem}
\label{thm:main}
There are universal positive constants
$c_{\rm LB}, c_\epsilon,
  \bar L_f, \bar L_g,
  \bar\rho$ with the following property.  Fix any
$\Delta>0$.  There is a finite constant $\bar C_f(\Delta)>0$,
depending only on $\Delta$, such that, for every $\kappa\ge2$,
$\sigma\ge0$, and $\epsilon$ satisfying
\begin{equation}
  0<\epsilon\le c_\epsilon\min\{1,\sqrt\Delta\},
  \label{eq:epsilon-range-main}
\end{equation}
and every $N\in\mathbb N_0$ satisfying
\begin{equation}
  N\le
  c_{\rm LB}\frac{\Delta\kappa^2}{\epsilon^2}
  \max\left\{1,
    \frac{\sigma^2\kappa^6}{\epsilon^4}
  \right\},
  \label{eq:main-call-budget}
\end{equation}
there exists a finite dimension $d$ such that, for every $\mathsf A
  \in
  \mathcal A_N^{\mathcal{SFO}}(d,d+2)$, there exist a pair $(f,g)\in
  \mathcal F(\bar L_f,\bar L_g,
  \bar\rho,\bar C_f(\Delta),(2\kappa)^{-1},\Delta)$ and an oracle $\mathbb O
  \in
  \mathbb O_\sigma^{\mathcal{SFO}}(f,g)$ such that the output $\widehat{\bm{x}}_N^{\mathsf A}$ obeys
\begin{equation}
  \mathbb E_{\mathsf A,\mathbb O}
  \left\lVert
    \nabla\Phi(\widehat{\bm{x}}_N^{\mathsf A})
  \right\rVert
  >\frac32\epsilon,
  \qquad
  \mathbb E_{\mathsf A,\mathbb O}
  \left\lVert
    \nabla\Phi(\widehat{\bm{x}}_N^{\mathsf A})
  \right\rVert^2
  >\frac94\epsilon^2.
  \label{eq:main-failure}
\end{equation}
In particular, 
\[
  \mathfrak R_N^{\mathcal{SFO}}
  \bigl(
    \mathcal A^{\mathcal{SFO}},
    \mathcal F(\bar L_f,\bar L_g,
      \bar\rho,\bar C_f(\Delta),(2\kappa)^{-1},\Delta),
    \mathbb O_\sigma^{\mathcal{SFO}}
  \bigr)
  \ge\frac32\epsilon>\epsilon.
\]
Finally, as the intrinsic condition number of the constructed lower problem satisfies
$\kappa\le\kappa_y\le14\kappa$, the complexity bound can equivalently be expressed as
\begin{equation}
  \Omega\!\left(
    \frac{\Delta\kappa_y^2}{\epsilon^2}
    \max\left\{1,
      \frac{\sigma^2\kappa_y^6}{\epsilon^4}
    \right\}
  \right).
  \label{eq:main-kappa-y-form}
\end{equation}
\end{theorem}

By Theorem \ref{thm:main}, the noise-dominated term is $
\Omega\!\left(\Delta \sigma^2 \kappa_y^8 \epsilon^{-6}\right)$.
To our knowledge, this is the first $\epsilon^{-6}$ complexity lower
bound for smooth stochastic NC-SC bilevel optimization under the
standard stochastic first-order oracle, showing that the
$\epsilon$-dependence of the existing upper bounds is optimal
\citep{chen2025near, gu2026sgha}. When $\sigma=0$, the same construction
yields a deterministic lower bound
\(
\Omega\!\left(\Delta \kappa_y^2 \epsilon^{-2}\right)
\)
for randomized algorithms. This is slightly weaker in its dependence
on $\kappa_y$ than the deterministic lower bound of
\citet{chen2025condition}, which establishes the bound $\Omega\!\left(\Delta \kappa_y^{2.5} \epsilon^{-2}\right)$.

We emphasize that the theorem is high-dimensional. Once
the parameters and the oracle-call budget \(N\) are fixed, the ambient
dimension \(d\) is chosen sufficiently large, uniformly over all
algorithms in
\(\mathcal A_N^{\mathcal{SFO}}(d,d+2)\). A sufficient choice of \(d\)
is given in \eqref{eq:dimension-choice}. Thus, the theorem does not
assert an $\epsilon^{-6}$ lower bound in any fixed finite dimension.

The proof has an analytic part and a probabilistic part.  The analytic part shows that the hypergradient is large whenever the coordinate progress is below $T$. The probabilistic part shows that, below the stated call budget, the coordinate progress remains below $T$ with constant probability. The hard instance is first defined relative to a fixed orthonormal frame. $\eta$ sets the scale of the rescaled zero-chain: each chain coordinate contributes $O(\eta^2)$ to the initial gap, while the corresponding gradient scale is $O(\eta)$.  $T$ is the number of chain coordinates, and $p$ is the success probability of the fresh Bernoulli variable that permits dependence on the next relevant
coordinate.  The random rotation hides the fixed frame and is used only in the probabilistic progress argument.

We choose $T=\Theta(\Delta/\eta^2)$, which uses the initial-gap budget
up to numerical constants.  The stochastic lower-level oracle
randomizes a compensated perturbation inside a two-dimensional
strongly convex quadratic block, with gradient norm $O(\eta^2/\kappa)$.  After Bernoulli importance
weighting, its contribution to the oracle variance is bounded by
$O(\eta^4/(p\kappa^2))$.  We therefore choose $p$ as $p^{-1}=\Theta\!\left(\max\{1,\sigma^2\kappa^2\eta^{-4}\}\right)$, and further choose $\eta=\Theta(\epsilon/\kappa)$ so that the hypergradient at
any point with coordinate progress below $T$ exceeds $\epsilon$.
Combining these three relations with the $T/p$ query bound for the
probability-$p$ zero-chain gives the
$\kappa^8\epsilon^{-6}$ term in the noise-dominated regime.  The
random rotation affects the dimension required for the progress
argument but does not change this parameter scaling.

\subsection{Hard Instance Construction}
\label{sec:construction}
In the hard bilevel instance constructed below, the lower variable is denoted by
$\bm{y}=(\bm{z},\bm{v})\in\mathbb{R}^d\times\mathbb{R}^2$.  The
$\bm z$-block yields $\bm z^*(\bm x)=\kappa\bm x$.  The
two-dimensional $\bm v$-block scales the gradient contribution of
the compensated term by
$\Theta(\kappa^{-1})$, while preserving
the corresponding term in the hyper-objective after the lower problem
is minimized.

Throughout this section, every symbol of the form $C_{\bullet}$
denotes a fixed numerical constant independent of
$T,d,\bm U,\eta,\kappa,\Delta,\sigma$, and $\epsilon$.  The class bound
$\bar C_f(\Delta)$ is allowed to depend on $\Delta$, as stated in
Lemma~\ref{lem:regularity}.

\subsubsection{Scaled Nonconvex Zero-chain}

Following the zero-chain construction of \cite{carmon2020lower}, define
\begin{equation}
  \Psi(t)=
  \begin{cases}
    0,&t\le\frac12,\\[1mm]
    \exp\!\left(1-\dfrac{1}{(2t-1)^2}\right),&t>\frac12,
  \end{cases}
  \qquad
  \Phi_0(t)=\sqrt e\int_{-\infty}^t e^{-s^2/2}\,ds,
  \label{eq:Psi-Phi}
\end{equation}
and
\begin{equation}
  Q(t_1,t_2)=\Psi(-t_1)\Phi_0(-t_2)-\Psi(t_1)\Phi_0(t_2).
  \label{eq:Q-definition}
\end{equation}
For $\bm{a}=(a_1,\ldots,a_T)$, set $a_0=1$ and let
\begin{equation}
  F_T(\bm{a})=\sum_{i=1}^T Q(a_{i-1},a_i).
  \label{eq:unscaled-chain}
\end{equation}
Given $0<\eta\le1$, set $z_0=\eta$ and define
\begin{equation}
  H_{\eta,T}(\bm{z})
  =\eta^2F_T(\bm{z}/\eta)
  =\eta^2\sum_{i=1}^T\mathcal{L}_i^\eta(\bm{z}),
  \qquad
  \mathcal{L}_i^\eta(\bm{z})=Q(z_{i-1}/\eta,z_i/\eta).
  \label{eq:scaled-chain}
\end{equation}

The function $H_{\eta,T}$ has a sequential dependence structure.  Its
$i$th summand depends on coordinates $i-1$ and $i$, and for $i\ge2$ this
summand and all of its derivatives vanish when
$\lvert z_{i-1}\rvert\le\eta/2$.  The quantity
$\operatorname{prog}_\eta(\bm z)$ is the largest index whose
coordinate has magnitude greater than $\eta$.  If
$j=\operatorname{prog}_\eta(\bm z)+1\le T$, then
Lemma~\ref{lem:standard-chain} gives
$\lvert\partial_jH_{\eta,T}(\bm z)\rvert\ge\eta$.
We use standard properties of the smooth nonconvex zero-chain from
\citet{carmon2020lower} and
\citet{arjevani2023lower}; see, e.g., the derivative and zero-chain
bounds around Lemmas 2--3 and 6--7 of \citet{arjevani2023lower}.
We state the scaled form needed here and omit the proof.

\begin{lemma}[
\citep{arjevani2023lower}]
\label{lem:standard-chain}
There are fixed numerical constants
$C_\Delta,C_{{\rm ch},2},C_{{\rm ch},3}>0$ such that, for all
$T\ge1$, $0<\eta\le1$, and $\bm{z}\in\mathbb{R}^T$,
\begin{equation}
    H_{\eta,T}(\bm{0})-\inf_{\bm{q}\in\mathbb{R}^T}H_{\eta,T}(\bm{q})
    \le C_\Delta\eta^2T,
\end{equation}
and
\begin{alignat}{2}
  &\left\lVert
    \nabla H_{\eta,T}(\bm{z})
  \right\rVert_\infty
  \le 23\eta,
  &&\quad
  \left\lVert
    \nabla H_{\eta,T}(\bm{z})
  \right\rVert
  \le 23\eta\sqrt T,
  \label{eq:chain-gradient-bound}
  \\
  &\left\lVert
    D^2H_{\eta,T}(\bm{z})
  \right\rVert_{\mathrm{op}}
  \le C_{{\rm ch},2},
  &&\quad
  \left\lVert
    D^3H_{\eta,T}(\bm{z})
  \right\rVert_{\mathrm{op}}
  \le \frac{C_{{\rm ch},3}}{\eta}.
  \label{eq:chain-higher-bounds}
\end{alignat}
If $j=\operatorname{prog}_\eta(\bm{z})+1\le T$, then
\begin{equation}
  \partial_jH_{\eta,T}(\bm{z})\le-\eta.
  \label{eq:chain-frontier-gradient}
\end{equation}
Finally, if $i\ge2$ and $\left\lvert z_{i-1}\right\rvert \le\eta/2$, then
$\mathcal{L}_i^\eta$ and all its derivatives vanish at $\bm{z}$.
\end{lemma}

By \eqref{eq:scaled-chain}, for every $r\ge0$,
\begin{equation}
  D^rH_{\eta,T}(\bm{z})
  =\eta^{2-r}D^rF_T(\bm{z}/\eta).
  \label{eq:chain-scaling-identity}
\end{equation}
Thus each summand contributes $O(\eta^2)$ to the initial gap; whenever
$\operatorname{prog}_\eta(\bm z)<T$, the derivative in coordinate
$\operatorname{prog}_\eta(\bm z)+1$ has magnitude $\Theta(\eta)$; and
the Hessian remains $O(1)$.

\subsubsection{Smooth Decomposition of the Zero-chain}
\label{subsec:extractor}
For the stochastic construction, we use soft indicators of the tail
coordinates $z_i,\ldots,z_T$ to isolate the incoming-coordinate summand for
randomization. This decomposition adapts the smooth probability-$p$
zero-chain construction of
\citet[Sections~3.3 and~5.1]{arjevani2023lower} by rescaling it with parameter \(\eta\).
Let
\[
  \chi_+(s)=
  \begin{cases}
    0,&s\le0,\\
    e^{-1/s},&s>0,
  \end{cases}
\]
and, for $a<b$, define the flat transition
\begin{equation}
  \gamma_{a,b}(t)
  =\frac{\chi_+(t-a)}
  {\chi_+(t-a)+\chi_+(b-t)}.
  \label{eq:flat-transition}
\end{equation}
$\gamma_{a,b}(t)$ is $C^\infty$, equals zero on $(-\infty,a]$, equals one on
$[b,\infty)$, and satisfies
$\gamma_{a,b}^{(r)}(a)=\gamma_{a,b}^{(r)}(b)=0$
for every integer $r\ge1$. Define
\begin{alignat}{2}
  &\Gamma(t)
    =\gamma_{1/4,1/2}(t),
  &&\quad
    \upsilon(t)
    =\Gamma\!\left(\left\lvert t\right\rvert\right),
  \\
  &\bm{V}_i^\eta(\bm{z})
    =
    \bigl(
      \upsilon(z_i/\eta),
      \ldots,
      \upsilon(z_T/\eta)
    \bigr),
  &&\quad
    h_i^\eta(\bm{z})
    =
    \Gamma\!\left(
      1-\left\lVert \bm{V}_i^\eta(\bm{z})\right\rVert
    \right).
  \label{eq:extractor-gates}
\end{alignat}
Indeed, if $\bm{x}=\bm{z}/\eta$, then $h_i^\eta(\bm{z})$ has the same
tail-norm form as the smooth progress indicator $\Theta_i(\bm{x})$ in
\citet[Eq.~(19)]{arjevani2023lower}, up to the particular choice of
the one-dimensional transition function.  The thresholds $\eta/4$
and $\eta/2$ define a fixed transition region.  A tail coordinate
with magnitude greater than $\eta/2$ makes every corresponding soft
indicator vanish locally, while the flatness of the chain summands
eliminates terms lying more than one coordinate beyond the current
progress index.

Finally, define
\begin{equation}
  b_{\eta,T}(\bm{z})
  =\eta^2\sum_{i=1}^T h_i^\eta(\bm{z})\mathcal{L}_i^\eta(\bm{z}),
  \qquad
  A_{\eta,T}(\bm{z})=H_{\eta,T}(\bm{z})-b_{\eta,T}(\bm{z}).
  \label{eq:frontier-decomposition}
\end{equation}
For $\xi\sim\operatorname{Ber}(p)$, randomizing only $b_{\eta,T}$
in the decomposition $H_{\eta,T}=A_{\eta,T}+b_{\eta,T}$ gives
\[
  A_{\eta,T}(\bm z)+\frac{\xi}{p}b_{\eta,T}(\bm z)
  =
  H_{\eta,T}(\bm z)
  +\left(\frac{\xi}{p}-1\right)b_{\eta,T}(\bm z).
\]
This has the same importance-weighted form as the randomized zero-chain in \citet[Eq.~(31)]{arjevani2023lower}.
In our bilevel construction, however, we do not randomize the chain function directly. Instead, $A_{\eta,T}$ enters the deterministic
upper-level term, while $b_{\eta,T}$ is used to construct the compensated lower-level block, which is multiplied in its entirety by $\xi/p$.

Although $b_{\eta,T}$ is written as a sum of $T$ terms, at any point
at most one summand can have a nonzero derivative of order at most
three.  The next lemma gives this derivative bound together with
the required coordinate-dependence identities.  Its proof is given in
Appendix~\ref{app:extractor-proof}.

\begin{lemma}
\label{lem:extractor}
For $r=0,1,2,3$,
\begin{equation}
  \left\lVert D^rb_{\eta,T}(\bm{z})\right\rVert _{\mathrm{op}}
  \le C_{\rm ext}\eta^{2-r},
  \label{eq:b-derivative-bounds}
\end{equation}
where $C_{\rm ext}$ is a fixed numerical constant. Let
\[
  m=\operatorname{prog}_{\eta/2}(\bm{z}),
  \qquad
  k=\operatorname{prog}_{\eta/4}(\bm{z}).
\]
Then $m\le k$.  If $m<T$, only the $(m+1)$--st summand in
\eqref{eq:frontier-decomposition} can have a nonzero derivative of
order at most three at $\bm{z}$; if $m=T$, the value and derivatives
through order three of every summand vanish at $\bm z$.
Moreover,
\begin{align}
  &A_{\eta,T}(\bm{z})=A_{\eta,T}(\bm{P}_k\bm{z}),
  &
  &b_{\eta,T}(\bm{z})=b_{\eta,T}(\bm{P}_{k+1}\bm{z}),
  \label{eq:extractor-cylinder-values}
 \\
  &D^rA_{\eta,T}(\bm{z})
  =D^rA_{\eta,T}(\bm{P}_k\bm{z})\circ(\bm{P}_k,\ldots,\bm{P}_k),
  & &1\le r\le3,
  \label{eq:extractor-cylinder-A}\\
  &D^rb_{\eta,T}(\bm{z})
  =D^rb_{\eta,T}(\bm{P}_{k+1}\bm{z})
    \circ(\bm{P}_{k+1},\ldots,\bm{P}_{k+1}),
  & &1\le r\le3.\label{eq:extractor-cylinder-b}
\end{align}
The identities remain valid when a coordinate is exactly at a
threshold $\eta/4$ or $\eta/2$.
\end{lemma}

Equations~\eqref{eq:extractor-cylinder-values}--\eqref{eq:extractor-cylinder-b}
show that $A_{\eta,T}$ and its first three derivatives depend only on
the first $k$ coordinates, while $b_{\eta,T}$ may also depend on
coordinate $k+1$.  These coordinate-dependence identities are used in the
random-rotation argument.

\subsubsection{Bilevel Objectives and Stochastic Oracle}

Define an even, compactly supported cutoff function
\begin{equation}
  \omega(t)=1-\gamma_{1,2}(\lvert t\rvert),
  \qquad
  \omega(t)=
  \begin{cases}
    1,&|t|\le1,\\
    0,&|t|\ge2.
  \end{cases}
  \label{eq:omega-definition}
\end{equation}
We will later use the scaled cutoff function $\psi_\eta(v)=v\omega(v/\eta)$. Then $\psi_\eta(v)=v$ on
$[-\eta,\eta]$ and $\psi_\eta(v)=0$ outside
$[-2\eta,2\eta]$. Since
$\varphi(t):=t\omega(t)$ is a fixed smooth compactly
supported function, the numerical constant
$C_\psi\coloneqq\max_{0\le r\le3}\lVert\varphi^{(r)}\rVert_\infty$
is finite. The identity $\psi_\eta(v)=\eta\varphi(v/\eta)$ gives
\begin{equation}
  \left\lVert \psi_\eta^{(r)}\right\rVert _\infty
  \le C_\psi\eta^{1-r},\quad 0\le r\le3.
  \label{eq:psi-derivative-scale}
\end{equation}

To obtain a stronger dependence of the stochastic lower bound on the
condition number $\kappa_y$, we introduce a $2\times2$ matrix that will
serve as a quadratic block in the hard instance. Fix $\kappa\ge2$ and
define
\begin{equation}
  \bm{M}_\kappa
  \coloneqq
  \frac12
  \begin{pmatrix}
    1+\kappa^{-1} & -(1-\kappa^{-1})\\
    -(1-\kappa^{-1}) & 1+\kappa^{-1}
  \end{pmatrix}.
  \label{eq:M-kappa}
\end{equation}
Let $d_\kappa \coloneqq \frac{\kappa+1}{2}$ and $
  s_\kappa \coloneqq \frac{\kappa-1}{2}$.
Then
\begin{equation}
  \bm{M}_\kappa^{-1}
  =
  \begin{pmatrix}
    d_\kappa & s_\kappa\\
    s_\kappa & d_\kappa
  \end{pmatrix},
  \qquad
  \lambda_{\min}(\bm{M}_\kappa)=\kappa^{-1},
  \qquad
  \lambda_{\max}(\bm{M}_\kappa)=1.
  \label{eq:amplifier-spectrum}
\end{equation}
For $\kappa\ge2$, these quantities satisfy
\begin{equation}
  s_\kappa\ge\frac\kappa4,
  \qquad
  d_\kappa\le\frac{3\kappa}{4},
  \qquad
  1<\frac{d_\kappa}{s_\kappa}\le3.
  \label{eq:amplifier-elementary-bounds}
\end{equation}

Fix an orthonormal frame
$\bm{U}=(\bm{u}_1,\ldots,\bm{u}_T)$ in $\mathbb{R}^d$.  
A Haar-random choice
of $\bm U$ is introduced only in Section~\ref{sec:rotation-lifting}.  Let $R=230\eta \sqrt T$, and define the soft projection, following~\cite{carmon2020lower} and~\cite{arjevani2023lower}, by
\begin{equation}
  \rho_R(\bm{z})=\frac{\bm{z}}{\sqrt{1+\left\lVert \bm{z}\right\rVert ^2/R^2}}.
  \label{eq:soft-projection}
\end{equation}
Define the projected chain coordinates
 $ \bm q_{\bm U}(\bm z)
  \coloneqq
  \bm U^\top \rho_R(\bm z)$.
Applying the deterministic and incoming-coordinate components in
\eqref{eq:frontier-decomposition} to these projected coordinates, define
\begin{equation}
  a_{\bm U}(\bm z)
  \coloneqq
  A_{\eta,T}(\bm q_{\bm U}(\bm z)),
  \qquad
  \theta_{\bm U}(\bm z)
  \coloneqq
  b_{\eta,T}(\bm q_{\bm U}(\bm z)).
  \label{eq:rotated-pieces}
\end{equation}

We then define three components of the hard instance: a pseudo-Huber term, an incoming-coordinate term scaling
$\theta_{\bm U}(\bm z)$ by $1/s_\kappa$, and a compensation term:
\begin{align}
  &h_R(\bm{z}) 
  = \frac{1}{4} R^2\left(
    \sqrt{1+\frac{\left\lVert \bm{z}\right\rVert ^2}{R^2}}-1
  \right),
  \label{eq:pseudo-huber}\\
  &\vartheta_{\kappa,\bm{U}}(\bm{z})
  =\frac{\theta_{\bm{U}}(\bm{z})}{s_\kappa},
  \label{eq:attenuated-frontier}\\
  &\mathcal{R}_{\kappa,\bm{U}}(\bm{z},\bm{v})
  =\frac{d_\kappa}{2}\vartheta_{\kappa,\bm{U}}(\bm{z})^2
    -\vartheta_{\kappa,\bm{U}}(\bm{z})\psi_\eta(v_1),
  \qquad
  \bm{v}=(v_1,v_2)\in\mathbb R^2.
  \label{eq:compensated-block}
\end{align}

Finally, we define the upper- and lower-level objectives of the hard instance as follows:
\begin{align}
  f_{\bm{U}}^\kappa(\bm{x},\bm{z},\bm{v})
  &=a_{\bm{U}}(\bm{z})+v_2+h_R(\bm{z}),
  \label{eq:hard-f}\\
  g_{\bm{U}}^\kappa(\bm{x},\bm{z},\bm{v})
  &=\frac1{2\kappa}\left\lVert \bm{z}\right\rVert ^2-\left\langle \bm{x},\bm{z}\right\rangle 
    +\frac12\bm{v}^\top\bm{M}_\kappa\bm{v}
    +\mathcal{R}_{\kappa,\bm{U}}(\bm{z},\bm{v}).
  \label{eq:hard-g}
\end{align}
Note that the upper-level variable $\bm{x}$ appears only in the lower-level objective. We introduce stochasticity by randomizing only the compensated block. Specifically, for
$\xi\sim\operatorname{Ber}(p)$ with $p\in(0,1]$, set
\begin{equation}
  \widehat f_{\bm{U}}^\kappa=f_{\bm{U}}^\kappa,
  \qquad
  \widehat g_{\bm{U}}^\kappa(\bm{x},\bm{z},\bm{v};\xi)
  =\frac1{2\kappa}\left\lVert \bm{z}\right\rVert ^2-\left\langle \bm{x},\bm{z}\right\rangle 
    +\frac12\bm{v}^\top\bm{M}_\kappa\bm{v}
    +\frac\xi p\mathcal{R}_{\kappa,\bm{U}}(\bm{z},\bm{v}).
  \label{eq:sample-hard-instance}
\end{equation}
The $\mathcal{SFO}$ returns
$(\nabla\widehat f_{\bm{U}}^\kappa,\nabla\widehat g_{\bm{U}}^\kappa)$.

To explain the two auxiliary coordinates, consider the region in
which $\psi_\eta(v_1)=v_1$.
Lemma~\ref{lem:population-solution} shows that the population lower-level minimizer lies in this region. Hence, the $\bm{v}$-subproblem is
\[
\min_{\bm{v}\in\mathbb{R}^2}
\left\{
\frac{1}{2}\bm{v}^\top \bm{M}_\kappa \bm{v}
+
\frac{d_\kappa}{2}\vartheta_{\kappa,\bm{U}}(\bm{z})^2
-
\vartheta_{\kappa,\bm{U}}(\bm{z}) v_1
\right\},
\]
whose first-order condition is
\[
  \bm M_\kappa\bm v=
  \begin{pmatrix}\vartheta_{\kappa,\bm{U}}(\bm{z})\\0\end{pmatrix}.
\]
Since $(\bm M_\kappa^{-1})_{21}=s_\kappa$, the second coordinate of
the solution is $v_2^*=s_\kappa\vartheta_{\kappa,\bm{U}}(\bm{z})=\theta_{\bm U}(\bm{z})$.  Since the
upper objective contains $v_2$, substitution of the lower solution
recovers $\theta_{\bm U}(\bm{z})$ at its original scale.  Moreover, since
$(\bm M_\kappa^{-1})_{11}=d_\kappa$,
\[
  \min_{\bm v}
  \left\{
    \frac12\bm v^\top\bm M_\kappa\bm v
    -\vartheta_{\kappa,\bm{U}}(\bm{z}) v_1
    +\frac{d_\kappa}{2}\vartheta_{\kappa,\bm{U}}(\bm{z})^2
  \right\}
  =0.
\]
Thus minimizing over $\bm v$ adds no $\bm z$-dependent value to the
 lower-level objective, and
$\bm z^*(\bm x)=\kappa\bm x$ is unchanged. Both terms in $\mathcal R_{\kappa,\bm U}$ are multiplied by the same Bernoulli weight; otherwise the response for $\xi=0$ could still
depend on the incoming coordinate through the deterministic
compensation term.

\subsection{Properties of the Hard Instance}
\label{sec:analytic}

The next lemma gives the derivative bounds obtained after composition with the soft projection $\rho_R$. Its proof is given in
Appendix~\ref{app:composition-proof}.

\begin{lemma}
\label{lem:composition-scales}
Define
\begin{align}
  &C_A
  =\max\left\{
      23+C_{\rm ext},
      C_{{\rm ch},2}+C_{\rm ext},
      C_{{\rm ch},3}+C_{\rm ext}
    \right\},
  \\
  &C_\theta
  =C_{\rm ext}\left(1+\frac{18}{230}
      +\frac{36}{230^2}\right),
  \qquad
  C_a
  =C_A\left(1+\frac{18}{230}
      +\frac{36}{230^2}\right).
\end{align}
Uniformly in $T,d,\bm{U}$, the soft projection satisfies
\begin{equation}
  \left\lVert D\rho_R\right\rVert \le1,
  \qquad
  \left\lVert D^2\rho_R\right\rVert _{\mathrm{op}}\le\frac{6}{R},
  \qquad
  \left\lVert D^3\rho_R\right\rVert _{\mathrm{op}}\le\frac{36}{R^2}.
  \label{eq:rho-higher-derivatives}
\end{equation}
The rotated pieces satisfy
\begin{alignat}{3}
  &\left\lvert \theta_{\bm{U}}(\bm{z})\right\rvert
    \le C_\theta\eta^2,
  &&\quad
    \left\lVert D^r\theta_{\bm{U}}(\bm{z})\right\rVert_{\mathrm{op}}
    \le C_\theta\eta^{2-r},
  &&\quad
    1\le r\le3,
  \label{eq:theta-composition-scales}
  \\
  &\left\lVert Da_{\bm{U}}(\bm{z})\right\rVert
    \le C_a\eta\sqrt T,
  &&\quad
    \left\lVert D^2a_{\bm{U}}(\bm{z})\right\rVert_{\mathrm{op}}
    \le C_a,
  &&\quad
    \left\lVert D^3a_{\bm{U}}(\bm{z})\right\rVert_{\mathrm{op}}
    \le \frac{C_a}{\eta}.
  \label{eq:a-composition-scales}
\end{alignat}
The pseudo-Huber term satisfies
\begin{equation}
  h_R\ge0,
  \qquad h_R(\bm{0})=0,
  \qquad \left\lVert \nabla h_R(\bm{z})\right\rVert \le \frac{1}{4} R,
  \qquad 0\preceq\nabla^2h_R(\bm{z})\preceq \frac{1}{4} \bm{I}. 
  \label{eq:huber-properties}
\end{equation}
\end{lemma}

Since $s_\kappa=\Theta(\kappa)$,
$\vartheta_{\kappa,\bm U}=\theta_{\bm U}/s_\kappa$ and its derivatives
contain a factor of order $\kappa^{-1}$.  The next lemma gives the
corresponding global bounds for $\mathcal R_{\kappa,\bm U}$. Its proof is given in
Appendix~\ref{app:regularity-proof}.

\begin{lemma} 
\label{lem:amplifier-derivatives}
Define
\begin{alignat}{3}
  &C_{\rm grad}=8C_\theta(2C_\theta+C_\psi),
  &&\qquad
  C_{\rm Hess}=16C_\theta(2C_\theta+C_\psi),
  &&\qquad
  C_{\rm third}=32C_\theta(2C_\theta+C_\psi).
  \label{eq:amplifier-derivative-constants}
\end{alignat}
For every $\kappa\ge2$, $0<\eta\le1$, $\bm z\in\mathbb R^d$, and
$\bm v\in\mathbb R^2$,
\begin{alignat}{3}
  &\left\lvert\vartheta_{\kappa,\bm U}(\bm z)\right\rvert
  \le\frac{4C_\theta\eta^2}{\kappa},
  &&\quad
  \left\lVert D^r\vartheta_{\kappa,\bm U}(\bm z)\right\rVert_{\rm op}
  \le\frac{4C_\theta\eta^{2-r}}{\kappa},
  &&\quad
  1\le r\le3,
  \label{eq:attenuated-derivative-scales} \\
  &\left\lVert\nabla_{(z,v)}
    \mathcal R_{\kappa,\bm U}(\bm z,\bm v)\right\rVert
  \le C_{\rm grad}\frac{\eta^2}{\kappa},
  &&\quad
  \left\lVert D^2_{(z,v)}
    \mathcal R_{\kappa,\bm U}(\bm z,\bm v)\right\rVert_{\rm op}
  \le C_{\rm Hess}\frac{\eta}{\kappa},
  \\
  &
  \left\lVert D^3_{(z,v)}
    \mathcal R_{\kappa,\bm U}(\bm z,\bm v)\right\rVert_{\rm op}
  \le\frac{C_{\rm third}}{\kappa}.
  \label{eq:amplifier-derivative-scales}
\end{alignat}
\end{lemma}

\subsubsection{Lower-level Solution and Hyper-objective}
The next lemma solves the lower-level problem and gives the
resulting hyper-objective.  In particular, the main lower solution is
$\bm{z}^*(\bm{x})=\kappa \bm{x}$, and the hyper-objective contains
$\theta_{\bm U}$ at its original scale. Its proof is given in
Appendix~\ref{app:population-solution-proof}.
\begin{lemma}
\label{lem:population-solution}
Define
\begin{equation}
  \eta_0
  =\min\left\{1,\frac1{3C_\theta},
    \frac1{8C_\theta C_\psi}\right\}.
  \label{eq:eta-zero}
\end{equation}
If $0<\eta\le\eta_0$, then
\begin{equation}
  \mathbb{E}_\xi\widehat g_{\bm{U}}^\kappa=g_{\bm{U}}^\kappa,
  \qquad
  \bm{z}_{\bm{U}}^*(\bm{x})=\kappa \bm{x},
  \qquad
  \bm v_{\bm U}^*(\bm x)
  =\bm M_\kappa^{-1}
  \begin{pmatrix}
    \vartheta_{\kappa,\bm U}(\kappa\bm x)\\0
  \end{pmatrix}.
  \label{eq:exact-lower-solution}
\end{equation}
In particular,
\begin{equation}
  v_{1,\bm U}^*(\bm x)
  =\frac{d_\kappa}{s_\kappa}\theta_{\bm U}(\kappa\bm x),
  \qquad
  v_{2,\bm U}^*(\bm x)=\theta_{\bm U}(\kappa\bm x).
  \label{eq:exact-auxiliary-solution}
\end{equation}
Consequently, the hyper-objective is
\begin{equation}
  \Phi_{\bm{U}}^\kappa(\bm{x})
  =H_{\eta,T}\!\left(\bm{U}^\top\rho_R(\kappa \bm{x})\right)
    +h_R(\kappa \bm{x}).
  \label{eq:exact-hyperobjective}
\end{equation}
\end{lemma}

\subsubsection{Smoothness, Strong Convexity, and Condition Number}
We next verify that the construction belongs to the prescribed function class. The following lemma establishes the required smoothness and strong convexity, and relates the construction parameter $\kappa$ to the intrinsic lower-level condition number $\kappa_y$. Its proof is given in
Appendix~\ref{app:regularity-proof}.
\begin{lemma}
\label{lem:regularity}
Define
\begin{alignat}{2}
  &c_{\rm reg}
    =\min\left\{
      \eta_0,
      \frac{1}{2C_{\rm Hess}}
    \right\},
  &&\qquad
    C_\kappa
    =14.
\end{alignat}
Then
\begin{alignat}{2}
  &\bar L_f
    =C_a+\frac{1}{4},
  &&\qquad
    \bar L_g
    =8+C_{\rm Hess},
  \\
  &\bar\rho
    =C_{\rm third}+1,
  &&\qquad
    \bar C_f(\Delta)
    =1+\frac{C_a+115/2}{\sqrt{C_\Delta}}\sqrt{\Delta}.
\end{alignat}
Assume $\kappa\ge2$ and
\begin{equation}
  0<\eta\le c_{\rm reg},
  \qquad
  C_\Delta\eta^2T\le\Delta.
  \label{eq:regularity-range}
\end{equation}
Uniformly in $T,d,\bm{U}$:
\begin{enumerate}
  \item $\nabla f_{\bm{U}}^\kappa$ and $\nabla g_{\bm{U}}^\kappa$ are globally
  $\bar L_f$- and $\bar L_g$-Lipschitz in all variables;
  \item $\sup_{\bm{x},\bm{z},\bm v}\left\lVert \nabla_{(z,v)}f_{\bm{U}}^\kappa(\bm{x},\bm{z},\bm v)\right\rVert 
  \le\bar C_f(\Delta)$;
  \item $g_{\bm{U}}^\kappa(\bm{x},\cdot,\cdot)$ is
  $(2\kappa)^{-1}$-strongly convex;
  \item $\nabla^2g_{\bm{U}}^\kappa$ is globally
  $\bar\rho$-Lipschitz.
\end{enumerate}
Moreover the actual lower condition number satisfies
$\kappa\le\kappa_y\le C_\kappa\kappa$.
\end{lemma}

\subsubsection{Unbiasedness and Variance Bounds}
The stochastic lower-level objective \eqref{eq:sample-hard-instance} multiplies
$\mathcal R_{\kappa,\bm U}$ by a Bernoulli importance weight $\xi/p$.  The next
lemma verifies conditional unbiasedness and bounded variance and gives
an admissible choice of $p$. Its proof is given in
Appendix~\ref{app:oracle-proof}.
\begin{lemma}
\label{lem:oracle-properties}
Define
\begin{equation}
  C_{\rm var}
  =16C_\theta^2\left[
    (4C_\theta+C_\psi)^2+C_\psi^2
  \right].
  \label{eq:variance-constant}
\end{equation}
For $0<\eta\le1$, consider any possibly random query
$(\bm x_t,\bm z_t,\bm v_t)$ generated by
$\mathsf A\in\mathcal A_N^{\mathcal{SFO}}$. Conditionally on the pre-query sigma-field $\mathscr F_{t-1}$,
\begin{alignat}{1}
  &\mathbb E\!\left[\nabla\widehat g_{\bm U}^\kappa\mid\mathscr F_{t-1}\right]
    =\nabla g_{\bm U}^\kappa,
  \label{eq:oracle-unbiasedness}\\
  &\mathbb E\!\left[
    \left\lVert\nabla\widehat g_{\bm U}^\kappa-\nabla g_{\bm U}^\kappa\right\rVert^2
    \mathrel{}\middle|\mathscr F_{t-1}\right]
    \le C_{\rm var}\frac{\eta^4}{\kappa^2}\frac{1-p}{p}.
  \label{eq:oracle-variance-bound}
\end{alignat}
The upper stochastic gradient is exact and has zero variance.
Consequently, the choice
\begin{equation}
  p=
  \begin{cases}
    1,&\sigma=0,\\[1mm]
    \displaystyle
    \min\left\{1,\frac{C_{\rm var}\eta^4}{\sigma^2\kappa^2}\right\},
    &\sigma>0,
  \end{cases}
  \label{eq:p-choice}
\end{equation}
ensures variance at most $\sigma^2$ and satisfies
\begin{equation}
  \frac1p
  =\max\left\{1,\frac{\sigma^2\kappa^2}{C_{\rm var}\eta^4}\right\}.
  \label{eq:p-inverse}
\end{equation}
\end{lemma}

Lemma~\ref{lem:amplifier-derivatives} gives
$\lVert\nabla_{(\bm z,\bm v)}\mathcal R_{\kappa,\bm U}\rVert
=O(\eta^2/\kappa)$.  Multiplication by $\xi/p$ preserves the
conditional mean and gives conditional variance
$O(\eta^4(1-p)/(\kappa^2p))$.  Thus \eqref{eq:p-choice} gives
\[
  \frac1p
  =\Theta\!\left(
    \max\left\{1,\frac{\sigma^2\kappa^2}{\eta^4}\right\}
  \right)
\]
up to numerical constants.  Proposition~\ref{prop:T-over-p} will further combine
the high-probability coordinate-progress bound with a Bernoulli
success-count estimate, and gives a constant probability that the
rotated chain coordinates at the output have progress below $T$
whenever $N\le T/(8p)$.

\subsubsection{Initial Gap and Hypergradient Lower Bound}
We now bound the initial optimality gap and establish a lower bound on the hypergradient norm whenever the coordinate progress is below $T$. 
Note that $\lVert\rho_R(\bm z)\rVert\le R$, so the normalized projected vector
$\rho_R(\bm z)/\eta$ used in the random-rotation argument has norm at most
$R/\eta$. 
However, $D\rho_R(\bm z)$ becomes small when $\lVert\bm z\rVert$ is large. Hence, the term $h_R(\bm z)$ is included so that the lower bound of the hypergradient norm in Lemma~\ref{lem:gap-gradient} remains valid in this region while its Hessian remains uniformly bounded. Its proof is given in Appendix~\ref{app:gap-gradient-proof}.

\begin{lemma}
\label{lem:gap-gradient}
For every orthonormal frame $\bm{U}$,
\begin{equation}
  \Phi_{\bm{U}}^\kappa(\bm{0})-\inf_{\bm{x}}\Phi_{\bm{U}}^\kappa(\bm{x})
  \le C_\Delta\eta^2T.
  \label{eq:hard-instance-gap}
\end{equation}
If
\begin{equation}
  \operatorname{prog}_{\eta/4}
    (\bm{U}^\top\rho_R(\kappa \bm{x}))<T,
  \label{eq:unfinished-assumption}
\end{equation}
then
\begin{equation}
  \left\lVert \nabla\Phi_{\bm{U}}^\kappa(\bm{x})\right\rVert
  >\frac{\kappa\eta}{2}.
  \label{eq:amplified-gradient-lower}
\end{equation}
\end{lemma}
\subsection{Parameter Choice}
\label{sec:parameter choice}

The preceding estimates hold for every fixed frame $\bm U$, so the
parameters can be chosen before $\bm U$ is randomized.

The construction uses three parameter relations:
$\kappa\eta\asymp\epsilon$, $T\eta^2\asymp\Delta$, and
$p^{-1}\asymp\max\{1,\sigma^2\kappa^2/\eta^4\}$.  They set the
stationarity scale, use the gap budget up to constants, and enforce the
variance bound, respectively.  Substitution into $T/p$ yields the
claimed order.  The calculation below specifies constants that satisfy
all assumptions.

Choose
\begin{equation}
  \eta=\frac{4\epsilon}{\kappa},
  \qquad
  X=\frac{\Delta}{2C_\Delta\eta^2},
  \qquad
  T=\lfloor X\rfloor.
  \label{eq:scaling-preview-parameters}
\end{equation}
The universal upper bound on $\epsilon$ in
\eqref{eq:epsilon-range-main} is chosen in
Appendix~\ref{sec:assembly} so that $X\ge32$ and all the smallness
conditions in Lemmas~\ref{lem:population-solution} and
\ref{lem:regularity} hold.  Since $X\ge2$,
\[
  \lfloor X\rfloor\ge X-1\ge\frac{X}{2},
\]
and therefore
\begin{equation}
  \frac{\Delta}{4C_\Delta\eta^2}
  \le T
  \le\frac{\Delta}{2C_\Delta\eta^2}.
  \label{eq:scaling-preview-T-bounds}
\end{equation}
The upper bound gives
\[
  C_\Delta\eta^2T\le\frac{\Delta}{2}\le\Delta,
\]
so the constructed hyper-objective lies in the required initial-gap
class.  Moreover, if the rotated chain coordinates at the output have
progress below $T$, then
Lemma~\ref{lem:gap-gradient} and
\eqref{eq:scaling-preview-parameters} give
\begin{equation}
  \left\lVert \nabla\Phi_{\bm{U}}^\kappa(\widehat{\bm{x}}_N^{\mathsf A})\right\rVert
  >\frac{\kappa\eta}{2}
  =2\epsilon.
  \label{eq:scaling-preview-two-epsilon}
\end{equation}

Choose $p$ as in \eqref{eq:p-choice}.  Combining
\eqref{eq:scaling-preview-T-bounds} with \eqref{eq:p-inverse} gives
\begin{align}
  \frac{T}{8p}
  &\ge
  \frac{\Delta}{32C_\Delta\eta^2}
  \max\left\{
    1,\frac{\sigma^2\kappa^2}{C_{\rm var}\eta^4}
  \right\}
  \notag\\
  &=
  \frac{\Delta\kappa^2}{512C_\Delta\epsilon^2}
  \max\left\{
    1,
    \frac{\sigma^2\kappa^6}
         {256C_{\rm var}\epsilon^4}
  \right\}.
  \label{eq:scaling-preview-T-over-p}
\end{align}
Since $C_{\rm var}$ is a fixed numerical constant, decreasing a
universal prefactor if necessary turns
\eqref{eq:scaling-preview-T-over-p} into
\begin{equation}
  \frac{T}{8p}
  =
  \Omega\!\left(
    \frac{\Delta\kappa^2}{\epsilon^2}
    \max\left\{
      1,\frac{\sigma^2\kappa^6}{\epsilon^4}
    \right\}
  \right).
  \label{eq:scaling-preview-final-order}
\end{equation}
In the noise-dominated regime, this is
\[
  \Omega\!\left(
    \Delta\sigma^2\kappa^8\epsilon^{-6}
  \right).
\]

It remains to relate the progress index to the number of successful
Bernoulli samples.  Let
\[
  S_N=\sum_{t=1}^N\xi_t.
\]
Freshness gives $\mathbb{E}S_N=Np$.  Hence, whenever
$N\le T/(8p)$, Markov's inequality yields
\begin{equation}
  \mathbb{P}(S_N\ge T)
  \le\frac{\mathbb{E}S_N}{T}
  =\frac{Np}{T}
  \le\frac18.
  \label{eq:conditional-counting}
\end{equation}
Under the dimension condition stated below, the random-rotation
argument in Section~\ref{sec:rotation-lifting} shows, with probability
at least $7/8$, that
\begin{equation}
  \operatorname{prog}_{\eta/4}\!\left(
    \bm{U}^\top\rho_R(\kappa\widehat{\bm{x}}_N^{\mathsf A})
  \right)
  \le \min\{T,S_N\}.
  \label{eq:conditional-progress-requirement}
\end{equation}
Together with \eqref{eq:conditional-counting}, this implies that the
rotated chain coordinates at the output have progress below $T$ with
probability at least $3/4$.
Equation~\eqref{eq:scaling-preview-two-epsilon} then gives
\[
  \mathbb{E}
  \left\lVert\nabla\Phi_{\bm{U}}^\kappa(\widehat{\bm{x}}_N^{\mathsf A})\right\rVert
  >
  \frac34(2\epsilon)
  =\frac32\epsilon>\epsilon.
\]

If the frame $\bm U$ were known to the algorithm, one query could have
nonzero components in many chain directions, and
\eqref{eq:conditional-progress-requirement} would not hold.  We will
therefore take $\bm U$ to be Haar-uniform and prove the progress bound
with high probability.  This step does not alter $\eta$, $T$, or $p$.

\subsection{Random Rotation for Adaptive Randomized Algorithms}
\label{sec:rotation-lifting}

We compare the geometric progress of the queries with the Bernoulli
success count.  For the point queried on call \(t\), define
\begin{equation}
  k_t
  \coloneqq
  \operatorname{prog}_{\eta/4}\!\left(
    \bm{U}^\top\rho_R(\bm{z}_t)
  \right).
  \label{eq:geometric-progress}
\end{equation}
The number of successful Bernoulli draws and the corresponding upper
bound on coordinate progress are
\begin{equation}
  S_t\coloneqq\sum_{s=1}^t\xi_s,
  \qquad
  L_t\coloneqq\min\{T,S_t\},
  \qquad
  S_0=L_0=0.
  \label{eq:success-and-visible-progress}
\end{equation}
The goal is to prove, with high probability, that
$k_t\le L_{t-1}$ before call $t$.  On this event, a response with
$\xi_t=0$ depends only on the first $L_{t-1}$ frame columns, whereas a
response with $\xi_t=1$ may depend on one additional frame column.
This dependence on the initial frame columns is combined with the
spherical-cap estimate in the coupled induction below.

\subsubsection{Random-rotation Lemmas}

If the frame were fixed and known, an adaptive algorithm could query a
vector with nonzero components in all chain directions.  We therefore
choose
\[
  \bm{U}=(\bm{u}_1,\ldots,\bm{u}_T)
\]
Haar-uniformly from the Stiefel manifold
\[
  \operatorname{St}(d,T)
  =
  \left\{
    \bm{V}\in\mathbb{R}^{d\times T}:\bm{V}^\top \bm{V}=\bm{I}_T
  \right\}.
\]
The coordinates of an ambient vector \(\bm{r}\in\mathbb{R}^d\) in the random frame are then
\[
  \bm{U}^\top \bm{r}
  =
  \bigl(
    \left\langle \bm{u}_1,\bm{r}\right\rangle ,\ldots,\left\langle \bm{u}_T,\bm{r}\right\rangle 
  \bigr).
\]
The proof constructs a coupled transcript whose dependence on $\bm U$
is through a random number of initial frame columns.  Conditional on
the sigma-field defined below, the remaining columns
have the Haar-completion law on the orthogonal complement of the
span of the recorded frame columns.  The spherical-cap bound therefore
controls the inner product of a bounded adaptive query with each
remaining column.

The number of recorded frame columns is random because it depends on
the previous Bernoulli samples.  A conditional Haar statement for a
fixed number of initial columns therefore cannot be applied directly
with a random index.  The next lemma applies the corresponding
conditional law separately on
each event $\{J=\ell\}$ and combines the resulting identities.  This
justifies the corresponding conditioning for the adaptive transcript.
For a frame \(\bm U\), write
\(\bm U_{\le\ell}=(\bm u_1,\ldots,\bm u_\ell)\). Its proof is given in
Appendix~\ref{app:stopped-prefix-haar-proof}.

\begin{lemma}
\label{lem:stopped-prefix-haar}
Let \(T\ge1\), \(d\ge T\), let \(\bm U\) be Haar-uniform on
\(\operatorname{St}(d,T)\), and let \(Z\) be an independent random
element of a standard Borel space.  Let
\(J=J(Z)\in\{0,\ldots,T\}\) be Borel measurable.  For each deterministic
\(\ell\), let \(\mathsf K_\ell(\bm U_{\le\ell},\cdot)\) denote the
canonical Haar-completion kernel on the orthogonal complement of
\(\operatorname{span}(\bm U_{\le\ell})\), with the evident point mass
interpretation when \(\ell=T\).  Define \(\mathscr H_J\) stratumwise by
\[
  E\in\mathscr H_J
  \quad\Longleftrightarrow\quad
  E\cap\{J=\ell\}\in
  \sigma(Z,\bm U_{\le\ell})
  \quad\text{for every }0\le\ell\le T.
\]
Then, for every bounded Borel function
\(\varphi(Z,\bm U)\),
\[
  \mathbb E\!\left[\varphi(Z,\bm U)\mid\mathscr H_J\right]
  =
  \sum_{\ell=0}^T\bm 1_{\{J=\ell\}}
  \int \varphi\bigl(Z,(\bm U_{\le\ell},\bm V)\bigr)
  \,\mathsf K_\ell(\bm U_{\le\ell},d\bm V)
  \quad\text{a.s.}
\]
Consequently, on \(\{J=\ell<T\}\), each column \(\bm u_j\) with
\(j>\ell\) has, conditionally on \(\mathscr H_J\), the uniform spherical marginal in
\(\operatorname{span}(\bm U_{\le\ell})^\perp\).
\end{lemma}

The next lemma states the random-rotation argument for a
probability-$p$ zero-chain.  Given the first $\ell$
frame directions, a response with $\xi_t=0$ depends on no additional direction,
whereas a response with $\xi_t=1$ may depend on direction $\ell+1$.
In sufficiently high dimension, spherical concentration then bounds the coordinate progress of the vector associated with an adaptive query by $\ell$ with high probability. Its proof is given in
Appendix~\ref{app:haar-proof}.

\begin{lemma}
\label{lem:haar-cited}
There is a numerical constant \(C_{\rm H}>0\) with the following
property. Fix $\eta > 0$ and $R > \eta$. Let \(T\ge1\), \(d\ge T\), and let
\(\bm{U}=(\bm{u}_1,\ldots,\bm{u}_T)\) be Haar-uniform on
\(\operatorname{St}(d,T)\).  Consider an arbitrary randomized
adaptive procedure with random seed \(\omega_{\rm alg}\) that makes
\(N\) calls.  Let \(\mathcal T_{t-1}\in\mathsf T_{t-1}\) denote its
actual transcript before call \(t\), and define its pre-query
information by
\[
  \mathscr F_{t-1}
  =
  \sigma\bigl(
    \omega_{\rm alg},\mathcal T_{t-1}
  \bigr).
\]
Let \(\mathsf D\) be a standard Borel space
containing the public query data used by the oracle.  Before call
\(t\), the procedure chooses a query descriptor \(\bm{a}_t\in\mathsf D\)
and an associated vector \(\bm{r}_t\in\mathbb{R}^d\), both measurable
with respect to \(\mathscr F_{t-1}\), such that
\begin{equation}
  \left\lVert \bm{r}_t\right\rVert \le \frac{R}{\eta}.
  \label{eq:haar-probe-bound}
\end{equation}
Let \(p\in(0,1]\), and let $\xi_1,\ldots,\xi_N$ be i.i.d.\
$\operatorname{Ber}(p)$ variables such that
$\bm U$, $\omega_{\rm alg}$, and $(\xi_1,\ldots,\xi_N)$ are mutually
independent.  At call $t$, the oracle uses $\xi_t$.  In particular,
$\xi_t$ is independent of the actual pre-query sigma-field
$\mathscr F_{t-1}$.
Let \(\bm{Y}_t\) denote the oracle response, taking values in a standard
Borel response space \(\mathsf Y\).  We impose the following
response-dependence condition.  For every
\(t\in\{1,\ldots,N\}\), integer
\(\ell\in\{0,\ldots,T\}\), and bit \(b\in\{0,1\}\), there is a
Borel map
\begin{equation}
  M_{t,\ell}^{(b)}
  :
  \mathsf T_{t-1}\times\mathsf D\times\mathbb{R}^d
  \times\operatorname{St}(d,\ell_b)
  \longrightarrow
  \mathsf Y,
  \qquad
  \ell_b=\min\{T,\ell+b\},
  \label{eq:formal-prefix-map}
\end{equation}
where \(\mathsf T_{t-1}\) is the transcript space before call \(t\),
such that, on the event
\begin{equation}
  \left\{
    \xi_t=b,\quad
    \operatorname{prog}_{1/4}(\bm{U}^\top \bm{r}_t)\le\ell
  \right\},
  \label{eq:formal-prefix-event}
\end{equation}
the actual response satisfies
\begin{equation}
  \bm{Y}_t
  =
  M_{t,\ell}^{(b)}
  \bigl(
    \mathcal{T}_{t-1},\bm{a}_t,\bm{r}_t,
    \bm{u}_1,\ldots,\bm{u}_{\ell_b}
  \bigr).
  \label{eq:formal-prefix-identity}
\end{equation}
Thus a failed call has a representation using only the first
\(\ell\) columns, while a successful call has a representation using
only the first \(\min\{T,\ell+1\}\) columns.  The requirement is the
existence of the uniform Borel maps in
\eqref{eq:formal-prefix-map}; a pointwise statement that happens to
mention no later column is not sufficient.
After the \(N\) calls, append an additional vector \(\bm{r}_{N+1}\)
determined by the procedure's output and satisfying
\(\left\lVert \bm{r}_{N+1}\right\rVert \le\frac{R}{\eta}\), but return
no oracle response to this vector.  Put \(n=N+1\).  If
\(\delta\in(0,1)\) and 
\begin{equation}
  d
  \ge
  T+
  C_{\rm H}(\frac{R}{\eta})^2 n
  \log\left(\frac{2n^2T}{\delta}\right),
  \label{eq:haar-dimension-general}
\end{equation}
then, with probability at least \(1-\delta\), simultaneously,
\begin{align}
  &\operatorname{prog}_{1/4}(\bm{U}^\top \bm{r}_t)
  \le L_{t-1},
  \quad 1\le t\le N,
  \label{eq:haar-query-conclusion}\\
  &\operatorname{prog}_{1/4}(\bm{U}^\top \bm{r}_{N+1})
  \le L_N.
  \label{eq:haar-output-conclusion}
\end{align}
The probability is joint over \(\bm{U}\), the procedure's private
randomness, and the fresh Bernoulli variables.
\end{lemma}

The proof uses the adaptive random-rotation argument of
\citet[Appendix~B.1.2, especially Lemmas~13--14]
{arjevani2023lower}.
Lemma~14 gives the required conditional rotational invariance, and the
proof of Lemma~13 combines it with spherical concentration.
Lemma~\ref{lem:stopped-prefix-haar} extends the conditional law from a
fixed number of initial columns to the random number used here.  It
remains to verify the required dependence of the complete bilevel
oracle response on the initial frame columns.

For the present hard instance, the objects in
Lemma~\ref{lem:haar-cited} are defined as follows.  At call \(t\), the
public query descriptor is
\[
  \bm{a}_t=(\bm{x}_t,\bm{z}_t,\bm v_t)
  \in
  \mathsf D
  \coloneqq
  \mathbb{R}^d\times\mathbb{R}^d\times\mathbb{R}^2,
\]
and the associated vector is
\[
  \bm{r}_t=\frac{\rho_R(\bm{z}_t)}{\eta}.
\]
The complete oracle response is
\[
  \bm{Y}_t
  =
  \left(
    \nabla\widehat f_{\bm{U}}^\kappa(\bm{x}_t,\bm{z}_t,\bm v_t),
    \nabla\widehat g_{\bm{U}}^\kappa
      (\bm{x}_t,\bm{z}_t,\bm v_t;\xi_t)
  \right)
  \in
  \mathsf Y
  \coloneqq
  \mathbb{R}^{2d+2}\times\mathbb{R}^{2d+2},
\]
where
\(\widehat f_{\bm{U}}^\kappa=f_{\bm{U}}^\kappa\).
The original transcript \(\mathcal T_{t-1}\) consists only of the
query-response pairs observed before call \(t\).  The algorithm's
complete pre-query information is
\[
  \mathscr F_{t-1}
  =
  \sigma\bigl(
    \omega_{\rm alg},\mathcal T_{t-1}
  \bigr),
\]
as in \eqref{eq:prequery-filtration}. 
Since $R=230\eta\sqrt T$, and $\|\rho_R(\bm z_t)\|\le R$, the associated vector satisfies
\[
  \|\bm r_t\|
  =
  \frac{\|\rho_R(\bm z_t)\|}{\eta}
  \le
  \frac{R}{\eta}
  =
  230\sqrt T,
\]
so the norm condition in Lemma~\ref{lem:haar-cited} is satisfied.
Moreover,
\[
  \operatorname{prog}_{1/4}(\bm{U}^\top \bm{r}_t)
  =
  \operatorname{prog}_{\eta/4}
  \bigl(\bm{U}^\top\rho_R(\bm{z}_t)\bigr).
\]
At call \(t\), the number of initial frame columns is
\(\ell=L_{t-1}\), and the bit \(b\) is the realized Bernoulli sample
\(\xi_t\).  Lemma~\ref{lem:bilevel-interface} below constructs the
maps \(M_{t,\ell}^{(0)}\) and \(M_{t,\ell}^{(1)}\): the map for
$\xi_t=0$ uses only \(\bm{u}_1,\ldots,\bm{u}_\ell\), whereas the map
for $\xi_t=1$ uses only
\(\bm{u}_1,\ldots,\bm{u}_{\min\{T,\ell+1\}}\).

Finally, if \(\widehat{\bm{x}}_N^{\mathsf A}\) is the algorithm's output, the
additional output vector is
\[
  \bm{r}_{N+1}
  =
  \frac{\rho_R(\kappa\widehat{\bm{x}}_N^{\mathsf A})}{\eta}.
\]
The factor \(\kappa\widehat{\bm{x}}_N^{\mathsf A}\) appears because the lower-level
solution satisfies
\(\bm{z}^*(\widehat{\bm{x}}_N^{\mathsf A})=\kappa\widehat{\bm{x}}_N^{\mathsf A}\), and the hyper-objective evaluates the chain at this lower solution.

\subsubsection{\texorpdfstring{Probability--$p$}{Probability-p} Zero-chain Property}

In the following lemma, we verify that the complete joint oracle response satisfies the
response-dependence condition of Lemma~\ref{lem:haar-cited}, including
at the threshold points of the soft indicators. Its proof is given in
Appendix~\ref{app:interface-proof}.

\begin{lemma}
\label{lem:bilevel-interface}
Fix a query \((\bm{x},\bm{z},\bm v)\), and write
\begin{equation}
  \bm{q}=\bm{q}_{\bm{U}}(\bm{z})=\bm{U}^\top\rho_R(\bm{z}).
  \label{eq:q-at-query}
\end{equation}
For every fixed \(\ell\in\{0,\ldots,T\}\), there are
Borel maps \(M_\ell^{(0)}\) and \(M_\ell^{(1)}\), depending only on
the public query and respectively on
\[
  \bm{u}_1,\ldots,\bm{u}_\ell
  \quad\text{and}\quad
  \bm{u}_1,\ldots,\bm{u}_{\min\{T,\ell+1\}},
\]
such that, whenever
\begin{equation}
  \operatorname{prog}_{\eta/4}(\bm{q})\le\ell,
  \label{eq:bilevel-prefix-premise}
\end{equation}
the joint upper-lower oracle response with sample bit
\(\xi=b\) is exactly \(M_\ell^{(b)}\).  In particular, these maps do
not take the random progress index
\(\operatorname{prog}_{\eta/4}(\bm{q})\) as an argument.  They verify the
response-dependence condition
\eqref{eq:formal-prefix-map}--\eqref{eq:formal-prefix-identity}.
\end{lemma}

\subsubsection{Lower Bound on Oracle Calls}

Combining Lemmas~\ref{lem:haar-cited} and Lemma~\ref{lem:bilevel-interface}
with the Bernoulli success-count estimate~\eqref{eq:conditional-counting}
gives the following progress estimate. Its proof is given in
Appendix~\ref{app:progress-proof}.

\begin{proposition}
\label{prop:T-over-p}
Fix $
  \mathsf A
  \in
  \mathcal A_N^{\mathcal{SFO}}(d,d+2)$,
and let \(\widehat{\bm{x}}_N^{\mathsf A}\) be its output after \(N\) calls
to the oracle \eqref{eq:sample-hard-instance}.  Set $S_N=\sum_{t=1}^N\xi_t$.
If
\begin{equation}
  N\le\frac{T}{8p}
  \label{eq:N-progress-condition}
\end{equation}
and
\begin{equation}
  d
  \ge
  T+
  C_{\rm H}(230^2T)(N+1)\log(16(N+1)^2T),
  \label{eq:dimension-choice}
\end{equation}
then, jointly over \(\bm{U}\), the algorithm's private randomness, and the
fresh oracle samples,
\begin{equation}
  \mathbb{P}\left(
    \operatorname{prog}_{\eta/4}\!\left(
      \bm{U}^\top\rho_R(\kappa\widehat{\bm{x}}_N^{\mathsf A})
    \right)<T
  \right)
  \ge\frac34.
  \label{eq:unfinished-probability}
\end{equation}
\end{proposition}

\begin{remark}
The dimension condition \eqref{eq:dimension-choice} ensures, through
spherical concentration, that an adaptive query has a small projection
onto each remaining frame column with high probability.  It does not
add a factor to the oracle lower bound.  After the call budget $N$ is
fixed, the proof chooses a sufficiently large finite dimension $d$.
The conclusion then holds for every
$\mathsf A\in\mathcal A_N^{\mathcal{SFO}}(d,d+2)$, although the final
deterministic frame and hence the hard instance may depend on
$\mathsf A$.  The oracle complexity is determined by $T/p$; the
dimension condition is used only in the random-rotation argument.
\end{remark}
Proposition~\ref{prop:T-over-p} implies that, for
$N\le T/(8p)$, the rotated chain coordinates at the output have
progress below $T$ with probability at least $3/4$.
Lemma~\ref{lem:gap-gradient} then gives a hypergradient
above the target accuracy on this event.  Appendix~\ref{sec:assembly}
combines these estimates and proves Theorem~\ref{thm:main}.

\section{Conclusion and Future Work}
\label{sec:conclusion}

We established the stochastic first-order lower bound
\[
  \Omega\!\left(
    \frac{\Delta\kappa_y^2}{\epsilon^2}
    \max\left\{
      1,
      \frac{\sigma^2\kappa_y^6}{\epsilon^4}
    \right\}
  \right)
\]
for smooth NC-SC bilevel optimization over arbitrary adaptive randomized first-order algorithms. In the noise-dominated regime, this becomes
\(\Omega(\Delta\sigma^2\kappa_y^8\epsilon^{-6})\), showing that the
\(\epsilon^{-6}\) dependence of the best-known first-order upper
bounds \citep{chen2025near, gu2026sgha} is information-theoretically unavoidable under the bounded-variance
\(\mathcal{SFO}\) model. The optimal dependence on the condition number
remains open, however, in view of the upper bound
$
  \tilde{\mathcal O}
  \!\left(
    \bar\kappa_y^{11}\epsilon^{-6}
  \right)
$
of \citet{chen2025near}. Another interesting direction is to establish
lower bounds in a fixed ambient dimension, rather than allowing the
dimension to grow with the target accuracy.


\section*{AI Disclosure}


The hard-instance construction and its presentation were refined through iterative collaboration with OpenAI's GPT-5.6 Sol. The authors supplied the model with an
independently developed proof draft targeting an $\epsilon^{-6}$ lower bound, motivated by their unpublished upper-bound result showing that $\epsilon^{-4}$ complexity bound is attainable when the lower-level objective is quadratic or satisfies an averaged stochastic smoothness condition, but not necessarily for the general problem class. The authors independently
checked all mathematical arguments and take full responsibility for the
content of the paper. An accompanying Lean formalization, developed with assistance from
Codex and available at
\url{https://github.com/Wu-Qilong/Lower-Bounds-for-Nonconvex-Strongly-Convex-Bilevel-Optimization},
provides machine-checked verification of the lower-bound argument.

\bibliography{main}

\clearpage
\appendix

\section{Proof of Lemma~\ref{lem:extractor}}
\label{app:extractor-proof}

\begin{proof}
We divide the proof into four steps.

\emph{Step 1: the soft indicators are smooth and have dimension-free
derivative bounds.}
The absolute value in $\upsilon(t)=\Gamma(\left\lvert t\right\rvert )$ causes no
nonsmoothness at zero because $\Gamma$ is identically zero on a
neighborhood of zero.  Hence $\upsilon\in C^\infty(\mathbb{R})$.  Define
\begin{equation}
  K_\upsilon
  =\max_{0\le r\le3}\left\lVert \upsilon^{(r)}\right\rVert _\infty.
  \label{eq:K-upsilon}
\end{equation}
This is a finite numerical constant.
For $1\le r\le3$, directions $\bm{u}_1,\ldots,\bm{u}_r$, and a coordinate
$j\ge i$,
\[
  \bigl[D^rV_i^\eta(\bm{z})[\bm{u}_1,\ldots,\bm{u}_r]\bigr]_j
  =\eta^{-r}\upsilon^{(r)}(z_j/\eta)
    \prod_{\ell=1}^r(\bm{u}_\ell)_j.
\]
Consequently,
\begin{align}
  \left\lVert D^rV_i^\eta(\bm{z})[\bm{u}_1,\ldots,\bm{u}_r]\right\rVert ^2
  &\le K_\upsilon^2\eta^{-2r}
    \sum_{j=i}^T\prod_{\ell=1}^r\left\lvert (\bm{u}_\ell)_j\right\rvert ^2
  \notag\\
  &\le K_\upsilon^2\eta^{-2r}
    \prod_{\ell=1}^r\sum_{j=i}^T\left\lvert (\bm{u}_\ell)_j\right\rvert ^2
  \notag\\
  &\le K_\upsilon^2\eta^{-2r}
    \prod_{\ell=1}^r\left\lVert \bm{u}_\ell\right\rVert ^2.
  \label{eq:Vi-derivative-detail}
\end{align}

For each Euclidean dimension $n$, define
$\chi_n(\bm{\nu})=\Gamma(1-\left\lVert \bm{\nu}\right\rVert )$ on $\mathbb{R}^n$.  If
$\left\lVert \bm{\nu}\right\rVert \le1/2$, then
$1-\left\lVert \bm{\nu}\right\rVert \ge1/2$ and $\chi_n(\bm{\nu})=1$.  If $\left\lVert \bm{\nu}\right\rVert \ge3/4$, then
$1-\left\lVert \bm{\nu}\right\rVert \le1/4$ and $\chi_n(\bm{\nu})=0$.  On the fixed annulus
$1/2<\left\lVert \bm{\nu}\right\rVert <3/4$, the norm is smooth and its first three derivatives
have dimension-free operator bounds.  Near the origin, where the norm
itself is not differentiable, $\chi_n$ is constant.  Thus every
$\chi_n$ is globally $C^\infty$.  For completeness, on this annulus
put $q=\left\lVert \bm{\nu}\right\rVert $ and $\bm{e}=\bm{\nu}/q$.  Then
\[
  Dq[\bm{h}]=\left\langle \bm{e},\bm{h}\right\rangle ,
  \qquad
  D^2q=\frac{\bm{I}-\bm{e}\bm{e}^\top}{q}.
\]
Hence $\left\lVert Dq\right\rVert \le1$ and $\left\lVert D^2q\right\rVert _{\mathrm{op}}\le2$.  Moreover
$D\bm{e}[\bm{h}]=(\bm{I}-\bm{e}\bm{e}^\top)\bm{h}/q$, so $\left\lVert D\bm{e}\right\rVert _{\mathrm{op}}\le2$; differentiating
$D^2q=(\bm{I}-\bm{e}\bm{e}^\top)/q$ once more gives
$\left\lVert D^3q\right\rVert _{\mathrm{op}}\le12$.  Applying the chain rule to the fixed
one-dimensional function $t\mapsto\Gamma(1-t)$ now proves that the
first three derivatives of $\chi_n$ are bounded independently of the
dimension.  Therefore the following constant is finite:
\begin{equation}
  K_\chi
  =\max\left\{1,
    \max_{1\le r\le3}\sup_{n\ge1}\sup_{\bm{\nu}\in\mathbb{R}^n}
       \left\lVert D^r\chi_n(\bm{\nu})\right\rVert _{\mathrm{op}}
  \right\}.
  \label{eq:K-chi}
\end{equation}
Define
\begin{equation}
  C_{\rm gate}=K_\chi(1+K_\upsilon)^3.
  \label{eq:C-gate}
\end{equation}
To see explicitly why this one constant covers all derivative orders,
write $h_i^\eta=\chi_{T-i+1}\circ \bm{V}_i^\eta$.  For unit directions,
the chain rule gives
\begin{align*}
  Dh_i^\eta
    &=D\chi[DV_i^\eta],\\
  D^2h_i^\eta
    &=D^2\chi[DV_i^\eta,DV_i^\eta]
      +D\chi[D^2\bm{V}_i^\eta],\\
  D^3h_i^\eta
    &=D^3\chi[DV_i^\eta,DV_i^\eta,DV_i^\eta]
      +3D^2\chi[D^2\bm{V}_i^\eta,DV_i^\eta]
      +D\chi[D^3\bm{V}_i^\eta].
\end{align*}
Equation~\eqref{eq:Vi-derivative-detail} implies
$\left\lVert D^rV_i^\eta\right\rVert _{\mathrm{op}}\le K_\upsilon\eta^{-r}$.
Consequently the respective coefficients are bounded by
$K_\chi K_\upsilon$,
$K_\chi(K_\upsilon^2+K_\upsilon)$, and
$K_\chi(K_\upsilon^3+3K_\upsilon^2+K_\upsilon)$, each of which is
at most $C_{\rm gate}$.  Together with $0\le h_i^\eta\le1$, this gives
\begin{equation}
  \left\lVert D^rh_i^\eta(\bm{z})\right\rVert _{\mathrm{op}}\le C_{\rm gate}\eta^{-r},
  \qquad 0\le r\le3.
  \label{eq:hi-derivative-detail}
\end{equation}

\emph{Step 2: at most one summand has a nonzero value or a nonzero
derivative of order at most three.}
Assume first that $m<T$.  If $i\le m$, the vector
$\bm{V}_i^\eta(\bm{z})$ contains its $m$--th coordinate.  Because
$\left\lvert z_m\right\rvert >\eta/2$, this coordinate equals one.  Hence
$\left\lVert \bm{V}_i^\eta(\bm{z})\right\rVert \ge1$ and $h_i^\eta$ is identically zero in a
neighborhood of $\bm{z}$.  Therefore the value and all derivatives of
$h_i^\eta\mathcal{L}_i^\eta$ vanish at $\bm{z}$.

If $i\ge m+2$, then $i-1\ge m+1$.  The definition of $m$ implies
$\left\lvert z_{i-1}\right\rvert \le\eta/2$.  The flatness statement in
Lemma~\ref{lem:standard-chain} then shows that the value and all
derivatives of $\mathcal{L}_i^\eta$ vanish at $\bm{z}$, so the same is true
of the product.  The only index not covered by these two cases is
$i=m+1$.  If $m=T$, the first case applies to every $i\le T$ and all
summands have values and derivatives that vanish at $\bm{z}$.

\emph{Step 3: derivative bounds for $b_{\eta,T}$.}
The derivatives of the fixed two-variable function $Q$ through order
three are bounded.  Define
\begin{equation}
  C_{\rm link}
  =\max_{0\le r\le3}
    \left\{2^{r/2}\sup_{(a,b)\in\mathbb{R}^2}
       \left\lVert D^rQ(a,b)\right\rVert _{\mathrm{op}}\right\},
  \qquad
  C_{\rm ext}=8C_{\rm gate}C_{\rm link}.
  \label{eq:extractor-constant}
\end{equation}
The linear map
$\bm{z}\mapsto(z_{i-1}/\eta,z_i/\eta)$ has operator norm at most
$\sqrt2/\eta$; hence
\begin{equation}
  \left\lVert D^r\mathcal{L}_i^\eta(\bm{z})\right\rVert _{\mathrm{op}}
  \le C_{\rm link}\eta^{-r},
  \qquad 0\le r\le3.
  \label{eq:link-derivative-fixed}
\end{equation}
At the fixed point $\bm{z}$, Step~2 leaves at most one nonzero summand.
If $m=T$, the value and derivatives through order three of every
summand vanish at $\bm{z}$, and the desired bound is
immediate.  Hence it remains only to consider $m<T$.
Thus the product rule and \eqref{eq:hi-derivative-detail} yield
\begin{align*}
  \left\lVert D^rb_{\eta,T}(\bm{z})\right\rVert _{\mathrm{op}}
  &\le \eta^2\sum_{s=0}^r \binom{r}{s}
    \left\lVert D^sh_{m+1}^\eta(\bm{z})\right\rVert _{\mathrm{op}}
    \left\lVert D^{r-s}\mathcal{L}_{m+1}^\eta(\bm{z})\right\rVert _{\mathrm{op}}\\
  &\le C_{\rm gate}C_{\rm link}\eta^{2-r}
     \sum_{s=0}^r\binom rs\\
  &=2^rC_{\rm gate}C_{\rm link}\eta^{2-r}\\
  &\le C_{\rm ext}\eta^{2-r},
\end{align*}
where $2^r\le8$ for $r\le3$.  

\emph{Step 4: exact derivative-truncation identities.}
Recall that \(m=\operatorname{prog}_{\eta/2}(\bm{z})\) and
\(k=\operatorname{prog}_{\eta/4}(\bm{z})\). Since
\(\eta/2>\eta/4\), we necessarily have \(m\leq k\).

We first assume that \(k<T\). By the definition of \(k\), every
coordinate \(j>k\) satisfies \(|z_j|\leq\eta/4\), and therefore
\[
\upsilon(z_j/\eta)=0,
\qquad j>k.
\]
The derivatives of \(\upsilon(z_j/\eta)\) through order three also
vanish. This remains true when \(|z_j|=\eta/4\), because \(\Gamma\)
is flat at \(1/4\).

Suppose first that \(m=k<T\). If \(k\geq1\), then
\(|z_k|>\eta/2\), so \(\upsilon(z_k/\eta)=1\). Hence coordinate
\(k\) makes every soft indicator \(h_i^\eta\) with \(i\leq k\)
identically zero in a neighborhood of \(\bm z\). When \(k=0\), there
are no such soft indicators. Moreover, the
value and all derivatives of every summand with \(i\geq k+2\) vanish,
since
\[
|z_{i-1}|
\leq \frac{\eta}{4}
\leq \frac{\eta}{2}.
\]
Finally, \(\bm{V}_{k+1}^\eta(\bm{z})=0\), so
\(h_{k+1}^\eta(\bm{z})=\Gamma(1)=1\). It follows that, at \(\bm{z}\) and
throughout the set obtained by fixing the first $k$ coordinates and
varying the coordinates $j>k$ subject to $|z_j|\le\eta/4$,
\begin{equation}
\label{eq:frontier-case-m-equals-k}
b_{\eta,T}(\bm{z})
=
\eta^2\mathcal L_{k+1}^\eta(\bm{z}),
\qquad
A_{\eta,T}(\bm{z})
=
\eta^2\sum_{i=1}^{k}\mathcal L_i^\eta(\bm{z}).
\end{equation}
The sum is interpreted as zero when \(k=0\). The first expression
in \eqref{eq:frontier-case-m-equals-k} depends only on coordinates
through \(k+1\), whereas the second depends only on coordinates
through \(k\).

Now suppose that \(m<k\). The value and all derivatives of every
summand with \(i\geq m+2\) vanish because
\(|z_{i-1}|\leq\eta/2\). Hence
\begin{equation}
\label{eq:frontier-case-m-less-k}
H_{\eta,T}(\bm{z})
=
\eta^2\sum_{i=1}^{m+1}\mathcal L_i^\eta(\bm{z}),
\qquad
b_{\eta,T}(\bm{z})
=
\eta^2h_{m+1}^\eta(\bm{z})\mathcal L_{m+1}^\eta(\bm{z}).
\end{equation}
Since \(m+1\leq k\), the first expression in
\eqref{eq:frontier-case-m-less-k} depends only on the first \(k\)
coordinates. In the second expression, for every \(j>k\),
$\upsilon(z_j/\eta)$ and its derivatives through order three vanish.
Thus the soft indicator \(h_{m+1}^\eta\),
and hence the full product
\(h_{m+1}^\eta\mathcal L_{m+1}^\eta\), depends only on the first
\(k\) coordinates. This is stronger than the claimed dependence
through \(k+1\). Since \(A_{\eta,T}=H_{\eta,T}-b_{\eta,T}\), the
function \(A_{\eta,T}\) also depends only on the first \(k\)
coordinates.

We next verify the coordinate-truncation conclusion explicitly. Define
\[
\bm{z}_A(t)=\bm{P}_k\bm{z}+t(\bm{I}-\bm{P}_k)\bm{z},
\qquad
\bm{z}_b(t)=\bm{P}_{k+1}\bm{z}+t(\bm{I}-\bm{P}_{k+1})\bm{z},
\qquad t\in[0,1].
\]
These paths satisfy
\[
\bm{z}_A(\bm{0})=\bm{P}_k\bm{z},
\qquad \bm{z}_A(\bm{1})=\bm{z},
\qquad
\bm{z}_b(\bm{0})=\bm{P}_{k+1}\bm{z},
\qquad \bm{z}_b(\bm{1})=\bm{z}.
\]
For every \(t\in[0,1]\) and \(j>k\),
\[
|[\bm{z}_A(t)]_j|
=
t|z_j|
\leq\frac{\eta}{4}.
\]
Similarly, for every \(j>k+1\),
\[
|[\bm{z}_b(t)]_j|
=
t|z_j|
\leq\frac{\eta}{4}.
\]
The first \(k\) coordinates remain unchanged along both paths.
Moreover, \(|z_{k+1}|\leq\eta/4\). Consequently, for every
\(t\in[0,1]\),
\[
\operatorname{prog}_{\eta/4}(\bm{z}_A(t))
=
\operatorname{prog}_{\eta/4}(\bm{z}_b(t))
=
k,
\]
and
\[
\operatorname{prog}_{\eta/2}(\bm{z}_A(t))
=
\operatorname{prog}_{\eta/2}(\bm{z}_b(t))
=
m.
\]
Therefore the same case, either \(m=k\) or \(m<k\), and the same
explicit formulas
\eqref{eq:frontier-case-m-equals-k}--%
\eqref{eq:frontier-case-m-less-k}
remain valid along the corresponding paths.

Since the explicit formulas above contain no coordinates after \(k\)
for \(A_{\eta,T}\) and none after \(k+1\) for \(b_{\eta,T}\) throughout
the corresponding regions where the varied coordinates remain in
$[-\eta/4,\eta/4]$, we obtain
\begin{equation}
\label{eq:frontier-value-cylinder}
A_{\eta,T}(\bm{z})
=
A_{\eta,T}(\bm{P}_k\bm{z}),
\qquad
b_{\eta,T}(\bm{z})
=
b_{\eta,T}(\bm{P}_{k+1}\bm{z}).
\end{equation}
More generally, for every \(1\leq r\leq3\), the order-$r$ derivatives
satisfy
\begin{equation}
\label{eq:frontier-A-jet-cylinder}
D^rA_{\eta,T}(\bm{z})
=
D^rA_{\eta,T}(\bm{P}_k\bm{z})
\circ(\bm{P}_k,\ldots,\bm{P}_k),
\end{equation}
and
\begin{equation}
\label{eq:frontier-b-jet-cylinder}
D^rb_{\eta,T}(\bm{z})
=
D^rb_{\eta,T}(\bm{P}_{k+1}\bm{z})
\circ(\bm{P}_{k+1},\ldots,\bm{P}_{k+1}).
\end{equation}

If \(k=T\), the identities
\eqref{eq:frontier-value-cylinder}--%
\eqref{eq:frontier-b-jet-cylinder}
are automatic from the convention \(\bm{P}_s=\bm{I}\) for \(s\geq T\).

It remains to check the two threshold boundaries. At
\(|z_j|=\eta/4\), every derivative entering from the transition side
vanishes because \(\Gamma\) is flat at \(1/4\). At
\(|z_j|=\eta/2\), the function \(\upsilon(z_j/\eta)\) is flat at
one, while the value and all derivatives of the next summand vanish by
Lemma~\ref{lem:standard-chain}. Hence the value and derivative identities
\eqref{eq:frontier-value-cylinder}--%
\eqref{eq:frontier-b-jet-cylinder}
remain valid when a coordinate lies exactly at either threshold.
This proves all the stated identities.
\end{proof}

\section{Proofs for Section \ref{sec:analytic}}
\label{app:analytic-proofs}

\subsection{Proof of Lemma~\ref{lem:composition-scales}}
\label{app:composition-proof}

\begin{proof}
Write
\[
  s(\bm{z})=\sqrt{1+\left\lVert \bm{z}\right\rVert ^2/R^2},
  \qquad
  \alpha(\bm{z})=s(\bm{z})^{-1},
  \qquad
  \bm{r}=\rho_R(\bm{z})=\bm{z}/s(\bm{z}).
\]
Direct differentiation gives
\begin{equation}
  D\rho_R(\bm{z})
  = \alpha \bm{I} - \frac{\alpha^3}{R^2}\bm{z}\bm{z}^{\top}
  = \frac{\bm{I}}{s(\bm{z})}
    -\frac{\bm{r}\bm{r}^\top}{R^2s(\bm{z})}.
  \label{eq:rho-jacobian}
\end{equation}

The two distinct eigenvalues of this symmetric matrix are in $[0,1]$, so
$\left\lVert D\rho_R(\bm{z})\right\rVert _{\mathrm{op}}\le1$.
Indeed, if \(\bm{z}=0\), then \(\alpha(\bm{0})=1\), and
\[
  D\rho_R(\bm{0})=\bm{I}.
\]
Consequently,
\[
  \|D\rho_R(\bm{0})\|_{\mathrm{op}}=\|\bm{I}\|_{\mathrm{op}}=1.
\]
Suppose now that \(\bm{z}\neq0\), and define the unit vector $\bm{e}=\frac{\bm{z}}{\|\bm{z}\|}$. Every vector \(\bm{h}\in\mathbb{R}^d\) has the orthogonal decomposition
\[
  \bm{h}=\bm{h}_\parallel+\bm{h}_\perp,
  \qquad
  \bm{h}_\parallel=\langle \bm{e},\bm{h}\rangle \bm{e},
  \qquad
  \bm{h}_\perp=\bm{h}-\langle \bm{e},\bm{h}\rangle \bm{e},
\]
where
\[
  \bm{h}_\parallel\in\operatorname{span}\{\bm{z}\},
  \qquad
  \langle \bm{z},\bm{h}_\perp\rangle=0.
\]

We first examine the action of \(D\rho_R(\bm{z})\) on the orthogonal
subspace \(\bm{z}^\perp\).  Since
\(\langle \bm{z},\bm{h}_\perp\rangle=0\), 
\begin{align}
  D\rho_R(\bm{z})\bm{h}_\perp
  =
  \alpha(\bm{z})\bm{h}_\perp
  -
  \frac{\alpha(\bm{z})^3}{R^2}
  \bm{z} \bm{z}^\top \bm{h}_\perp
  =
  \alpha(\bm{z})\bm{h}_\perp.
  \label{eq:soft-projection-tangential}
\end{align}
Thus every vector orthogonal to \(\bm{z}\) is an eigenvector with
eigenvalue \(\alpha(\bm{z})\).

Next, consider the radial direction \(\bm{e}\).  Because \(\bm{z}=\|\bm{z}\|\bm{e}\),
we have
\[
  \bm{z}\bm{z}^\top \bm{e}
  =
  \bm{z}\langle \bm{z},\bm{e}\rangle
  =
  \|\bm{z}\|^2\bm{e}.
\]
Therefore,
\begin{align}
  D\rho_R(\bm{z})\bm{e}
  =
  \left(
    \alpha(\bm{z})
    -
    \frac{\alpha(\bm{z})^3\|\bm{z}\|^2}{R^2}
  \right)\bm{e}
  =
  \alpha(\bm{z})
  \left(
    1-\frac{\alpha(\bm{z})^2\|\bm{z}\|^2}{R^2}
  \right)\bm{e}.
  \label{eq:soft-projection-radial-first}
\end{align}
By the definition of \(\alpha(\bm{z})\),
\begin{align}
  1-\frac{\alpha(\bm{z})^2\|\bm{z}\|^2}{R^2}
  &=
  1-
  \frac{\|\bm{z}\|^2/R^2}
       {1+\|\bm{z}\|^2/R^2}
  \notag\\
  &=
  \frac{1}
       {1+\|\bm{z}\|^2/R^2}
  \notag\\
  &=
  \alpha(\bm{z})^2.
\end{align}
Substituting this identity into
\eqref{eq:soft-projection-radial-first} yields
\begin{equation}
  D\rho_R(\bm{z})\bm{e}
  =
  \alpha(\bm{z})^3\bm{e}.
  \label{eq:soft-projection-radial}
\end{equation}
Hence the eigenvalue in the radial direction is
\(\alpha(\bm{z})^3\), whereas the eigenvalue in every tangential direction
is \(\alpha(\bm{z})\).

Since \(D\rho_R(\bm{z})\) is symmetric, its operator norm is the largest
absolute value of its eigenvalues.  Moreover,
\[
  0<\alpha(\bm{z})
  =
  \frac{1}{\sqrt{1+\|\bm{z}\|^2/R^2}}
  \le1.
\]
Consequently,
\[
  0<\alpha(\bm{z})^3\le\alpha(\bm{z})\le1,
\]
and therefore
\begin{equation}
  \|D\rho_R(\bm{z})\|_{\mathrm{op}}
  =
  \max\{\alpha(\bm{z}),\alpha(\bm{z})^3\}
  =
  \alpha(\bm{z})
  \le1.
\end{equation}

We now verify the two higher-order
bounds explicitly.  For directions $\bm{h},\bm{k},\bm{\ell}$,
\begin{align*}
  &D\alpha[\bm{h}]
  =-\frac{\alpha^3}{R^2}\left\langle \bm{z},\bm{h}\right\rangle ,\\
  &D^2\alpha[\bm{h},\bm{k}]
  =\frac{3\alpha^5}{R^4}\left\langle \bm{z},\bm{h}\right\rangle \left\langle \bm{z},\bm{k}\right\rangle 
    -\frac{\alpha^3}{R^2}\left\langle \bm{h},\bm{k}\right\rangle ,\\
  &D^3\alpha[\bm{h},\bm{k},\bm{\ell}]
  =-\frac{15\alpha^7}{R^6}
      \left\langle \bm{z},\bm{h}\right\rangle
      \left\langle \bm{z},\bm{k}\right\rangle
      \left\langle \bm{z},\bm{\ell}\right\rangle
      \\
  &\hspace{2.5cm}
    +\frac{3\alpha^5}{R^4}
      \Bigl(
        \left\langle \bm{h},\bm{k}\right\rangle
        \left\langle \bm{z},\bm{\ell}\right\rangle
        +\left\langle \bm{h},\bm{\ell}\right\rangle
        \left\langle \bm{z},\bm{k}\right\rangle
        +\left\langle \bm{k},\bm{\ell}\right\rangle
        \left\langle \bm{z},\bm{h}\right\rangle
      \Bigr).
\end{align*}
Here and below $\alpha=\alpha(\bm{z})$.  Since $\alpha\le1$ and
$\alpha\left\lVert \bm{z}\right\rVert \le R$, unit directions satisfy
\begin{equation}
  \left\lvert D\alpha[\bm{h}]\right\rvert \le\frac1R,
  \qquad
  \left\lvert D^2\alpha[\bm{h},\bm{k}]\right\rvert \le\frac4{R^2}.
  \label{eq:alpha-first-two-bounds}
\end{equation}
Because $\rho_R(\bm{z})=\alpha(\bm{z})\bm{z}$,
\begin{equation}
  D^2\rho_R[\bm{h},\bm{k}]
  =D\alpha[\bm{h}]\bm{k}+D\alpha[\bm{k}]\bm{h}+D^2\alpha[\bm{h},\bm{k}]\bm{z}.
  \label{eq:rho-second-explicit}
\end{equation}
The last term must be bounded before discarding its radial factors:
\begin{align*}
  \left\lvert D^2\alpha[\bm{h},\bm{k}]\right\rvert \left\lVert \bm{z}\right\rVert 
  &\le \frac{3\alpha^5\left\lVert \bm{z}\right\rVert ^3}{R^4}
       +\frac{\alpha^3\left\lVert \bm{z}\right\rVert }{R^2}\\
  &=\frac{3\alpha^2(\alpha\left\lVert \bm{z}\right\rVert )^3}{R^4}
    +\frac{\alpha^2(\alpha\left\lVert \bm{z}\right\rVert )}{R^2} \\
  &\le\frac4R.
\end{align*}
Combining this with \eqref{eq:alpha-first-two-bounds} in
\eqref{eq:rho-second-explicit} gives
$\left\lVert D^2\rho_R\right\rVert _{\mathrm{op}}\le \frac{6}{R}$.

Similarly,
\begin{equation}
  D^3\rho_R[\bm{h},\bm{k},\bm{\ell}]
  = D^2\alpha[\bm{h},\bm{k}]\bm{\ell}+D^2\alpha[\bm{h},\bm{\ell}]\bm{k}
    +D^2\alpha[\bm{k},\bm{\ell}]\bm{h}+D^3\alpha[\bm{h},\bm{k},\bm{\ell}]\bm{z}.
  \label{eq:rho-third-explicit}
\end{equation}
The first three terms contribute at most $\frac{12}{R^2}$.  For the last
term, the displayed formula for $D^3\alpha$ yields
\begin{align*}
  \left\lvert D^3\alpha[\bm{h},\bm{k},\bm{\ell}]\right\rvert \left\lVert \bm{z}\right\rVert 
  &\le \frac{15\alpha^7\left\lVert \bm{z}\right\rVert ^4}{R^6}
      +\frac{9\alpha^5\left\lVert \bm{z}\right\rVert ^2}{R^4}\\
  &=\frac{15\alpha^3(\alpha\left\lVert \bm{z}\right\rVert )^4}{R^6}
    +\frac{9\alpha^3(\alpha\left\lVert \bm{z}\right\rVert )^2}{R^4}
  \le\frac{24}{R^2}.
\end{align*}
Equation~\eqref{eq:rho-third-explicit} therefore gives
$\left\lVert D^3\rho_R\right\rVert _{\mathrm{op}}\le \frac{36}{R^2}$, completing
\eqref{eq:rho-higher-derivatives}.
Since $\bm{U}$ is an isometry from $\mathbb{R}^T$ into $\mathbb{R}^d$, the same bounds
hold for $\bm{q}_{\bm{U}}=\bm{U}^\top\rho_R$.

For $\theta_{\bm{U}}=b_{\eta,T}\circ \bm{q}_{\bm{U}}$, the first derivative satisfies
\[
  \left\lVert D\theta_{\bm{U}}\right\rVert 
  \le\left\lVert Db_{\eta,T}\right\rVert
      \left\lVert D\bm{q}_{\bm{U}}\right\rVert 
  \le C_{\rm ext}\eta
  \le C_\theta\eta.
\]
For the second derivative, the chain rule gives, for unit directions
$\bm{w}_1,\bm{w}_2$,
\begin{align*}
  D^2\theta_{\bm{U}}[\bm{w}_1,\bm{w}_2]
  ={}&D^2b_{\eta,T}
      [D\bm{q}_{\bm{U}}\bm{w}_1,D\bm{q}_{\bm{U}}\bm{w}_2]
      +Db_{\eta,T}[D^2\bm{q}_{\bm{U}}[\bm{w}_1,\bm{w}_2]].
\end{align*}
Using Lemma~\ref{lem:extractor}, \eqref{eq:rho-higher-derivatives},
$R=230\eta\sqrt T$, and $T\ge1$ gives
\begin{align*}
  \left\lVert D^2\theta_{\bm{U}}\right\rVert _{\mathrm{op}}
  &\le C_{\rm ext}
    +C_{\rm ext}\eta\frac6R\\
  &=C_{\rm ext}\left(1+\frac{6}{230\sqrt T}\right)\\
  &\le C_{\rm ext}\left(1+\frac{18}{230}
      +\frac{36}{230^2}\right) \\
  &=C_\theta.
\end{align*}
For the third derivative, the multivariate chain rule produces one
term of type $D^3b(D\bm{q})^3$, three permutations of type
$D^2b(D^2\bm{q},D\bm{q})$, and one term $Db(D^3\bm{q})$.  Therefore
\begin{align*}
  \left\lVert D^3\theta_{\bm{U}}\right\rVert _{\mathrm{op}}
  &\le \frac{C_{\rm ext}}{\eta}
    +3C_{\rm ext}\frac6R
    +C_{\rm ext}\eta\frac{36}{R^2}\\
  &=\frac{C_{\rm ext}}{\eta}
    \left(1+\frac{18}{230\sqrt T}
      +\frac{36}{230^2T}\right)\\
  &\le\frac{C_\theta}{\eta}.
\end{align*}
The value bound
$\left\lvert \theta_{\bm{U}}\right\rvert \le C_{\rm ext}\eta^2\le C_\theta\eta^2$
follows from the $r=0$ case of Lemma~\ref{lem:extractor}.  This proves
\eqref{eq:theta-composition-scales}.

The proof for $a_{\bm{U}}$ is the same, using $A_{\eta,T}=H_{\eta,T}-b_{\eta,T}$.  Specifically,
\begin{align*}
  &\left\lVert DA_{\eta,T}\right\rVert 
  \le23\eta\sqrt T+C_{\rm ext}\eta
  \le(23+C_{\rm ext})\eta\sqrt T
  \le C_A\eta\sqrt T,\\
  &\left\lVert D^2A_{\eta,T}\right\rVert _{\mathrm{op}}
  \le C_{{\rm ch},2}+C_{\rm ext}\le C_A,\\
  &\left\lVert D^3A_{\eta,T}\right\rVert _{\mathrm{op}}
  \le\frac{C_{{\rm ch},3}+C_{\rm ext}}{\eta}
  \le\frac{C_A}{\eta}.
\end{align*}
The first derivative of $a_{\bm{U}}=A\circ \bm{q}_{\bm{U}}$ is therefore at most
$C_A\eta\sqrt T\le C_a\eta\sqrt T$.  For the next two orders, the
same chain-rule formulas give
\begin{align*}
  \left\lVert D^2a_{\bm{U}}\right\rVert _{\mathrm{op}}
  &\le C_A+C_A\eta\sqrt T\frac6R
  =C_A\left(1+\frac6{230}\right)
  \le C_a,\\
  \left\lVert D^3a_{\bm{U}}\right\rVert _{\mathrm{op}}
  &\le\frac{C_A}{\eta}+3C_A\frac6R
    +C_A\eta\sqrt T\frac{36}{R^2}\\
  &=\frac{C_A}{\eta}\left(
    1+\frac{18}{230\sqrt T}
      +\frac{36}{230^2\sqrt T}\right)\\
  &\le\frac{C_a}{\eta}.
\end{align*}
This proves \eqref{eq:a-composition-scales}.

Finally,
\[
  \nabla h_R(\bm{z})=\frac{\bm{z}}{4s(\bm{z})}=\frac{1}{4}\rho_R(\bm{z}).
\]
Thus $\left\lVert \nabla h_R\right\rVert \le \frac{1}{4} R$.  Differentiating once more
gives $\nabla^2h_R=\frac{1}{4} D\rho_R$, whose eigenvalues lie in
$[0,\frac{1}{4}]$ by \eqref{eq:rho-jacobian}.  Nonnegativity and
$h_R(\bm{0})=0$ follow immediately from \eqref{eq:pseudo-huber}.
\end{proof}

\subsection{Proof of Lemma~\ref{lem:population-solution}}
\label{app:population-solution-proof}

\begin{proof}
The first identity follows from
$\mathbb E[\xi/p]=1$. Fix $\bm z$. Consider the candidate
\begin{equation}
  \bar{\bm v}
  =\bm M_\kappa^{-1}
  \begin{pmatrix}\vartheta_{\kappa,\bm U}\\0\end{pmatrix}
  =
  \begin{pmatrix}
    d_\kappa\vartheta_{\kappa,\bm U}\\s_\kappa\vartheta_{\kappa,\bm U}
  \end{pmatrix}
  =
  \begin{pmatrix}
    (d_\kappa/s_\kappa)\theta_{\bm U}\\\theta_{\bm U}
  \end{pmatrix}.
  \label{eq:population-v-candidate}
\end{equation}
By \eqref{eq:theta-composition-scales},
\eqref{eq:amplifier-elementary-bounds}, and
\eqref{eq:eta-zero},
\begin{equation}
  |\bar v_1|
  \le3C_\theta\eta^2
  \le\eta.
  \label{eq:population-v-cutoff-region}
\end{equation}
Hence $\psi_\eta(\bar v_1)=\bar v_1$ and
$\psi_\eta'(\bar v_1)=1$.

For fixed $\bm z$, the Hessian of the $\bm v$-dependent lower block is
\begin{equation}
  \bm M_\kappa
  -\vartheta_{\kappa,\bm U}\psi_\eta''(v_1)\bm e_1\bm e_1^\top.
  \label{eq:population-v-hessian}
\end{equation}
The bounds $s_\kappa\ge\kappa/4$ and
$\lVert\psi_\eta''\rVert_\infty\le C_\psi/\eta$ give
\begin{equation}
  |\vartheta_{\kappa,\bm U}|\lVert\psi_\eta''\rVert_\infty
  \le\frac{4C_\theta C_\psi\eta}{\kappa}
  \le\frac1{2\kappa}.
  \label{eq:population-v-hessian-perturbation}
\end{equation}
Since $\lambda_{\min}(\bm M_\kappa)=\kappa^{-1}$, the block in
\eqref{eq:population-v-hessian} is globally bounded below by
$(2\kappa)^{-1}\bm I_2$. Thus the $\bm v$-subproblem is globally
strongly convex. Its gradient is
\[
  \bm M_\kappa\bm v
  -\vartheta_{\kappa,\bm U}\psi_\eta'(v_1)\bm e_1.
\]
At \eqref{eq:population-v-candidate}, this gradient equals zero, so
$\bar{\bm v}$ is the unique global minimizer.

At this minimizer,
$\bm M_\kappa\bar{\bm v}=\vartheta_{\kappa,\bm U}\bm e_1$ and
$\bar v_1=d_\kappa\vartheta_{\kappa,\bm U}$. Therefore the minimized block has
value
\begin{align*}
  \frac12\bar{\bm v}^{\top}\bm M_\kappa\bar{\bm v}
  +\frac{d_\kappa}{2}\vartheta_{\kappa,\bm U}^2
  -\vartheta_{\kappa,\bm U}\psi_\eta(\bar v_1)
  =\frac{d_\kappa}{2}\vartheta_{\kappa,\bm U}^2
   +\frac{d_\kappa}{2}\vartheta_{\kappa,\bm U}^2
   -d_\kappa\vartheta_{\kappa,\bm U}^2
  =0.
\end{align*}
After minimizing in $\bm v$, the lower objective is exactly
\[
  \bm z\longmapsto
  \frac1{2\kappa}\lVert\bm z\rVert^2-\langle\bm x,\bm z\rangle.
\]
Its unique minimizer is $\bm z^*(\bm x)=\kappa\bm x$. Together with
\eqref{eq:population-v-candidate}, this proves
\eqref{eq:exact-lower-solution}--\eqref{eq:exact-auxiliary-solution}.

Finally, the upper objective reads the second auxiliary coordinate,
which equals $\theta_{\bm U}(\kappa\bm x)$. Using $H_{\eta,T}=A_{\eta,T}+b_{\eta,T}$,
\begin{align*}
  \Phi_{\bm U}^\kappa(\bm x)
  &=a_{\bm U}(\kappa\bm x)
    +\theta_{\bm U}(\kappa\bm x)+h_R(\kappa\bm x)\\
  &=H_{\eta,T}(\bm q_{\bm U}(\kappa\bm x))+h_R(\kappa\bm x),
\end{align*}
which is \eqref{eq:exact-hyperobjective}.
\end{proof}

\subsection{Proofs of Lemmas~\ref{lem:amplifier-derivatives} and
\ref{lem:regularity}}
\label{app:regularity-proof}

\begin{proof}
We first prove Lemma~\ref{lem:amplifier-derivatives}. Since
$s_\kappa\ge\kappa/4$, \eqref{eq:theta-composition-scales} immediately
gives \eqref{eq:attenuated-derivative-scales}.  The first derivatives
of the compensated block are
\begin{equation}
  D_z\mathcal R_{\kappa,\bm U} =(d\vartheta_{\kappa,\bm U}-\psi_\eta)D\vartheta_{\kappa,\bm U},
  \qquad
  \partial_{v_1}\mathcal R_{\kappa,\bm U}=-\vartheta_{\kappa,\bm U}\psi_\eta',
  \qquad
  \partial_{v_2}\mathcal R_{\kappa,\bm U}=0.
  \label{eq:R-first-derivatives}
\end{equation}
By \eqref{eq:amplifier-elementary-bounds},
\eqref{eq:attenuated-derivative-scales}, and
\eqref{eq:psi-derivative-scale},
\[
  |d\vartheta_{\kappa,\bm U}-\psi_\eta|
  \le4C_\theta\eta^2+C_\psi\eta
  \le(4C_\theta+C_\psi)\eta.
\]
Consequently,
\begin{alignat}{2}
  &\lVert D_z\mathcal R_{\kappa,\bm U}\rVert
  \le4C_\theta(4C_\theta+C_\psi)\frac{\eta^2}{\kappa},
  &&\qquad
  |\partial_{v_1}\mathcal R_{\kappa,\bm U}|
  \le4C_\theta C_\psi\frac{\eta^2}{\kappa}.
  \label{eq:R-gradient-block-bounds}
\end{alignat}
The norm of the full gradient is at most the sum of these two block
norms, proving the first bound in
\eqref{eq:amplifier-derivative-scales} with $C_{\rm grad}$.

Differentiating \eqref{eq:R-first-derivatives} gives the nonzero
Hessian blocks
\begin{equation}
    \begin{split}
    &D^2_{zz}\mathcal R_{\kappa,\bm U}
    =(d\vartheta_{\kappa,\bm U}-\psi_\eta)D^2\vartheta_{\kappa,\bm U}
      +dD\vartheta_{\kappa,\bm U} D\vartheta_{\kappa,\bm U}^\top,\\
  &D^2_{zv_1}\mathcal R_{\kappa,\bm U}=-\psi_\eta'D\vartheta_{\kappa,\bm U},\\
  &\partial^2_{v_1v_1}\mathcal R_{\kappa,\bm U}=-\vartheta_{\kappa,\bm U}\psi_\eta''.
  \label{eq:R-second-derivatives}
\end{split}
\end{equation}
Using $d_\kappa\le\kappa$, $0<\eta\le1$, and the same derivative
bounds yields
\begin{alignat}{3}
  &\lVert D^2_{zz}\mathcal R_{\kappa,\bm U}\rVert_{\rm op}
  \le(32C_\theta^2+4C_\theta C_\psi)\frac{\eta}{\kappa},
  &&\quad
  \lVert D^2_{zv_1}\mathcal R_{\kappa,\bm U}\rVert
  \le4C_\theta C_\psi\frac{\eta}{\kappa},
  &&\quad
  |\partial^2_{v_1v_1}\mathcal R_{\kappa,\bm U}|
  \le4C_\theta C_\psi\frac{\eta}{\kappa}.
  \label{eq:R-Hessian-block-bounds}
\end{alignat}
The operator norm of a symmetric block matrix is at most the sum of
the norms of its diagonal blocks and twice the norms of its strict
upper-triangular blocks.  Therefore
\begin{equation}
  \left\lVert D^2_{(z,v)}
    \mathcal R_{\kappa,\bm U}(\bm z,\bm v)\right\rVert_{\rm op}
  \le C_{\rm Hess}\frac{\eta}{\kappa}.
  \label{eq:R-Hessian-small}
\end{equation}

The unperturbed lower-variable Hessian is
$\operatorname{diag}(\kappa^{-1}\bm I,\bm M_\kappa)$, whose smallest
eigenvalue is $\kappa^{-1}$.  By \eqref{eq:R-Hessian-small}, Weyl's
inequality, and $\eta\le c_{\rm reg}\le(2C_{\rm Hess})^{-1}$,
\begin{equation}
  D^2_{(z,v)}g_{\bm U}^\kappa
  \succeq
  \left(\frac1\kappa-C_{\rm Hess}\frac\eta\kappa\right)\bm I
  \succeq\frac1{2\kappa}\bm I.
  \label{eq:population-strong-convexity}
\end{equation}
This proves the claimed strong convexity.

In the variable order $(\bm x,\bm z,\bm v)$, write the full Hessian as
\begin{equation}
  \nabla^2g_{\bm U}^\kappa
  =
  \begin{pmatrix}
    0&-\bm I&0\\
    -\bm I&\kappa^{-1}\bm I&0\\
    0&0&\bm M_\kappa
  \end{pmatrix}
  +
  \begin{pmatrix}
    0&0\\
    0&D^2_{(z,v)}\mathcal R_{\kappa,\bm U}
  \end{pmatrix},
  \label{eq:full-g-Hessian}
\end{equation}
where the second display is blocked as $\bm x$ versus
$(\bm z,\bm v)$.  The $(\bm x,\bm z)$ block of the first matrix obeys
the elementary estimate
\[
  \lVert\bm w_z\rVert^2
  +\lVert-\bm w_x+\kappa^{-1}\bm w_z\rVert^2
  \le3(\lVert\bm w_x\rVert^2+\lVert\bm w_z\rVert^2),
\]
and hence has operator norm at most $\sqrt3<3$.  Since
$\lVert\bm M_\kappa\rVert_{\rm op}=1$, the first matrix has operator
norm less than three.  Together with
\eqref{eq:R-Hessian-small}, this gives
\[
  \sup_{\bm x,\bm z,\bm v}
  \lVert\nabla^2g_{\bm U}^\kappa(\bm x,\bm z,\bm v)\rVert_{\rm op}
  \le3+C_{\rm Hess}\le8+C_{\rm Hess}=\bar L_g.
\]
Hence $\nabla g_{\bm U}^\kappa$ is globally
$\bar L_g$-Lipschitz.

It remains to prove the third-derivative bound asserted in
Lemma~\ref{lem:amplifier-derivatives}. Differentiating
\eqref{eq:R-second-derivatives}, the four nonzero tensor types are
\[
\begin{array}{@{}l@{\qquad}l@{}}
  D^3_{zzz}\mathcal R_{\kappa,\bm U}
  =(d\vartheta_{\kappa,\bm U}-\psi_\eta)D^3\vartheta_{\kappa,\bm U}
   +d\displaystyle\sum_{\rm three\ placements}
    D^2\vartheta_{\kappa,\bm U}\otimes D\vartheta_{\kappa,\bm U},
  &
  D^3_{zzv_1}\mathcal R_{\kappa,\bm U}=-\psi_\eta'D^2\vartheta_{\kappa,\bm U},
  \\[1mm]
  D^3_{zv_1v_1}\mathcal R_{\kappa,\bm U}=-\psi_\eta''D\vartheta_{\kappa,\bm U},
  &
  \partial^3_{v_1v_1v_1}\mathcal R_{\kappa,\bm U}=-\vartheta_{\kappa,\bm U}\psi_\eta'''.
\end{array}
\]
The $zzz$ block is at most
$(64C_\theta^2+4C_\theta C_\psi)/\kappa$, and each of the remaining
three block types is at most $4C_\theta C_\psi/\kappa$.  Expanding a
unit trilinear form gives one $zzz$ term, three $zzv_1$ terms, three
$zv_1v_1$ terms, and one $v_1v_1v_1$ term. Thus
\begin{equation}
  \lVert D^3_{(z,v)}\mathcal R_{\kappa,\bm U}\rVert_{\rm op}
  \le\frac{C_{\rm third}}\kappa,
  \label{eq:R-third-derivative-bound}
\end{equation}
which completes the proof of
\eqref{eq:amplifier-derivative-scales}.  Since $\kappa\ge2$, the
right-hand side of \eqref{eq:R-third-derivative-bound} is at most
$C_{\rm third}$.  All remaining terms in $g$ are quadratic, so
integrating that bound along a line segment proves that $\nabla^2g$
is globally $C_{\rm third}$-Lipschitz, and hence globally
$\bar\rho=(C_{\rm third}+1)$-Lipschitz.

For the upper objective, the only nonzero Hessian block is
$D^2_{zz}f_{\bm U}^\kappa=D^2a_{\bm U}+D^2h_R$; the coordinate $v_2$
appears linearly. Lemma~\ref{lem:composition-scales} and
\eqref{eq:huber-properties} give
\[
  \sup_{\bm x,\bm z,\bm v}
  \lVert\nabla^2f_{\bm U}^\kappa(\bm x,\bm z,\bm v)\rVert_{\rm op}
  \le C_a+\frac{1}{4}=\bar L_f.
\]
Thus, $\nabla f$ is globally $\bar L_f$-Lipschitz.  Moreover,
\begin{align}
  \left\lVert \nabla_{(z,v)}f_{\bm{U}}^\kappa\right\rVert 
  &\le\left\lVert Da_{\bm{U}}\right\rVert +\left\lVert \nabla h_R\right\rVert +1
  \notag\\
  &\le(C_a+\frac{115}{2})\eta\sqrt T+1\\
  &\le1+\frac{C_a+115/2}{\sqrt{C_\Delta}}\sqrt{\Delta} \\
  &=\bar C_f(\Delta).
  \label{eq:upper-lower-gradient-bound}
\end{align}
The last inequality follows from
$C_\Delta\eta^2T\le\Delta$: dividing by $C_\Delta$ and taking
square roots gives
$\eta\sqrt T\le\sqrt{\Delta/C_\Delta}$.

The lower-variable Hessian satisfies
\begin{align*}
  L_y
  &=
  \sup_{\bm{x},\bm{z},\bm v}
  \left\|
    D^2_{(z,v)}g_{\bm{U}}^\kappa(\bm{x},\bm{z},\bm v)
  \right\|_{\mathrm{op}}
  \\
  &\le
  \left\|
    \operatorname{diag}(\kappa^{-1}\bm I,\bm M_\kappa)
  \right\|_{\mathrm{op}}
  +
  \left\|
    D^2_{(z,v)}\mathcal{R}_{\kappa,\bm{U}}(\bm{z},\bm v)
  \right\|_{\mathrm{op}}
  \\
  &\le1+C_{\rm Hess}\frac\eta\kappa \\
  &\le1+\frac1{2\kappa}\\
  &\le\frac54.
\end{align*}
Here the penultimate inequality uses
$C_{\rm Hess}\eta\le1/2$, while the last uses $\kappa\ge2$.
Together with $\mu_g^{\rm act}\ge\mu_g=(2\kappa)^{-1}$ from
\eqref{eq:population-strong-convexity}, this gives
$\kappa_y=L_y/\mu_g^{\rm act}\le(5/2)\kappa
\le14\kappa=C_\kappa\kappa$. It remains to prove that the actual lower-level condition number cannot be substantially smaller than the construction parameter
\(\kappa\).  
Define
\[
  \bm{q}^\circ
  =
  (\eta,\ldots,\eta)\in\mathbb R^T,
  \qquad
  \bm{r}^\circ=\bm{U}\bm{q}^\circ\in\mathbb R^d.
\]
Since \(\bm{U}^\top \bm{U}=\bm{I}_T\), the linear map
\(\bm{q}\mapsto \bm{U}\bm{q}\) is an
isometry on \(\mathbb R^T\).  Consequently,
\[
  \left\lVert \bm{r}^\circ\right\rVert ^2
  =
  \left\lVert \bm{U}\bm{q}^\circ\right\rVert ^2
  =
  (\bm{q}^\circ)^\top \bm{U}^\top \bm{U}\bm{q}^\circ
  =
  \left\lVert \bm{q}^\circ\right\rVert ^2
  =
  \eta^2T,
\]
Thus,
\begin{equation}
    \left\lVert \bm{r}^\circ\right\rVert 
  =
  \eta\sqrt T
  =
  \frac{R}{230}
  <
  R.
  \label{eq:reverse-r-inside-ball}
\end{equation}

We next verify explicitly that every point in the open ball of radius
\(R\) belongs to the range of \(\rho_R\).  In particular, define
\begin{equation}
  \bm{z}^\circ
  =
  \frac{\bm{r}^\circ}
  {\sqrt{1-\left\lVert \bm{r}^\circ\right\rVert ^2/R^2}}.
  \label{eq:reverse-z-circ}
\end{equation}
This vector is finite because
\(\left\lVert \bm{r}^\circ\right\rVert <R\).  Moreover,
\begin{align*}
  \frac{\left\lVert \bm{z}^\circ\right\rVert ^2}{R^2}
  =
  \frac{\left\lVert \bm{r}^\circ\right\rVert ^2/R^2}
  {1-\left\lVert \bm{r}^\circ\right\rVert ^2/R^2},
\end{align*}
which implies
\[
1+\frac{\left\lVert \bm{z}^\circ\right\rVert ^2}{R^2}
  =
  \frac{1}
  {1-\left\lVert \bm{r}^\circ\right\rVert ^2/R^2}
\]
It follows that
\begin{align}
  \rho_R(\bm{z}^\circ)
  &=
  \frac{\bm{z}^\circ}
  {\sqrt{1+\left\lVert \bm{z}^\circ\right\rVert ^2/R^2}}
  \notag\\
  &=
  \frac{
    \bm{r}^\circ/
    \sqrt{1-\left\lVert \bm{r}^\circ\right\rVert ^2/R^2}
  }{
    1/\sqrt{1-\left\lVert \bm{r}^\circ\right\rVert ^2/R^2}
  }
  \notag\\
  &=
  \bm{r}^\circ\notag\\
  &=
  \bm{U}\bm{q}^\circ.
  \label{eq:reverse-soft-projection}
\end{align}
Therefore,
\begin{equation}
  \bm{q}_{\bm{U}}(\bm{z}^\circ)
  =
  \bm{U}^\top\rho_R(\bm{z}^\circ)
  =
  \bm{U}^\top \bm{U}\bm{q}^\circ
  =
  \bm{q}^\circ.
  \label{eq:reverse-hidden-coordinate}
\end{equation}

Every coordinate of \(\bm{q}^\circ\) has magnitude \(\eta\), and hence
\[
  \left\lvert q_i^\circ\right\rvert 
  =
  \eta
  >
  \frac{\eta}{2},
  \qquad
  1\le i\le T.
\]
By the definition of the progress function, this gives
\begin{equation}
  \operatorname{prog}_{\eta/2}(\bm{q}^\circ)=T.
  \label{eq:reverse-full-progress}
\end{equation}
The full-progress case of Lemma~\ref{lem:extractor} states that the
value and derivatives through order three of every summand defining
\(b_{\eta,T}\) vanish at \(\bm{q}^\circ\).
In particular,
\begin{equation}
  b_{\eta,T}(\bm{q}^\circ)=0,
  \qquad
  Db_{\eta,T}(\bm{q}^\circ)=0,
  \qquad
  D^2b_{\eta,T}(\bm{q}^\circ)=0.
  \label{eq:reverse-b-zero-jet}
\end{equation}

Recall that
\[
  \theta_{\bm{U}}(\bm{z})
  =
  b_{\eta,T}(\bm{q}_{\bm{U}}(\bm{z})).
\]
Combining \eqref{eq:reverse-hidden-coordinate} and
\eqref{eq:reverse-b-zero-jet} yields
\begin{equation}
  \theta_{\bm{U}}(\bm{z}^\circ)
  =
  b_{\eta,T}(\bm{q}^\circ)
  =
  0.
  \label{eq:reverse-theta-value-zero}
\end{equation}
The first-order chain rule gives
\[
  D\theta_{\bm{U}}(\bm{z}^\circ)
  =
  Db_{\eta,T}(\bm{q}^\circ)D\bm{q}_{\bm{U}}(\bm{z}^\circ)
  =
  0.
\]
The second-order chain rule gives, for arbitrary directions
\(\bm{w}_1,\bm{w}_2\in\mathbb R^d\),
\begin{align*}
  D^2\theta_{\bm{U}}(\bm{z}^\circ)[\bm{w}_1,\bm{w}_2]
  ={}&
  D^2b_{\eta,T}(\bm{q}^\circ)
  \bigl[
    D\bm{q}_{\bm{U}}(\bm{z}^\circ)\bm{w}_1,
    D\bm{q}_{\bm{U}}(\bm{z}^\circ)\bm{w}_2
  \bigr]
  \\
  &+
  Db_{\eta,T}(\bm{q}^\circ)
  \bigl[
    D^2\bm{q}_{\bm{U}}(\bm{z}^\circ)[\bm{w}_1,\bm{w}_2]
  \bigr]
  \\
  ={}&0.
\end{align*}
Consequently,
\begin{equation}
  \theta_{\bm{U}}(\bm{z}^\circ)=0,
  \qquad
  D\theta_{\bm{U}}(\bm{z}^\circ)=0,
  \qquad
  D^2\theta_{\bm{U}}(\bm{z}^\circ)=0.
  \label{eq:reverse-theta-zero-jet}
\end{equation}

Since $\vartheta_{\kappa,\bm U}=\theta_{\bm U}/s_\kappa$, its value
and first two derivatives also vanish at $\bm z^\circ$.  Substitution
into \eqref{eq:R-second-derivatives} shows that, for every
$\bm v\in\mathbb R^2$,
\begin{equation}
  D^2_{(z,v)}\mathcal{R}_{\kappa,\bm{U}}
  (\bm{z}^\circ,\bm v)=0.
  \label{eq:reverse-compensation-hessian-zero}
\end{equation}

Recall that
\[
  g_{\bm{U}}^\kappa(\bm{x},\bm{z},\bm v)
  =
  \frac1{2\kappa}\left\lVert \bm{z}\right\rVert ^2
  -
  \left\langle \bm{x},\bm{z}\right\rangle 
  +
  \frac12\bm v^\top\bm M_\kappa\bm v
  +
  \mathcal{R}_{\kappa,\bm{U}}(\bm{z},\bm v).
\]
When \(\bm{x}\) is fixed, the term \(-\left\langle \bm{x},\bm{z}\right\rangle \) is linear in \(\bm{z}\) and
therefore has zero lower-variable Hessian.  Combining this observation
with \eqref{eq:reverse-compensation-hessian-zero}, we obtain, for
every \(\bm{x}\) and $\bm v$,
\begin{equation}
  D^2_{(z,v)}g_{\bm{U}}^\kappa(\bm{x},\bm{z}^\circ,\bm v)
  =
  \operatorname{diag}(\kappa^{-1}\bm I,\bm M_\kappa).
  \label{eq:reverse-exact-lower-hessian}
\end{equation}

By Definition~\ref{def:condition-numbers}, the actual curvature
modulus and lower smoothness are $\mu_g^{\rm act}$ and $L_y$.
The matrix in \eqref{eq:reverse-exact-lower-hessian} has smallest
eigenvalue $\kappa^{-1}$ and operator norm one by
\eqref{eq:amplifier-spectrum}.
Since an infimum cannot exceed the value at a particular point,
\eqref{eq:reverse-exact-lower-hessian} implies
\begin{equation}
  \mu_g^{\rm act}\le\frac1\kappa.
  \label{eq:reverse-mu-upper}
\end{equation}
Likewise, since a supremum cannot be smaller than the value at a
particular point,
\begin{equation}
  L_y\ge1.
  \label{eq:reverse-L-lower}
\end{equation}
Therefore the actual lower-level condition number satisfies
\begin{align}
  \kappa_y
  =
  \frac{L_y}{\mu_g^{\rm act}}
  \ge
  \frac{1}{1/\kappa}
  =
  \kappa.
  \label{eq:reverse-condition-number}
\end{align}
Together with the previously established upper bound
\(\kappa_y\le C_\kappa\kappa\), this proves
\[
  \kappa
  \le
  \kappa_y
  \le
  C_\kappa\kappa.
\]
Hence the construction parameter and the actual lower-level condition
number are equivalent up to universal numerical constants:
\[
  \kappa_y=\Theta(\kappa).
\]
\end{proof}

\subsection{Proof of Lemma~\ref{lem:oracle-properties}}
\label{app:oracle-proof}

\begin{proof}
The only random term is
$(\xi/p)\mathcal R_{\kappa,\bm U}$.  Freshness gives
$\mathbb E[\xi/p\mid\mathscr F_{t-1}]=1$, so differentiation and
conditional expectation yield \eqref{eq:oracle-unbiasedness}.  In the
variable order $(\bm x,\bm z,\bm v)$, the complete gradient noise is
\begin{equation}
  \nabla\widehat g_{\bm{U}}^\kappa-\nabla g_{\bm{U}}^\kappa
  =\left(\frac\xi p-1\right)
  \begin{pmatrix}
    0\\
    (d_\kappa\vartheta_{\kappa,\bm U}-\psi_\eta)
      D\vartheta_{\kappa,\bm U}\\
    -\vartheta_{\kappa,\bm U}\psi_\eta'\bm e_1
  \end{pmatrix}.
  \label{eq:complete-noise-vector}
\end{equation}
The bounds used in \eqref{eq:R-gradient-block-bounds} give
\begin{alignat}{2}
  &\left\lVert
    (d_\kappa\vartheta_{\kappa,\bm U}-\psi_\eta)
    D\vartheta_{\kappa,\bm U}
  \right\rVert
  \le4C_\theta(4C_\theta+C_\psi)\frac{\eta^2}{\kappa},
  &&\quad
  \left\lvert\vartheta_{\kappa,\bm U}\psi_\eta'\right\rvert
  \le4C_\theta C_\psi\frac{\eta^2}{\kappa}.
  \label{eq:noise-block-fixed-bounds}
\end{alignat}
Thus the deterministic vector in
\eqref{eq:complete-noise-vector} has squared norm at most
$C_{\rm var}\eta^4/\kappa^2$ by the definition
\eqref{eq:variance-constant}.  For $\xi\sim\operatorname{Ber}(p)$,
\begin{align*}
  \mathbb{E}\left(\frac\xi p-1\right)^2
  &=p\left(\frac1p-1\right)^2+(1-p)\\
  &=\frac{(1-p)^2}{p}+1-p \\
  &=\frac{1-p}{p}.
\end{align*}
This proves \eqref{eq:oracle-variance-bound}. 

If $p=1$, the noise is zero.  Otherwise, we choose
$p=C_{\rm var}\eta^4/(\sigma^2\kappa^2)$, and
\[
  C_{\rm var}\frac{\eta^4}{\kappa^2}\frac{1-p}{p}
  =\sigma^2(1-p)\le\sigma^2.
\]
Equation \eqref{eq:p-inverse} follows from the two branches.
\end{proof}

\section{Proofs for Section \ref{sec:rotation-lifting}}
\label{app:progress-proofs}

\subsection{Proof of Lemma~\ref{lem:stopped-prefix-haar}}
\label{app:stopped-prefix-haar-proof}

\begin{proof}
For a fixed number \(\ell\) of initial frame columns, disintegration of
Haar measure with respect to these columns gives the completion kernel
\(\mathsf K_\ell\); this is the conditional-rotation statement used in
\citet[Lemma~14]{arjevani2023lower}.  To pass to the random index $J$,
test the displayed candidate conditional expectation against a
bounded \(\mathscr H_J\)-measurable function.  On the stratum
\(\{J=\ell\}\), that test function is a Borel function of
\((Z,\bm U_{\le\ell})\).  The event \(\{J=\ell\}\) is measurable with
respect to \(Z\) and is therefore independent of \(\bm U\).  Applying
the conditional Haar disintegration on this stratum gives the
conditional-expectation identity.  The defining property of a Borel
probability kernel also makes the kernel integral Borel in
\((Z,\bm U_{\le\ell})\), so the displayed candidate is
\(\mathscr H_J\)-measurable.  Summing over the finite partition
\(\{J=0\},\ldots,\{J=T\}\) proves the formula.  The marginal statement
follows from rotational invariance of the completion kernel.
\end{proof}

\subsection{Proof of Lemma~\ref{lem:haar-cited}}
\label{app:haar-proof}

\begin{proof}
We use the adaptive random-rotation argument of
\citet[Appendix~B.1.2, especially Lemmas~13--14]
{arjevani2023lower}.  Lemma~14 provides the relevant conditional
rotational invariance, while the proof of Lemma~13 combines it with
the standard spherical-cap estimate.  The coupled-transcript argument
below is our specialization of these
ingredients and is not a verbatim restatement of either cited lemma.

\emph{Step 1: augment the transcript with the first recorded frame
columns.}
For the purpose of the conditional-Haar argument, after call \(t\)
we record the Bernoulli bits \(\xi_1,\ldots,\xi_t\) and the first
\(L_t\) columns \(\bm{u}_1,\ldots,\bm{u}_{L_t}\) in an augmented transcript.  Let
\(\mathcal T_t^{\rm orig}\) denote the transcript actually observed
by the original procedure, and define
\[
  \mathcal T_t
  =
  \bigl(
    \mathcal T_t^{\rm orig},
    \xi_1,\ldots,\xi_t,
    \bm{u}_1,\ldots,\bm{u}_{L_t}
  \bigr).
\]
The additional transcript entries are auxiliary variables used only in the
proof; they are not additional oracle observations available to the original
procedure.  Throughout the remainder of this proof,
\(\mathcal T_t^{\rm orig}\) denotes the original query-response
transcript, whereas \(\mathcal T_t\) denotes the augmented
transcript defined above.

For the algorithm fixed in Proposition~\ref{prop:T-over-p}, the
original query and output rules are induced by
$\operatorname{Query}_t^{\mathsf A}$ and
$\operatorname{Output}_N^{\mathsf A}$ from
Definition~\ref{def:algorithm-class}.  The query descriptor and its
associated vector are obtained from the algorithmic query through the Borel map
\[
  (\bm{x}_t,\bm{z}_t,\bm v_t)
  \longmapsto
  \left(
    (\bm{x}_t,\bm{z}_t,\bm v_t),
    \frac{\rho_R(\bm{z}_t)}{\eta}
  \right).
\]

Let \(\pi_t^{\rm orig}\) denote the canonical projection onto the
original transcript component, so that
\[
  \pi_t^{\rm orig}(\mathcal T_t)
  =
  \mathcal T_t^{\rm orig}.
\]
The original query rule is extended to the augmented transcript space
by
\[
  \operatorname{Query}_t^{\rm aug}
  (\omega_{\rm alg},\mathcal T_{t-1})
  =
  \operatorname{Query}_t^{\rm orig}
  \bigl(
    \omega_{\rm alg},
    \pi_{t-1}^{\rm orig}(\mathcal T_{t-1})
  \bigr),
\]
and the output rule is extended by
\[
  \operatorname{Output}^{\rm aug}
  (\omega_{\rm alg},\mathcal T_N)
  =
  \operatorname{Output}^{\rm orig}
  \bigl(
    \omega_{\rm alg},
    \pi_N^{\rm orig}(\mathcal T_N)
  \bigr).
\]
Thus the additional Bernoulli bits and frame columns are available
for the measurability argument but do not affect the actual queries
or the final output of the original procedure.

For uniform notation, let \(J_t\) be the number of recorded frame
columns associated with each of the \(n=N+1\) vectors, and set
\begin{equation}
  J_t
  =
  \begin{cases}
    L_{t-1},&1\le t\le N,\\
    L_N,&t=n=N+1.
  \end{cases}
  \label{eq:probe-prefix-length}
\end{equation}
Thus \(J_t\) is the number of columns recorded in the augmented
transcript when \(\bm{r}_t\) is formed.  In particular, the additional
output vector \(\bm{r}_n=\bm{r}_{N+1}\) is assigned \(J_n=L_N\), because it receives no
new Bernoulli draw or oracle response.

\emph{Step 2: construct a coupled transcript with the required
measurability.}
Let \(\omega_{\rm alg}\) be the random seed of the algorithm from
Definition~\ref{def:algorithm-class}.  By construction, it is
independent of \(\bm{U}\) and of \(\xi_1,\ldots,\xi_N\).  After
conditioning on \(\omega_{\rm alg}\), the procedure's
query and output rules are Borel deterministic maps of their
transcripts.  We construct recursively an augmented coupled
transcript \(\widetilde{\mathcal T}_t\), whose original
query-response component is
\begin{equation}
  \widetilde{\mathcal T}_t^{\rm orig}
  =
  \pi_t^{\rm orig}(\widetilde{\mathcal T}_t),
  \label{eq:coupled-original-projection}
\end{equation}
together with coupled query descriptors
\(\widetilde{\bm a}_t\), coupled vectors
\(\widetilde{\bm r}_t\), and coupled responses
\(\widetilde{\bm Y}_t\).

Start with the same original and augmented initial transcripts:
\[
  \widetilde{\mathcal T}_0^{\rm orig}
  =
  \mathcal T_0^{\rm orig},
  \qquad
  \widetilde{\mathcal T}_0
  =
  \mathcal T_0.
\]
Suppose \(\widetilde{\mathcal T}_{t-1}\) has been constructed.
Feed this augmented transcript and the same seed
\(\omega_{\rm alg}\) to the augmented query rule, and denote its
next query descriptor and associated vector by
\begin{equation}
  \widetilde{\bm a}_t
  =
  (\widetilde{\bm x}_t,\widetilde{\bm z}_t,\widetilde{\bm v}_t)
  =
  \operatorname{Query}_t^{\rm aug}
  \bigl(
    \omega_{\rm alg},
    \widetilde{\mathcal T}_{t-1}
  \bigr)
  =
  \operatorname{Query}_t^{\rm orig}
  \bigl(
    \omega_{\rm alg},
    \widetilde{\mathcal T}_{t-1}^{\rm orig}
  \bigr),
  \qquad
  \widetilde{\bm r}_t
  =
  \frac{\rho_R(\widetilde{\bm z}_t)}{\eta}.
  \label{eq:coupled-query-rule}
\end{equation}
Define the coupled response for every outcome, including outcomes on
which geometric progress has already failed, by
\begin{equation}
  \widetilde{\bm Y}_t
  =
  M_{t,L_{t-1}}^{(\xi_t)}
  \bigl(
    \widetilde{\mathcal T}_{t-1}^{\rm orig},
    \widetilde{\bm a}_t,
    \widetilde{\bm r}_t,
    \bm u_1,\ldots,\bm u_{L_t}
  \bigr).
  \label{eq:coupled-response-rule}
\end{equation}
Equivalently, the first argument in this display is
\(\pi_{t-1}^{\rm orig}
  (\widetilde{\mathcal T}_{t-1})\), so it belongs to the original
transcript space \(\mathsf T_{t-1}\), exactly as required by
\eqref{eq:formal-prefix-map}.  The definition is legitimate because
\[
  L_t=\min\{T,L_{t-1}+\xi_t\}.
\]
Thus the last argument of \eqref{eq:coupled-response-rule} contains
exactly the first \(L_t\) frame columns required by
\eqref{eq:formal-prefix-map}.

Form the new original component
\(\widetilde{\mathcal T}_t^{\rm orig}\) by applying the original
transcript-update rule to
\(\widetilde{\mathcal T}_{t-1}^{\rm orig}\), the coupled query,
and \(\widetilde{\bm Y}_t\).  Then define the new augmented
coupled transcript by
\begin{equation}
  \widetilde{\mathcal T}_t
  =
  \bigl(
    \widetilde{\mathcal T}_t^{\rm orig},
    \xi_1,\ldots,\xi_t,
    \bm u_1,\ldots,\bm u_{L_t}
  \bigr).
  \label{eq:coupled-augmented-transcript}
\end{equation}
After call \(N\), feed \(\widetilde{\mathcal T}_N\) to
\(\operatorname{Output}^{\rm aug}\), apply the output-to-vector map
$\bm x\mapsto\rho_R(\kappa\bm x)/\eta$, and denote the resulting vector
by \(\widetilde{\bm r}_{N+1}\).

The recursion never evaluates
\(\operatorname{prog}_{1/4}(\bm{U}^\top\widetilde{\bm{r}}_t)\) and never reads a frame column
after \(\bm{u}_{L_t}\).  Consequently, on every event
\(\{L_t=\ell\}\),
\begin{equation}
  \widetilde{\mathcal{T}}_t
  \quad\text{is measurable with respect to}\quad
  \sigma\bigl(
    \omega_{\rm alg},\xi_1,\ldots,\xi_t,\bm{u}_1,\ldots,\bm{u}_\ell
  \bigr).
  \label{eq:coupled-transcript-measurability}
\end{equation}
This is proved directly by induction from
\eqref{eq:coupled-query-rule}--\eqref{eq:coupled-response-rule}.
More precisely, define the information available when the vector \(\bm r_t\) is
formed by
\begin{equation}
  \widetilde{\mathcal I}_t
  =
  \begin{cases}
    \sigma\bigl(
      \omega_{\rm alg},\xi_1,\ldots,\xi_{t-1},
      \bm{u}_1,\ldots,\bm{u}_{J_t}
    \bigr),
      &1\le t\le N,\\[1mm]
    \sigma\bigl(
      \omega_{\rm alg},\xi_1,\ldots,\xi_N,
      \bm{u}_1,\ldots,\bm{u}_{J_n}
    \bigr),
      &t=N+1.
  \end{cases}
  \label{eq:coupled-preprobe-sigma-field}
\end{equation}
For a random index $J\in\{0,\ldots,T\}$, the notation
$\sigma(Z,\bm u_1,\ldots,\bm u_J)$ denotes the sigma-field characterized
stratumwise by
\[
  E\cap\{J=\ell\}
  \in\sigma(Z,\bm u_1,\ldots,\bm u_\ell),
  \qquad 0\le\ell\le T.
\]
Then \(\widetilde{\bm{r}}_1,\ldots,\widetilde{\bm{r}}_t\), the subspace
\(\mathcal U_t\) spanned by the first \(J_t\) frame columns, and the
Gram-Schmidt basis introduced below are
all \(\widetilde{\mathcal I}_t\)-measurable.  In particular, no
response or Bernoulli bit from call \(t\) is used when
\(\widetilde{\bm{r}}_t\) is formed.

We also record the agreement property.  For \(1\le t\le N\), define coupled progress event as
\begin{equation}
  \widetilde{\mathcal E}_t
  =
  \left\{
    \operatorname{prog}_{1/4}(\bm{U}^\top\widetilde{\bm{r}}_t)
    \le L_{t-1}
  \right\}.
  \label{eq:coupled-progress-query-event}
\end{equation}
If
\(\widetilde{\mathcal E}_1,\ldots,
\widetilde{\mathcal E}_{t-1}\) hold, induction shows that
\[
  \widetilde{\mathcal T}_{t-1}^{\rm orig}
  =
  \mathcal T_{t-1}^{\rm orig},
  \qquad
  \widetilde{\mathcal T}_{t-1}
  =
  \mathcal T_{t-1},
  \qquad
  \widetilde{\bm a}_t=\bm a_t,
  \qquad
  \widetilde{\bm r}_t=\bm r_t.
\]
Indeed, both transcript equalities are trivial before the first call.
Suppose they hold before call \(s\) and
\(\widetilde{\mathcal E}_s\) also holds.  Then
\eqref{eq:formal-prefix-identity}, with
\(\ell=L_{s-1}\) and \(b=\xi_s\), says that the actual response equals
the coupled response in \eqref{eq:coupled-response-rule}, because
both maps receive the same original-transcript argument
\(\mathcal T_{s-1}^{\rm orig}\).  Hence the original
query-response transcripts remain identical after call \(s\).
The Bernoulli bits and the first \(L_s\) frame columns are shared by
construction,
so \eqref{eq:coupled-augmented-transcript} also makes the augmented
transcripts identical.  Thus the actual and coupled processes
coincide up to, but not beyond, their first progress violation.  If
all \(N\) query events hold, their output vectors coincide as well:
\[
  \widetilde{\bm{r}}_{N+1}=\bm{r}_{N+1}.
\]

\emph{Step 3: apply the conditional Haar law and state the cap bound.}
For each index \(t\), take \(Z\) in
Lemma~\ref{lem:stopped-prefix-haar} to consist of the algorithm seed and
the Bernoulli bits available before that vector is formed, and take \(J=J_t\).
The index \(J_t\) is a Borel function of those bits and is independent
of \(\bm U\).  Equations
\eqref{eq:coupled-transcript-measurability} and
\eqref{eq:coupled-preprobe-sigma-field} show, stratum by stratum, that
the coupled transcript and coupled vector are measurable with respect
to \(\mathscr H_{J_t}\).  Hence, on
\(\{J_t=\ell\}\), the remaining columns retain the Haar-completion law
in \(\operatorname{span}\{\bm u_1,\ldots,\bm u_\ell\}^\perp\).  This is
the precise point at which the restriction of the frame dependence to
the first \(J_t\) columns and independence of the Bernoulli bits from
the frame are used.

Moreover, there is a numerical constant \(c_{\rm sph}>0\) such that
the following spherical-cap estimate holds: if \(\bm{u}\) is uniform on
the unit sphere of an \(m\)-dimensional Euclidean space and \(\bm{w}\) is
a fixed unit vector in that space, then, for every
\(\tau\in(0,1)\),
\begin{equation}
  \mathbb{P}\left(
    \left\lvert \left\langle \bm{u},\bm{w}\right\rangle \right\rvert >\tau
  \right)
  \le
  2\exp\left(-c_{\rm sph}m\tau^2\right).
  \label{eq:spherical-cap-bound}
\end{equation}
The Haar disintegration conditional on a fixed number of initial
columns and this spherical-cap estimate are the two geometric
ingredients imported from the cited Appendix of
\citep{arjevani2023lower}; the extension to the random index was proved
in Lemma~\ref{lem:stopped-prefix-haar}.

\emph{Step 4: define the components of the associated vectors
orthogonal to the recorded frame columns.}
For notational economy, throughout the geometric calculation below
we write \(\bm{r}_t\) for the coupled vector \(\widetilde{\bm{r}}_t\).  The
transfer back to the actual vectors is performed in Step~7.
For a subspace \(\mathcal V\subseteq\mathbb{R}^d\), let
\(\bm{\Pi}_{\mathcal V}\) denote the Euclidean orthogonal projector onto
\(\mathcal V\), and let \(\bm{I}_d\) denote the identity map on
\(\mathbb{R}^d\).  For the vector \(\bm r_t\), define the span of the first
\(J_t\) frame columns by
\begin{equation}
  \mathcal U_t
  =
  \operatorname{span}
  \{\bm{u}_1,\ldots,\bm{u}_{J_t}\},
  \label{eq:exposed-frame-subspace}
\end{equation}
with the convention \(\mathcal U_t=\{0\}\) when \(J_t=0\).
Remove the component in \(\mathcal U_t\) from the first \(t\) vectors and
define
\begin{equation}
  \mathcal W_t
  =
  (\bm{I}_d-\bm{\Pi}_{\mathcal U_t})
  \operatorname{span}\{\bm{r}_1,\ldots,\bm{r}_t\}.
  \label{eq:residual-probe-subspace}
\end{equation}
Let
\begin{equation}
  m_t=\dim(\mathcal W_t).
  \label{eq:residual-probe-dimension}
\end{equation}
Since \(\mathcal W_t\) is the image of a space spanned by \(t\)
vectors,
\begin{equation}
  m_t\le t\le n.
  \label{eq:mt-bound}
\end{equation}
Apply Gram-Schmidt, in the displayed order, to the residual vectors
\[
  (\bm{I}_d-\bm{\Pi}_{\mathcal U_t})\bm{r}_1,\ldots,
  (\bm{I}_d-\bm{\Pi}_{\mathcal U_t})\bm{r}_t,
\]
discarding zero residuals.  This gives an orthonormal basis
\begin{equation}
  \bm{w}_{t,1},\ldots,\bm{w}_{t,m_t}
  \label{eq:residual-probe-basis}
\end{equation}
of \(\mathcal W_t\).  Gram-Schmidt with this deterministic
tie-breaking rule is a Borel operation on each stratum determined by
which residuals vanish.  Hence the basis is
\(\widetilde{\mathcal I}_t\)-measurable.  It may depend on the coupled
transcript and on the first \(J_t\) frame columns, but conditional on
those objects it is fixed.

To make the probability union bound range over a deterministic finite
index set, pad this basis to \(n\) vectors by defining, for every
\(s\in\{1,\ldots,n\}\),
\begin{equation}
  \bar{\bm w}_{t,s}
  =
  \begin{cases}
    \bm w_{t,s},&s\le m_t,\\
    \bm 0,&s>m_t.
  \end{cases}
  \label{eq:padded-residual-basis}
\end{equation}
The padding does not change any basis expansion below.  It only gives
a fixed ambient index set for the cap events.  Since \(m_t\) and the
un-padded basis are \(\widetilde{\mathcal I}_t\)-measurable, the padded
vectors are also \(\widetilde{\mathcal I}_t\)-measurable.

The component of \(\bm{r}_t\) orthogonal to \(\mathcal U_t\) is
\begin{equation}
  \bm{r}_t^\perp
  =
  (\bm{I}_d-\bm{\Pi}_{\mathcal U_t})\bm{r}_t.
  \label{eq:probe-orthogonal-component}
\end{equation}
Because \(\bm{r}_t^\perp\in\mathcal W_t\), it has the unique expansion
\begin{equation}
  \bm{r}_t^\perp
  =
  \sum_{s=1}^{m_t}
  \alpha_{t,s}\bm{w}_{t,s},
  \qquad
  \alpha_{t,s}
  =
  \left\langle \bm{w}_{t,s},\bm{r}_t^\perp\right\rangle .
  \label{eq:probe-basis-expansion}
\end{equation}
Define the coefficient vector
\begin{equation}
  \bm{\alpha}_t
  =
  (\alpha_{t,1},\ldots,\alpha_{t,m_t})
  \in\mathbb{R}^{m_t}.
  \label{eq:probe-coefficient-vector}
\end{equation}
Orthonormality, contractivity of orthogonal projection, and
\eqref{eq:haar-probe-bound} give
\begin{equation}
  \left\lVert \bm{\alpha}_t\right\rVert ^2
  =
  \sum_{s=1}^{m_t}\alpha_{t,s}^2
  =
  \left\lVert \bm{r}_t^\perp\right\rVert ^2
  \le
  \left\lVert \bm{r}_t\right\rVert ^2
  \le
  (\frac{R}{\eta})^2.
  \label{eq:probe-coefficient-bound}
\end{equation}

\emph{Step 5: choose the scalar cap threshold and control all
adaptive basis directions.}
Define the scalar threshold
\begin{equation}
  \tau
  =
  \frac{\eta}{4R\sqrt n}.
  \label{eq:cap-threshold}
\end{equation}
This symbol is a scalar and is unrelated to the coefficient vector
\(\bm{\alpha}_t\).  Since \(R\ge \eta\) and \(n\ge1\), one has
\(\tau\in(0,1)\), as required in
\eqref{eq:spherical-cap-bound}. 

Conditional on \(\widetilde{\mathcal I}_t\), and on each stratum
\(\{J_t=\ell\}\), each remaining frame column is uniform in an
orthogonal complement of dimension \(d-\ell\).  Since \(J_t\le T\),
\begin{equation}
  d-J_t
  \ge
  d-T.
  \label{eq:remaining-haar-dimension}
\end{equation}
The coupled vectors are
\(\widetilde{\mathcal I}_t\)-measurable and therefore fixed under
this conditioning; they impose no additional conditions on the
remaining frame columns.

For every fixed triple
\((t,s,j)\in\{1,\ldots,n\}\times
\{1,\ldots,n\}\times\{1,\ldots,T\}\), define the padded bad cap
event
\begin{equation}
  \mathcal B_{t,s,j}
  =
  \left\{
    s\le m_t,\quad J_t<j,\quad
    \left\lvert
      \left\langle \bm{u}_j,\bar{\bm w}_{t,s}\right\rangle
    \right\rvert >\tau
  \right\}.
  \label{eq:individual-cap-event}
\end{equation}
If \(s>m_t\) or \(j\le J_t\), this event is empty.  Otherwise,
\(\bar{\bm w}_{t,s}=\bm w_{t,s}\) is a unit vector fixed under
conditioning on \(\widetilde{\mathcal I}_t\).
Conditional on \(\widetilde{\mathcal I}_t\), the basis in
\eqref{eq:residual-probe-basis} is fixed.  On each stratum
\(\{J_t=\ell\}\), the remaining columns
\(\bm{u}_{\ell+1},\ldots,\bm{u}_T\) retain their conditional Haar law in
\(\operatorname{span}\{\bm{u}_1,\ldots,\bm{u}_\ell\}^\perp\).  Thus
\eqref{eq:spherical-cap-bound}, together with
\eqref{eq:remaining-haar-dimension}, gives
\begin{equation}
  \mathbb{P}\left(
    \mathcal B_{t,s,j}
    \,\middle|\,\widetilde{\mathcal I}_t
  \right)
  \le
  2\exp\left(
    -c_{\rm sph}(d-T)\tau^2
  \right).
  \label{eq:conditional-cap-probability}
\end{equation}

The events in \eqref{eq:individual-cap-event} are indexed by the fixed
finite set
\begin{equation}
  \{1,\ldots,n\}\times
  \{1,\ldots,n\}\times
  \{1,\ldots,T\},
  \qquad
  \left\lvert
    \{1,\ldots,n\}\times
    \{1,\ldots,n\}\times
    \{1,\ldots,T\}
  \right\rvert
  =n^2T.
  \label{eq:number-of-cap-events}
\end{equation}
The inactive triples correspond to empty events and therefore cost no
probability, while retaining them makes the union index set
deterministic.  Because
\(\tau^2=\eta^2(16 R^2n)^{-1}\), we may choose the numerical
constant \(C_{\rm H} \geq \frac{16}{c_{\rm sph}}\) in
\eqref{eq:haar-dimension-general}, so that 
\begin{equation}
  2\exp\left(
    -c_{\rm sph}(d-T)\tau^2
  \right)
  \le
  \frac{\delta}{n^2T}.
  \label{eq:cap-probability-budget}
\end{equation}
The explicitly constructed coupled transcript is essential here.  It
would not be valid to condition the actual transcript on all previous
progress events and then assert that the remaining columns retain
their Haar law, because those progress events themselves depend on the
remaining columns.  In contrast, the coupled continuation in
\eqref{eq:coupled-response-rule} satisfies the preceding
transcript-measurability statement on every outcome, including after a
progress violation.

\emph{Step 6: bound the coordinates after \(J_t\) for every coupled
vector.}
Fix an index \(t\), and suppose that none of the padded events
\(\mathcal B_{t,s,j}\), with \(1\le s\le n\) and
\(1\le j\le T\), occurs.  Fix an index \(j>J_t\).
For every active basis index \(s\le m_t\), the guards in
\eqref{eq:individual-cap-event} are then satisfied and
\(\bar{\bm w}_{t,s}=\bm w_{t,s}\).  Hence
\(
  \lvert\langle\bm u_j,\bm w_{t,s}\rangle\rvert\le\tau
\)
for all \(1\le s\le m_t\).
Since the frame columns are orthonormal,
\[
  \bm{u}_j\perp\mathcal U_t.
\]
Consequently,
\begin{equation}
  \left\langle \bm{u}_j,\bm{r}_t\right\rangle 
  =
  \left\langle \bm{u}_j,\bm{r}_t^\perp\right\rangle .
  \label{eq:remove-exposed-component}
\end{equation}
Using \eqref{eq:probe-basis-expansion}, Cauchy-Schwarz,
\eqref{eq:probe-coefficient-bound}, the complements of the events
in \eqref{eq:individual-cap-event}, and \(m_t\le n\), we obtain
\begin{align}
  \left\lvert \left\langle \bm{u}_j,\bm{r}_t\right\rangle \right\rvert 
  &=
  \left\lvert 
    \sum_{s=1}^{m_t}
    \alpha_{t,s}\left\langle \bm{u}_j,\bm{w}_{t,s}\right\rangle 
  \right\rvert 
  \notag\\
  &\le
  \left(
    \sum_{s=1}^{m_t}\alpha_{t,s}^2
  \right)^{1/2}
  \left(
    \sum_{s=1}^{m_t}
    \left\lvert \left\langle \bm{u}_j,\bm{w}_{t,s}\right\rangle \right\rvert ^2
  \right)^{1/2}
  \notag\\
  &\le
  \frac{R}{\eta}\sqrt{m_t}\,\tau 
  \notag\\
  &\le
  \frac{R}{\eta}\sqrt n
  \left(
    \frac{\eta}{4 R\sqrt n}
  \right)
  =
  \frac14.
  \label{eq:hidden-coordinate-bound}
\end{align}
Thus every coordinate of \(\bm{U}^\top \bm{r}_t\) with index greater
than \(J_t\) has absolute value at most \(1/4\), and hence
\begin{equation}
  \operatorname{prog}_{1/4}(\bm{U}^\top \bm{r}_t)
  \le
  J_t.
  \label{eq:probe-progress-from-cap-event}
\end{equation}

\emph{Step 7: prove progress for the coupled process and transfer it
to the actual process.}
Extend the notation from \eqref{eq:coupled-progress-query-event} to
the response-free output by defining, for every
\(t\in\{1,\ldots,n\}\),
\begin{equation}
  \widetilde{\mathcal E}_t
  =
  \left\{
    \operatorname{prog}_{1/4}(\bm{U}^\top\widetilde{\bm{r}}_t)\le J_t
  \right\}.
  \label{eq:individual-progress-event}
\end{equation}
For each fixed triple, Step~5 gives the conditional estimate
\eqref{eq:conditional-cap-probability}, and the dimension choice gives
the budget \eqref{eq:cap-probability-budget}.  Taking expectations
removes the conditioning, so
\begin{equation}
  \mathbb P(\mathcal B_{t,s,j})
  \le
  \frac{\delta}{n^2T}
  \qquad
  \text{for every fixed }(t,s,j).
  \label{eq:unconditional-cap-probability}
\end{equation}
Step~6 shows, for every \(t\), that
\begin{equation}
  \widetilde{\mathcal E}_t^c
  \subseteq
  \bigcup_{s=1}^{n}\bigcup_{j=1}^{T}
  \mathcal B_{t,s,j}.
  \label{eq:progress-failure-padded-union}
\end{equation}
Indeed, if every event on the right is absent, Step~6 bounds every
coordinate after \(J_t\) by \(1/4\), which is exactly
\(\widetilde{\mathcal E}_t\).  Therefore a union bound over the fixed
set of \(n^2T\) triples gives
\begin{align}
  \mathbb{P}\left(
    \bigcup_{t=1}^n
    \widetilde{\mathcal E}_t^c
  \right)
  \le
  \sum_{t=1}^{n}\sum_{s=1}^{n}\sum_{j=1}^{T}
  \mathbb P(\mathcal B_{t,s,j})
  \le
  n^2T\frac{\delta}{n^2T}
  =
  \delta.
  \label{eq:first-failure-union-bound}
\end{align}
Thus, with probability at least \(1-\delta\), all coupled progress
events hold simultaneously.

On this simultaneous event, the agreement induction in Step~2 shows
that
\[
  \widetilde{\mathcal{T}}_t=\mathcal{T}_t,
  \qquad
  \widetilde{\bm{a}}_t=\bm{a}_t,
  \qquad
  \widetilde{\bm{r}}_t=\bm{r}_t,
  \qquad
  1\le t\le N,
\]
and also \(\widetilde{\bm{r}}_{N+1}=\bm{r}_{N+1}\).  For
\(1\le t\le N\), the identity \(J_t=L_{t-1}\) therefore turns
\(\widetilde{\mathcal E}_t\) into
\eqref{eq:haar-query-conclusion}.  For the additional output vector,
\(n=N+1\), \(\widetilde{\bm{r}}_n=\bm{r}_{N+1}\), and \(J_n=L_N\);
hence \(\widetilde{\mathcal E}_n\) becomes
\eqref{eq:haar-output-conclusion}.  This proves both conclusions for
the actual adaptive procedure.
\end{proof}

\subsection{Proof of Lemma~\ref{lem:bilevel-interface}}
\label{app:interface-proof}

\begin{proof}
Fix \(\ell\in\{0,\ldots,T\}\), assume
\eqref{eq:bilevel-prefix-premise}, and put
\[
  \ell^+=\min\{T,\ell+1\}.
\]
For \(s\in\{\ell,\ell^+\}\), define the padded coordinate vector
computable from the first \(s\) frame columns by
\begin{equation}
  \bm{q}^{[s]}(\bm{z})
  = \bm{P}_s \bm{q} = 
  \sum_{j=1}^s
  \left\langle \bm{u}_j,\rho_R(\bm{z})\right\rangle \bm{e}_j
  \in\mathbb{R}^T.
  \label{eq:padded-prefix-coordinate}
\end{equation}
The expression in
\eqref{eq:padded-prefix-coordinate} makes clear that it is a Borel
function only of the public query and the first \(s\) frame columns.
We now express every returned vector in terms of \(\bm{q}^{[\ell]}\) or
\(\bm{q}^{[\ell^+]}\), without using the random index
\(\operatorname{prog}_{\eta/4}(\bm{q})\).

\emph{Step 1: the upper gradient uses only the first \(\ell\) frame
columns.}
Recall that
\[
  f_{\bm{U}}^\kappa(\bm{x},\bm{z},\bm v)
  =
  A_{\eta,T}(\bm{q}_{\bm{U}}(\bm{z}))+v_2+h_R(\bm{z}).
\]
The upper objective is independent of \(\bm{x}\), and the chain rule gives
\begin{alignat}{3}
  &\nabla_xf_{\bm{U}}^\kappa=0,
  &&\quad
  \nabla_zf_{\bm{U}}^\kappa
  =
  D\rho_R(\bm{z})^\top
  \bm{U}\nabla A_{\eta,T}(\bm{q})
  +\nabla h_R(\bm{z}),
  &&\quad
  \nabla_{\bm v}f_{\bm{U}}^\kappa=\bm e_2.
  \label{eq:upper-gradient-interface}
\end{alignat}
Let \(k=\operatorname{prog}_{\eta/4}(\bm{q})\).  The premise gives \(k\le\ell\).
The derivative-truncation identity
\eqref{eq:extractor-cylinder-A}, first at \(\bm{q}\) and then at
\(\bm{q}^{[\ell]}=\bm{P}_\ell \bm{q}\), gives
\begin{equation}
  \nabla A_{\eta,T}(\bm{q})
  =
  \bm{P}_\ell
  \nabla A_{\eta,T}(\bm{q}^{[\ell]}).
  \label{eq:A-gradient-prefix-support}
\end{equation}
Therefore
\begin{equation}
  \bm{U}\nabla A_{\eta,T}(\bm{q})
  =
  \sum_{j=1}^{\ell}
  \bigl[
    \nabla A_{\eta,T}(\bm{q}^{[\ell]})
  \bigr]_j \bm{u}_j.
  \label{eq:A-frame-expansion}
\end{equation}
In particular, neither the ambient vectors nor their scalar
coefficients depend on a column \(\bm{u}_j\) with \(j>\ell\), and the
formula itself does not require knowing \(k\).

For completeness, formula \eqref{eq:rho-jacobian} gives
\begin{equation}
  D\rho_R(\bm{z})^\top \bm{u}_j
  =
  \frac{\bm{u}_j}{s(\bm{z})}
  -
  \frac{\left\langle \bm{u}_j,\rho_R(\bm{z})\right\rangle }{R^2s(\bm{z})}
  \rho_R(\bm{z}).
  \label{eq:rho-column-formula}
\end{equation}
Thus applying \(D\rho_R(\bm{z})^\top\) to
\eqref{eq:A-frame-expansion} still uses only
\(\bm{u}_1,\ldots,\bm{u}_\ell\).  The remaining terms
\(\nabla h_R(\bm{z})\) and \(\bm e_2\) are functions only of the public
query.  Consequently the complete upper gradient has a Borel
representation using only the first \(\ell\) columns.

\emph{Step 2: a failed lower sample introduces no dependence on an
additional frame column.}
Differentiating \eqref{eq:sample-hard-instance} in the variable order
\((\bm{x},\bm{z},\bm v)\) gives
\begin{equation}
  \nabla\widehat g_{\bm{U}}^\kappa(\bm{x},\bm{z},\bm v;\xi)
  =
  \begin{pmatrix}
    -\bm{z}\\[1mm]
    \kappa^{-1}\bm{z}-\bm{x}
    +(\xi/p)
      \bigl(
        d_\kappa\vartheta_{\kappa,\bm U}(\bm z)-\psi_\eta(v_1)
      \bigr)
      D\vartheta_{\kappa,\bm U}(\bm z)\\[1mm]
    \bm M_\kappa\bm v
    -(\xi/p)\vartheta_{\kappa,\bm U}(\bm z)
      \psi_\eta'(v_1)\bm e_1
  \end{pmatrix}.
  \label{eq:sample-lower-full-gradient}
\end{equation}
When \(\xi=0\), both terms containing the scaled incoming-coordinate
term \(\vartheta_{\kappa,\bm U}\) disappear, and
\eqref{eq:sample-lower-full-gradient} reduces to
\begin{equation}
  \begin{pmatrix}
    -\bm{z}\\
    \kappa^{-1}\bm{z}-\bm{x}\\
    \bm M_\kappa\bm v
  \end{pmatrix}.
  \label{eq:failed-lower-gradient}
\end{equation}
This lower-gradient vector is determined entirely by the query and
contains no dependence on \(\bm U\).  Together with the upper gradient
from Step~1, which depends only on the first \(\ell\) frame columns, it
gives the required response map for a failed call; the complete failed
response need not be independent of those first \(\ell\) columns.

\emph{Step 3: a successful lower sample uses at most one additional
column.}
The quantities in \eqref{eq:sample-lower-full-gradient} that depend on
$\bm U$ are
\(\vartheta_{\kappa,\bm U}(\bm z)\) and
\(D\vartheta_{\kappa,\bm U}(\bm z)\). By definition,
\[
  \vartheta_{\kappa,\bm U}(\bm z)
  =s_\kappa^{-1}b_{\eta,T}(\bm q).
\]
The value-truncation identity
\eqref{eq:extractor-cylinder-values} gives
\begin{equation}
  \vartheta_{\kappa,\bm U}(\bm z)
  =
  s_\kappa^{-1}b_{\eta,T}(\bm{q}^{[\ell^+]}).
  \label{eq:attenuated-value-prefix}
\end{equation}
Indeed, \(\bm{q}\) and \(\bm{q}^{[\ell^+]}\) have the same progress index
\(k\le\ell\) and the same truncation through coordinate
\(\min\{T,k+1\}\).  Applying the value-truncation identity to both
vectors therefore gives the same value.
Since
\[
  (\bm{q}^{[\ell^+]})_j
  =
  \left\langle \bm{u}_j,\rho_R(\bm{z})\right\rangle 
  \quad
  \text{for }j\le\ell^+,
\]
the scalar \(\vartheta_{\kappa,\bm U}(\bm z)\) depends on the frame only through
\[
  \bm{u}_1,\ldots,\bm{u}_{\ell^+}.
\]
It is important to use the value-truncation identity here: the mere
fact that \(\vartheta_{\kappa,\bm U}(\bm z)\) is a scalar would not by itself prevent it
from carrying information about a frame column with index greater
than \(\ell^+\).

For the derivative, the ordinary chain rule gives
\begin{equation}
  D\vartheta_{\kappa,\bm U}(\bm z)
  =
  s_\kappa^{-1}D\rho_R(\bm{z})^\top
  \bm{U}\nabla b_{\eta,T}(\bm{q}).
  \label{eq:attenuated-gradient-interface}
\end{equation}
The derivative-truncation identity
\eqref{eq:extractor-cylinder-b} implies
\begin{equation}
  \nabla b_{\eta,T}(\bm{q})
  =
  \bm{P}_{\ell^+}
  \nabla b_{\eta,T}
  \bigl(\bm{q}^{[\ell^+]}\bigr).
  \label{eq:b-gradient-prefix-support}
\end{equation}
Consequently,
\begin{equation}
  \bm{U}\nabla b_{\eta,T}(\bm{q})
  =
  \sum_{j=1}^{\ell^+}
  \left[
    \nabla b_{\eta,T}
    \bigl(\bm{q}^{[\ell^+]}\bigr)
  \right]_j \bm{u}_j.
  \label{eq:b-frame-expansion}
\end{equation}
Combining \eqref{eq:attenuated-gradient-interface},
\eqref{eq:b-frame-expansion}, and
\eqref{eq:rho-column-formula} shows that
\(D\vartheta_{\kappa,\bm U}(\bm z)\) also uses no frame column after
\(\bm{u}_{\ell^+}\).

\emph{Step 4: define the uniform Borel response maps.}
Let \(M_\ell^{(0)}\) return the upper-gradient formula
\eqref{eq:upper-gradient-interface}, evaluated through
\eqref{eq:A-gradient-prefix-support}--\eqref{eq:A-frame-expansion},
together with the failed lower gradient
\eqref{eq:failed-lower-gradient}.  Let \(M_\ell^{(1)}\) return the
same upper gradient together with the successful specialization of
\eqref{eq:sample-lower-full-gradient}, evaluating
\(\vartheta_{\kappa,\bm U}(\bm z)\) and
\(D\vartheta_{\kappa,\bm U}(\bm z)\) through
\eqref{eq:attenuated-value-prefix} and
\eqref{eq:b-gradient-prefix-support}--\eqref{eq:b-frame-expansion}.
All operations involved are compositions of smooth finite-dimensional
maps and are therefore Borel.  The first map uses only
\(\bm{u}_1,\ldots,\bm{u}_\ell\), and the second only
\(\bm{u}_1,\ldots,\bm{u}_{\ell^+}\).  No map evaluates or receives \(k\).
In Lemma~\ref{lem:haar-cited}, take the public query descriptor to be
\(\bm{a}_t=(\bm{x}_t,\bm{z}_t,\bm v_t)\) and the associated vector to be
\(\bm{r}_t=\rho_R(\bm{z}_t)/\eta\).  For every \(t\), \(\ell\), and \(b\), define
\[
  M_{t,\ell}^{(b)}
  \bigl(
    \mathcal T,\bm{a},\bm{r},\bm{u}_1,\ldots,\bm{u}_{\ell_b}
  \bigr)
  :=
  M_\ell^{(b)}
  \bigl(
    \bm{a},\bm{u}_1,\ldots,\bm{u}_{\ell_b}
  \bigr).
\]
Thus the formal map may simply ignore the transcript argument
\(\mathcal T\) and the redundant vector argument \(\bm{r}\).
Since \(\eta>0\),
\[
  \operatorname{prog}_{1/4}(\bm{U}^\top \bm{r}_t)
  =
  \operatorname{prog}_{\eta/4}(\bm{U}^\top\rho_R(\bm{z}_t)).
\]

At call \(t\), take \(\ell=L_{t-1}\).  On the geometric event
\(\operatorname{prog}_{\eta/4}(\bm{U}^\top\rho_R(\bm{z}_t))\le L_{t-1}\),
the preceding identities show that the actual joint response satisfies
\[
  \bm{Y}_t
  =
  M_{t,L_{t-1}}^{(\xi_t)}
  \bigl(
    \mathcal T_{t-1},
    \bm{a}_t,
    \bm{r}_t,
    \bm{u}_1,\ldots,\bm{u}_{L_t}
  \bigr).
\]
Indeed, on the failed branch \(\xi_t=0\), one has
\(L_t=L_{t-1}\), and the response map uses only the first
\(L_{t-1}\) frame columns.  On the successful branch \(\xi_t=1\), one
has \(L_t=\min\{T,L_{t-1}+1\}\), and the response map uses only the
first \(L_t\) frame columns.  This is precisely
\eqref{eq:formal-prefix-identity}.  Hence the bilevel oracle satisfies
the stated response-dependence condition, and its response
maps do not evaluate the random progress index.
\end{proof}

\subsection{Proof of Proposition~\ref{prop:T-over-p}}
\label{app:progress-proof}

\begin{proof}
\emph{Step 1: associate a bounded vector with each query.}
At call \(t\), let \(\bm{a}_t=(\bm{x}_t,\bm{z}_t,\bm v_t)\) be the point selected by the
algorithm and define
\begin{equation}
  \bm{r}_t
  =
  \frac{\rho_R(\bm{z}_t)}{\eta}.
  \label{eq:normalized-query-probe}
\end{equation}
The soft projection satisfies
\[
  \left\lVert \rho_R(\bm{z}_t)\right\rVert <R =230\eta\sqrt T.
\]
Therefore
\begin{equation}
  \left\lVert \bm{r}_t\right\rVert 
  =
  \frac{\left\lVert \rho_R(\bm{z}_t)\right\rVert }{\eta}
  \le
  \frac{R}{\eta}
  =
  230\sqrt T.
  \label{eq:normalized-probe-bound}
\end{equation}

The construction has $T\ge1$, so
$R=230\eta\sqrt T>\eta$.
Moreover, \eqref{eq:p-choice} gives $0<p\le1$, the dimension condition
implies $d\ge T$, and \eqref{eq:normalized-probe-bound} gives the required norm bound. Thus the basic parameter and norm conditions of
Lemma~\ref{lem:haar-cited} are satisfied.

\emph{Step 2: include an output that was never queried.}
The algorithm may return a point \(\widehat{\bm{x}}_N^{\mathsf A}\) at which it never
called the oracle.  The hyper-objective at \(\widehat{\bm{x}}_N^{\mathsf A}\) evaluates
the rotated zero-chain at the lower  level solution
\[
  \bm{z}^*(\widehat{\bm{x}}_N^{\mathsf A})=\kappa\widehat{\bm{x}}_N^{\mathsf A}.
\]
Accordingly, append the additional output vector
\begin{equation}
  \bm{r}_{N+1}
  =
  \frac{\rho_R(\kappa\widehat{\bm{x}}_N^{\mathsf A})}{\eta}.
  \label{eq:normalized-output-probe}
\end{equation}
It obeys the same norm bound as \eqref{eq:normalized-probe-bound}.
No oracle response is returned, so this additional vector provides no
additional oracle information and does not change \(L_N\).

\emph{Step 3: apply the adaptive Haar lemma.}
Lemma~\ref{lem:bilevel-interface} verifies the required
response-dependence condition.  Choose \(\delta=1/8\).
Substituting \(R^2=230^2 \eta^2 T\) and \(2(N+1)^2T/\delta=16(N+1)^2T\) into
\eqref{eq:haar-dimension-general} gives exactly
\eqref{eq:dimension-choice}.  Hence, with probability at least
\(7/8\),
\begin{equation}
  \operatorname{prog}_{1/4}\!\left(
    \bm{U}^\top
    \frac{\rho_R(\kappa\widehat{\bm{x}}_N^{\mathsf A})}{\eta}
  \right)
  \le
  L_N
  =
  \min\{T,S_N\}.
  \label{eq:haar-applied}
\end{equation}
For every coordinate \(j\), because \(\eta>0\),
\[
  \left\lvert 
    \left[
      \bm{U}^\top\rho_R(\kappa\widehat{\bm{x}}_N^{\mathsf A})/\eta
    \right]_j
  \right\rvert >\frac14
\]
holds if and only if
\[
  \left\lvert 
    \left[
      \bm{U}^\top\rho_R(\kappa\widehat{\bm{x}}_N^{\mathsf A})
    \right]_j
  \right\rvert >\frac{\eta}{4}.
\]
Therefore
\begin{equation}
  \operatorname{prog}_{1/4}\!\left(
    \bm{U}^\top
    \frac{\rho_R(\kappa\widehat{\bm{x}}_N^{\mathsf A})}{\eta}
  \right)
  =
  \operatorname{prog}_{\eta/4}\!\left(
    \bm{U}^\top\rho_R(\kappa\widehat{\bm{x}}_N^{\mathsf A})
  \right).
  \label{eq:progress-rescaling-identity}
\end{equation}

\emph{Step 4: fewer than \(T\) successes occur with high
probability.}
Freshness gives
\[
  \mathbb{E}[\xi_t\mid\mathscr F_{t-1}]=p.
\]
Using the tower property and summing over the calls,
\begin{equation}
  \mathbb{E} S_N
  =
  \sum_{t=1}^N
  \mathbb{E}\!\left[
    \mathbb{E}[\xi_t\mid\mathscr F_{t-1}]
  \right]
  =
  Np.
  \label{eq:expected-success-count}
\end{equation}
Markov's inequality and \eqref{eq:N-progress-condition} now yield
\begin{equation}
  \mathbb{P}(S_N\ge T)
  \le
  \frac{\mathbb{E} S_N}{T}
  =
  \frac{Np}{T}
  \le
  \frac18.
  \label{eq:markov-successes}
\end{equation}
Thus the event
\[
  \mathcal E_{\rm B}\coloneqq\{S_N<T\}
\]
has probability at least \(7/8\).

\emph{Step 5: intersect the two good events.}
Let \(\mathcal E_{\rm H}\) be the event on which
\eqref{eq:haar-applied} holds.  On
\(\mathcal E_{\rm H}\cap\mathcal E_{\rm B}\),
\begin{align}
  \operatorname{prog}_{\eta/4}\!\left(
    \bm{U}^\top\rho_R(\kappa\widehat{\bm{x}}_N^{\mathsf A})
  \right)
  \le
  \min\{T,S_N\}
  =S_N
  <T.
  \label{eq:unfinished-on-good-events}
\end{align}
Finally, the union bound gives
\begin{align*}
  \mathbb{P}(\mathcal E_{\rm H}\cap\mathcal E_{\rm B})
  &=
  1-\mathbb{P}(\mathcal E_{\rm H}^c\cup\mathcal E_{\rm B}^c)\\
  &\ge
  1-\mathbb{P}(\mathcal E_{\rm H}^c)-\mathbb{P}(\mathcal E_{\rm B}^c)\\
  &\ge
  1-\frac18-\frac18
  =\frac34.
\end{align*}
Together with \eqref{eq:unfinished-on-good-events}, this proves
\eqref{eq:unfinished-probability}.
\end{proof}

\section{Proof of Lemma~\ref{lem:gap-gradient}}
\label{app:gap-gradient-proof}

\begin{proof}
At $\bm{x}=0$, $\rho_R(\bm{0})=0$ and $h_R(\bm{0})=0$, so
$\Phi_{\bm{U}}^\kappa(\bm{0})=H_{\eta,T}(\bm{0})$.  For every $\bm{x}$,
$h_R(\kappa \bm{x})\ge0$ and hence
\[
  \Phi_{\bm{U}}^\kappa(\bm{x})
  =H_{\eta,T}(\bm{U}^\top\rho_R(\kappa \bm{x}))+h_R(\kappa \bm{x})
  \ge\inf_{\bm{q}\in\mathbb{R}^T}H_{\eta,T}(\bm{q}).
\]
Taking the infimum over $\bm{x}$ and subtracting from the equality at zero
gives
\begin{align*}
  \Phi_{\bm{U}}^\kappa(\bm{0})-\inf_{\bm{x}}\Phi_{\bm{U}}^\kappa(\bm{x})
  \le H_{\eta,T}(\bm{0})-\inf_{\bm{q}}H_{\eta,T}(\bm{q})\le C_\Delta\eta^2T,
\end{align*}
where the last step is Lemma~\ref{lem:standard-chain}.  This proves
\eqref{eq:hard-instance-gap}.

We now prove the gradient statement. Recall that
\(R=230\eta\sqrt T\), and introduce the normalized variables
\begin{equation}
  \bm{w} = \frac{\kappa \bm{x}}{\eta},
  \qquad
  \bm{r}=\rho_{R/\eta}(\bm{w}),
  \qquad
  \bm{q}=\bm{U}^\top \bm{r}.
  \label{eq:normalized-variables}
\end{equation}
The radial definition gives the exact scaling
\[
  \rho_R(\eta \bm{w})
  =
  \eta\rho_{R/\eta}(\bm{w})
  =
  \eta\bm{r}.
\]
Together with $H_{\eta,T}(\eta \bm{q})=\eta^2F_T(\bm{q})$ and
\eqref{eq:pseudo-huber}, this yields
\begin{equation}
  \Phi_{\bm{U}}^\kappa(\eta \bm{w}/\kappa)
  =\eta^2\mathcal{G}_{\bm{U}}(\bm{w}),
  \label{eq:normalized-Phi}
\end{equation}
where
\begin{equation}
  \mathcal{G}_{\bm{U}}(\bm{w})
  =F_T(\bm{U}^\top\rho_{R/\eta}(\bm{w}))
   +\frac{1}{4}(\frac{R}{\eta})^2 
    \left(\sqrt{1+\frac{\eta^2\left\lVert \bm{w}\right\rVert ^2}{R^2}}-1\right).
  \label{eq:normalized-G}
\end{equation}

Let
\[
  s=\sqrt{1+\frac{\eta^2\left\lVert \bm{w}\right\rVert ^2}{R^2}},
  \qquad \bm{J}=D\rho_{R/\eta}(\bm{w}).
\]
Because the gradient of the pseudo-Huber term in normalized variables
is $\frac{1}{4}\bm{r}$, differentiation of \eqref{eq:normalized-G} gives
\begin{equation}
  \nabla\mathcal{G}_{\bm{U}}(\bm{w})=\bm{J}^\top \bm{U}\nabla F_T(\bm{q})+\frac{1}{4} \bm{r}.
  \label{eq:normalized-G-gradient}
\end{equation}
Assumption \eqref{eq:unfinished-assumption} becomes
$\operatorname{prog}_{1/4}(\bm{q})<T$, and therefore also $\operatorname{prog}_1(\bm{q})<T$.  Let
\[
  j=\operatorname{prog}_1(\bm{q})+1.
\]
Then $j\le T$ and $\left\lvert q_j\right\rvert \le1$.  The
incoming-coordinate derivative bound in Lemma~\ref{lem:standard-chain},
applied with scale one, gives
\begin{equation}
  \partial_jF_T(\bm{q})\le-1,
  \qquad
  \left\lVert \nabla F_T(\bm{q})\right\rVert \le23\sqrt T.
  \label{eq:normalized-frontier-facts}
\end{equation}

Formula \eqref{eq:rho-jacobian}, now at radius $R/\eta$, reads
\[
  \bm{J}=\frac{\bm{I}}s
  -\frac{\eta^2\bm{r}\bm{r}^\top}{R^2s}.
\]
Take the inner product of \eqref{eq:normalized-G-gradient} with the
unit vector $\bm{u}_j$.  The first part of $\bm{J}$ contributes
$(1/s)\partial_jF_T(\bm{q})$; the rank-one correction and the regularizer
may have the opposite sign.  The reverse triangle inequality $|a + b + c| \geq |a| - |b| - |c|$ therefore
gives
\begin{align}
  \left\lvert \left\langle \bm{u}_j,\nabla\mathcal{G}_{\bm{U}}(\bm{w})\right\rangle \right\rvert 
  \ge{}\frac1s
   -\frac{\eta^2\left\lvert \left\langle \bm{u}_j,\bm{r}\right\rangle \right\rvert \left\lVert \bm{r}\right\rVert 
      \left\lVert \nabla F_T(\bm{q})\right\rVert }{R^2s}
   -\frac{1}{4}\left\lvert \left\langle \bm{u}_j,\bm{r}\right\rangle \right\rvert .
  \label{eq:frontier-component-bound}
\end{align}
Here $\left\langle \bm{u}_j,\bm{r}\right\rangle =q_j$, so its absolute value is at most one.

We split into two exhaustive cases.  First suppose
$\left\lVert \bm{w}\right\rVert \le \frac{R}{2\eta}$.  Then
\[
  s
=
\sqrt{1+\frac{\eta^2\|\bm{w}\|^2}{R^2}}
\le
\sqrt{1+\frac14}
=
\frac{\sqrt5}{2},
  \qquad
  \left\lVert \bm{r}\right\rVert =\frac{\left\lVert \bm{w}\right\rVert }{s}\le \|\bm{w}\| \leq \frac{R}{2\eta}.
\]
Since also $s\ge1$, the middle term in
\eqref{eq:frontier-component-bound} satisfies
\[
  \frac{\eta^2\left\lvert \left\langle \bm{u}_j,\bm{r}\right\rangle \right\rvert \left\lVert \bm{r}\right\rVert 
      \left\lVert \nabla F_T(\bm{q})\right\rVert }{ R^2s}\leq\frac{\eta\cdot( R/2)\cdot23\sqrt T}
       {R^2}
  =\frac{23\sqrt T}{2(230\sqrt T)}
  =\frac1{20}.
\]
Using $\left\lVert \nabla\mathcal{G}_{\bm{U}}(\bm{w})\right\rVert \ge
\left\lvert \left\langle \bm{u}_j,\nabla\mathcal{G}_{\bm{U}}(\bm{w})\right\rangle \right\rvert $, 
\begin{equation}
  \left\lVert \nabla\mathcal{G}_{\bm{U}}(\bm{w})\right\rVert  \geq \left\lvert \left\langle \bm{u}_j,\nabla\mathcal{G}_{\bm{U}}(\bm{w})\right\rangle \right\rvert 
  \ge\frac2{\sqrt5}-\frac1{20}-\frac14
  >\frac12.
  \label{eq:small-w-component}
\end{equation}

Now suppose $\left\lVert \bm{w}\right\rVert >\frac{R}{2\eta}$.  The function
$t\mapsto t/\sqrt{1+t^2\eta^2/R^2}$ is increasing, so
\begin{equation}
\|\bm{r}\|
>
\frac{R/2}
{\eta\sqrt{1+1/4}}
=
\frac{R/2}{\eta\sqrt5/2}
=
\frac{R}{\eta\sqrt5}.
  \label{eq:r-large}
\end{equation}
Also $s> \frac{\sqrt5}{2}$.  The eigenvalues of $\bm{J}$ are at most $1/s$, and
therefore $\left\lVert \bm{J}\right\rVert _{\mathrm{op}}< \frac{2}{\sqrt5}$.  Applying the reverse triangle
inequality directly to \eqref{eq:normalized-G-gradient}, then using
\eqref{eq:normalized-frontier-facts} and \eqref{eq:r-large}, gives
\begin{align}
  \left\lVert \nabla\mathcal{G}_{\bm{U}}(\bm{w})\right\rVert 
  &\ge\frac{1}{4}\left\lVert \bm{r}\right\rVert -\left\lVert \bm{J}\right\rVert _{\mathrm{op}}
      \left\lVert \nabla F_T(\bm{q})\right\rVert \\
  &>\frac{R}{4\eta\sqrt5}
      -\frac{46}{\sqrt5}\sqrt T\\
  &=\frac{230/4-46}{\sqrt5}\sqrt T
  >\frac12.
  \label{eq:large-w-gradient}
\end{align}
Thus $\left\lVert \nabla\mathcal{G}_{\bm{U}}(\bm{w})\right\rVert >1/2$ in both cases.

Finally, $\bm{w}=\kappa \bm{x}/\eta$.  Applying the chain rule to
\eqref{eq:normalized-Phi} gives
\[
  \nabla_x\Phi_{\bm{U}}^\kappa(\bm{x})
  =\eta^2\left(\frac\kappa\eta\right)
    \nabla\mathcal{G}_{\bm{U}}(\bm{w})
  =\kappa\eta\nabla\mathcal{G}_{\bm{U}}(\bm{w}).
\]
Combining this identity with the last two case bounds proves
\eqref{eq:amplified-gradient-lower}.
\end{proof}

\section{Proof of Theorem~\ref{thm:main}}
\label{sec:assembly}

\begin{proof}
Fix $\kappa\ge2$, $\Delta>0$, $\sigma\ge0$, and
$\epsilon$ satisfying \eqref{eq:epsilon-range-main}.  Choose
\begin{equation}
  \eta=\frac{4\epsilon}{\kappa},
  \qquad
  X=\frac{\Delta}{2C_\Delta\eta^2},
  \qquad
  T=\lfloor X\rfloor.
  \label{eq:parameter-eta-T}
\end{equation}
We first check the rounding and all smallness conditions.  Since
$\eta^2=\frac{16\epsilon^2}{\kappa^2}$,
\[
  X=\frac{\Delta\kappa^2}{32C_\Delta\epsilon^2}.
\]
Fix the theorem constant
\begin{equation}
  c_\epsilon
  =\min\left\{
    1,\frac{c_{\rm reg}}4,
    \frac1{16\sqrt{C_\Delta}}
  \right\}.
  \label{eq:cepsilon-choice}
\end{equation}
Then $\epsilon\le c_\epsilon\sqrt\Delta$ and $\kappa\ge2$ imply
\[
  X
  =\frac{\Delta\kappa^2}{32C_\Delta\epsilon^2}
  \ge\frac{4}{32C_\Delta c_\epsilon^2}
  \ge32.
\]
In particular, $X\ge2$, and hence
\[
  \lfloor X\rfloor\ge X-1\ge X/2.
\]
Substituting the definition of $X$ gives the two-sided estimate
\begin{equation}
  \frac{\Delta}{4C_\Delta\eta^2}
  \le T
  \le\frac{\Delta}{2C_\Delta\eta^2}.
  \label{eq:T-two-sided}
\end{equation}
The right inequality implies $C_\Delta\eta^2T\le\Delta/2$.
Moreover, \eqref{eq:cepsilon-choice} and
$\epsilon\le c_\epsilon\le c_{\rm reg}/4$ give
\[
  \eta=\frac{4\epsilon}{\kappa}
  \le\frac{c_{\rm reg}}{\kappa}
  \le c_{\rm reg}
  \le\eta_0.
\]
Thus Lemmas~\ref{lem:population-solution} and
\ref{lem:regularity} apply.

Choose $p$ according to \eqref{eq:p-choice}.  At every call, take the
upper sample $\zeta_t$ to be a fixed dummy symbol and draw the lower
sample bits $\xi_t\stackrel{\mathrm{iid}}{\sim}\operatorname{Ber}(p)$,
independently of the Haar frame and the algorithm's random seed.  By
Lemma~\ref{lem:oracle-properties}, the resulting 
$\mathcal{SFO}$ is conditionally unbiased and has variance at most
$\sigma^2$.  The reciprocal $p^{-1}$ is given by \eqref{eq:p-inverse}.

Let $N$ satisfy \eqref{eq:main-call-budget}, where we now fix
\begin{equation}
  c_{\rm LB}
  =\frac1{512C_\Delta\max\{1,256C_{\rm var}\}}.
  \label{eq:clb-choice}
\end{equation}
We show that this choice ensures $N\le \frac{T}{8p}$.  From
\eqref{eq:T-two-sided} and \eqref{eq:p-inverse},
\begin{align}
  \frac{T}{8p}
  &\ge
  \frac{\Delta}{32C_\Delta\eta^2}
  \max\left\{1,\frac{\sigma^2\kappa^2}{C_{\rm var}\eta^4}\right\}
  \notag\\
  &=\frac{\Delta\kappa^2}{512C_\Delta\epsilon^2}
  \max\left\{1,
    \frac{\sigma^2\kappa^6}{256C_{\rm var}\epsilon^4}
  \right\}.
  \label{eq:T-over-p-scaled}
\end{align}
For every $a\ge0$ and $\beta>0$,
\begin{equation}
  \max\left\{1,\frac{a}{\beta}\right\}
  \ge\frac1{\max\{1,\beta\}}\max\{1,a\}.
  \label{eq:max-comparison}
\end{equation}
Indeed, if $\beta\ge1$, both $1\ge1/\beta$ and $a/\beta$ are at
least the corresponding term divided by $\beta$; if $\beta<1$, then
$a/\beta\ge a$ and the factor on the right is one.  Apply
\eqref{eq:max-comparison} with
$a=\sigma^2\kappa^6/\epsilon^4$ and
$\beta=256C_{\rm var}$.  Together with \eqref{eq:clb-choice}, this
proves $N\le \frac{T}{8p}$.

Set $n=N+1$ and choose $d$ satisfying
\eqref{eq:dimension-choice}.  Fix an arbitrary algorithm
\[
  \mathsf A
  \in
  \mathcal A_N^{\mathcal{SFO}}(d,d+2).
\]
Draw $\bm{U}$ Haar-uniformly from
$\operatorname{St}(d,T)$ and consider the hard instance
\eqref{eq:hard-f}--\eqref{eq:sample-hard-instance}.  This use of a
random frame is only a proof distribution; a deterministic frame will
be fixed below.  The frame, the algorithm seed, and the i.i.d.\
Bernoulli oracle samples are taken mutually independent on the product
probability space. Lemma~\ref{lem:gap-gradient} and the upper bound in
\eqref{eq:T-two-sided} give, for every $\bm{U}$,
\[
  \Phi_{\bm{U}}^\kappa(\bm{0})-\inf_{\bm{x}}\Phi_{\bm{U}}^\kappa(\bm{x})
  \le C_\Delta\eta^2T
  \le\Delta/2
  \le\Delta.
\]
Lemma~\ref{lem:regularity} verifies all class conditions,
including $\kappa_y=\Theta(\kappa)$.

We record the measurability and integrability needed to fix one frame.
The algorithmic query maps are Borel, the hard sample-gradient maps are
Borel in the frame, query, and Bernoulli bit, and the interaction has
only finitely many recursive steps.  Hence
\[
  (\bm U,\omega_{\rm alg},\xi)
  \longmapsto
  \left\lVert
    \nabla\Phi_{\bm U}^\kappa
      (\widehat{\bm x}_N^{\mathsf A})
  \right\rVert
\]
is jointly Borel measurable.  It is also uniformly bounded.  Indeed,
\eqref{eq:normalized-G-gradient},
\(\lVert\bm J\rVert_{\rm op}\le1\),
\(\lVert\nabla F_T\rVert\le23\sqrt T\), and
\(\lVert\bm r\rVert\le \frac{R}{\eta}=230\sqrt T\) give
\[
  \left\lVert\nabla\Phi_{\bm U}^\kappa(\bm x)\right\rVert
  \le\kappa\eta(23+230/4)\sqrt T
  \qquad\text{for every }\bm U\text{ and }\bm x.
\]
Thus its first and second powers are integrable, and Tonelli-Fubini
and Jensen may be applied below without an extended-value ambiguity.

Proposition~\ref{prop:T-over-p} shows that the rotated chain coordinates
at the output have progress below $T$ with probability at least $3/4$.
On that event,
Lemma~\ref{lem:gap-gradient} and the choice of $\eta$ give
\begin{equation}
  \left\lVert
    \nabla\Phi_{\bm{U}}^\kappa(\widehat{\bm{x}}_N^{\mathsf A})
  \right\rVert
  >\frac{\kappa\eta}{2}
  =2\epsilon.
  \label{eq:two-epsilon-event}
\end{equation}
The gradient norm is nonnegative outside the event, so
\begin{equation}
  \mathbb E_{\bm{U},\omega_{\rm alg},\xi}
    \left\lVert
      \nabla\Phi_{\bm{U}}^\kappa(\widehat{\bm{x}}_N^{\mathsf A})
    \right\rVert
  >\frac34(2\epsilon)=\frac32\epsilon.
  \label{eq:average-over-U}
\end{equation}
By Tonelli-Fubini, the left-hand side of
\eqref{eq:average-over-U} is the Haar average of the conditional
expectations at fixed frames.  If every deterministic frame had
conditional expectation at most $3\epsilon/2$, this average would be at
most $3\epsilon/2$, a contradiction.  Hence
there exists at least one deterministic frame $\bm{U}_0$ for which
\[
  \mathbb E_{\omega_{\rm alg},\xi}
    \left\lVert
      \nabla\Phi_{\bm{U}_0}^\kappa(\widehat{\bm{x}}_N^{\mathsf A})
    \right\rVert
  >\frac32\epsilon.
\]
The hard functions obtained by fixing $\bm{U}_0$ are deterministic.
The selected frame may depend on the algorithm as permitted by the
minimax quantifier order, but it does not depend on any realization of
the algorithm seed or oracle samples.
Finally, Jensen's inequality gives
\[
  \mathbb E_{\omega_{\rm alg},\xi}
    \left\lVert
      \nabla\Phi_{\bm{U}_0}^\kappa(\widehat{\bm{x}}_N^{\mathsf A})
    \right\rVert^2
  \ge
  \left(
    \mathbb E_{\omega_{\rm alg},\xi}
    \left\lVert
      \nabla\Phi_{\bm{U}_0}^\kappa(\widehat{\bm{x}}_N^{\mathsf A})
    \right\rVert
  \right)^2
  >\frac94\epsilon^2.
\]
This proves \eqref{eq:main-failure}.  Since
$\kappa\le\kappa_y\le14\kappa$, monotonicity and the max
comparison \eqref{eq:max-comparison} lose at most the fixed factor
$14^8$ when replacing $\kappa$ by $\kappa_y$.  This proves
\eqref{eq:main-kappa-y-form}.
\end{proof}

\end{document}